\documentclass[11pt,a4paper,usenames,dvipsnames,reqno]{amsart}
\usepackage{amssymb}
\usepackage{amscd}
\usepackage{amsfonts}
\usepackage{mathtools}
\usepackage[margin=1.25in]{geometry}
\usepackage{fancybox}
\usepackage{pdflscape}
\usepackage{rotating}
\usepackage[new]{old-arrows}
\numberwithin{equation}{section}
\usepackage[all,arc]{xy}
\usepackage{enumerate}
\usepackage{mathrsfs}
\usepackage{color}   
\usepackage{hyperref}
\hypersetup{
    colorlinks=true, 
    linktoc=all,     
    linkcolor=blue,  
}
\usepackage{amsthm}
\usepackage{amsmath}
\usepackage{graphicx}
\usepackage{wasysym}
\usepackage{pgf,tikz}
\usetikzlibrary{positioning}
\usepackage{ marvosym }
\usepackage{tikz-cd}
\usepackage{ textcomp }
\usepackage{bbm}
\usepackage{ stmaryrd }
\def\fillandplacepagenumber{%
 \par\pagestyle{empty}%
 \vbox to 0pt{\vss}\vfill
 \vbox to 0pt{\baselineskip0pt
   \hbox to\linewidth{\hss}%
   \baselineskip\footskip
   \hbox to\linewidth{%
     \hfil\thepage\hfil}\vss}}

\newcommand{\cc}{\mathbb C}

\newcommand{\zz}{\mathbb Z}

\newcommand{\qq}{\mathbb Q}
\newcommand{\rr}{\mathbb R}
\newcommand{\A}{\mathbb A}

\newcommand{\la}{\langle}
\newcommand{\ra}{\rangle}
\newcommand{\lra}{\longrightarrow}
\newcommand{\hra}{\hookrightarrow}

\newcommand{\al}{\alpha}
\newcommand{\be}{\beta}
\newcommand{\ga}{\gamma}
\newcommand{\de}{\delta}
\newcommand{\De}{\Delta}

\newcommand{\ep}{\epsilon}
\newcommand{\lam}{\lambda}
\newcommand{\Lam}{\Lambda}
\newcommand{\vp}{\varpi}
\newcommand{\sig}{\sigma}
\newcommand{\ka}{\kappa}

\newcommand{\quash}[1]{}
\newcommand{\W}{\mathbf{W}}

\DeclareMathOperator{\Gal}{Gal}

\DeclareMathOperator{\Tr}{Tr}

\DeclareMathOperator{\End}{End}

\DeclareMathOperator{\G}{G}
\DeclareMathOperator{\GL}{GL}

\DeclareMathOperator{\Hom}{Hom}

\DeclareMathOperator{\Ind}{Ind}

\DeclareMathOperator{\U}{U}

\DeclareMathOperator{\SO}{SO}

\DeclareMathOperator{\RO}{RO}
\DeclareMathOperator{\SRO}{SRO}
\DeclareMathOperator{\SL}{SL}

\DeclareMathOperator{\Lie}{Lie}
\DeclareMathOperator{\sspan}{span}
\DeclareMathOperator{\diag}{diag}

\DeclareMathOperator{\Ad}{Ad}

\DeclareMathOperator{\Res}{Res}
\DeclareMathOperator{\Orb}{Orb}

\DeclareMathOperator{\inv}{inv}
\DeclareMathOperator{\supp}{supp}
\DeclareMathOperator{\vol}{vol}

\newcommand{\Xv}{\mathrm{X}_n(F_v)}

\newcommand{\fu}{\mathfrak u}
\newcommand{\fg}{\mathfrak g}

\newcommand{\fl}{\mathfrak l}
\newcommand{\fp}{\mathfrak p}

\newcommand{\fgl}{\fg\fl}

\newcommand{\calf}{\mathcal{F}}
\newcommand{\calw}{\mathcal{W}}

\newcommand{\cala}{\mathcal{A}}
\newcommand{\frakc}{\mathfrak{C}}

\newcommand{\calb}{\mathcal{B}}
\newcommand{\calp}{\mathcal{P}}
\newcommand{\calg}{\mathcal{GR}}
\newcommand{\calh}{\mathcal{H}}
\newcommand{\cale}{\mathcal{E}}

\newcommand{\calq}{\mathcal{Q}}
\newcommand{\calo}{\mathcal{O}}

\newcommand{\calv}{\mathcal{V}}
\newcommand{\cald}{\mathcal{D}}

\newcommand{\call}{\mathcal L}

\newcommand{\Herm}{\mathcal{H}erm}

\newcommand{\Nm}{\mathrm{Nm}}

\newcommand{\iso}{\xrightarrow{\sim}}

\newcommand{\Fbar}{\overline{F}}
\newcommand{\bfun}{\mathbf{1}}

\makeatletter
\def\Ddots{\mathinner{\mkern1mu\raise\p@
\vbox{\kern7\p@\hbox{.}}\mkern2mu
\raise4\p@\hbox{.}\mkern2mu\raise7\p@\hbox{.}\mkern1mu}}
\makeatother

\newtheorem{Thm}{Theorem}[section]
\newtheorem{Prop}[Thm]{Proposition}

\newtheorem{Lem}[Thm]{Lemma}
\newtheorem{Cor}[Thm]{Corollary}

\newtheorem{Conj}[Thm]{Conjecture}

\theoremstyle{definition}
\newtheorem{Def}[Thm]{Definition}

\theoremstyle{remark}
\newtheorem{Rem}[Thm]{Remark}

\theoremstyle{definition}

\title[Relative endoscopic fundamental lemma]{The endoscopic fundamental lemma for unitary Friedberg--Jacquet periods}
\author{Spencer Leslie}

\date\today

\address{Department of Mathematics, Boston College, Chestnut Hill, MA, USA}

\email{spencer.leslie@bc.edu}

\begin{document}

\begin{abstract}
We prove the endoscopic fundamental lemma for the Lie algebra of the symmetric variety $U(2n)/U(n)\times U(n)$, where $U(n)$ denotes a unitary group of rank $n$. This is the first major step in the stabilization of the relative trace formula associated to the $U(n)\times U(n)$-periods of automorphic forms on $U(2n)$.
\end{abstract}

\maketitle
\setcounter{tocdepth}{1}
\tableofcontents

\section{Introduction}
In this paper, we prove the endoscopic fundamental lemma for the Lie algebra of the symmetric variety $U(2n)/U(n)\times U(n)$, stated below as Theorem \ref{Thm: main result intro}. Conjectured in \cite{Leslieendoscopy}, this is the first example of such a fundamental lemma and is the first major step in the stabilization of the relative trace formula associated to the $U(n)\times U(n)$-periods of automorphic forms on $U(2n)$. Let us now explain the context and motivation.
\subsection{Global motivation}
Let $E/F$ be a quadratic extension of number fields, $\A_E$ and $\A_F$ the associated rings of adeles. Let $W_1$ and $W_2$ be two $n$ dimensional Hermitian spaces over $E$. The direct sum $W_1\oplus W_2$ is also a Hermitian space and we have the embedding of unitary groups $$U(W_1)\times U(W_2)\hra U(W_1\oplus W_2).$$
Let $\pi$ be an irreducible cuspidal automorphic representation of $U(W_1\oplus W_2)(\A_F)$. Then $\pi$ is said to  be \emph{distinguished} by the subgroup $U(W_1)(\A_F)\times U(W_2)(\A_F)$ if the \emph{period integral}
\begin{equation}\label{period of interest}
\displaystyle\int_{[U(W_1)\times U(W_2)]}\varphi(h)dh
\end{equation}
is not equal to zero for some vector $\varphi$ in the $\pi$-isotypic subspace of automorphic forms on $U(W_1\oplus W_2)(\A_F)$. Here, $[H]=H(F)\backslash H(\A_F)$ for any $F$-group $H$. The integral \eqref{period of interest} converges by cuspidality.  These periods are a unitary version of the ``linear periods'' first studied by Friedberg and Jacquet \cite{FriedbergJacquet}, who showed that a cuspidal automorphic representation $\Pi$ of $\GL_{2n}(\A_F)$ is distinguished by $\GL_n(\A_F)\times \GL_n(\A_F)$ if and only if the central $L$-value $L(\frac{1}{2},\Pi)$ is non-zero and the exterior square $L$-function $L(s,\Pi,\wedge^2)$ has a pole at $s=1$. While the literature has stuck with the name linear periods for integrals over the subgroup $$\GL_n(\A_F)\times \GL_n(\A_F)\hra \GL_{2n}(\A_F),$$ the name ``unitary linear periods'' for the integrals (\ref{period of interest}) is clearly problematic. As a result, we refer to these periods as \textbf{unitary Friedberg--Jacquet periods}. 

Recently, these periods have appeared in the literature in several ways (for example, \cite{IchinoPrasanna}, \cite{PollackWanZydor}, \cite{GrahamShah}, and indirectly in \cite{LiZhang}) with interesting applications to arithmetic and relative functoriality. As a simple example, we have the following conjecture, which is a special case of conjectures of Getz and Wambach \cite{GetzWambach}.
\begin{Conj}\label{Conj 1}
Let $U(W_1\oplus W_2)(\A_F)$ be quasi-split and let $\pi$ be a generic cuspidal automorphic representation. Let $\Pi=BC(\pi)$ be the base change of $\pi$ to $\GL_{2n}(\A_E)$. The following are equivalent:
\begin{enumerate}
    \item the exterior square $L$-function $L(s,\Pi,\wedge^2)$ has a pole at $s=1$ and the central $L$-value $L(\frac{1}{2},\Pi)$ is non-zero,
    \item there exist $n$-dimensional Hermitian spaces $W_1'$ and $W'_2$ and a cuspidal automorphic representation $\pi'$ of on $U(W_1'\oplus W_2')(\A_F)$ nearly equivalent to $\pi$ such that $\pi'$ is distinguished by $U(W_1')(\A_F)\times U(W_2')(\A_F).$
    \end{enumerate}
\end{Conj}

Theorem 1.5 of \cite{PollackWanZydor} proves one direction of this conjecture under the assumption that $\pi$ is discrete series at a split place of $F$. In ongoing joint work with Jingwei Xiao and Wei Zhang, we formulate an extension of the above conjecture and propose a comparison of relative trace formulas (partially motivated by the more general setting of twisted base change in \cite{GetzWambach}) which enables us to prove cases of these conjectures. 

The crucial observation is that, unlike other relative trace formulas in the literature, the relative trace formula associated to the unitary Friedberg--Jacquet periods on $U(W_1\oplus W_2)(\A_F)$ is not \emph{stable}: when we consider the action of $U(W_1)\times U(W_2)$ on the {symmetric variety} $U(W_1\oplus W_2)/U(W_1)\times U(W_2)$, invariant polynomials distinguish only \emph{geometric orbits}. Appropriately, stability issues also arise in the local spectral theory of these periods \cite{WanBPfuture}. We must therefore ``stabilize" the geometric side of the relative trace formula to use the comparison of trace formulae with Xiao and Zhang to prove global results like Conjecture \ref{Conj 1}. 

\subsection{Local theory of endoscopy and the main result}
Now suppose that $E/F$ is a quadratic extension of non-archimedean local fields of characteristic zero and set $W=W_1\oplus W_2$. In \cite{Leslieendoscopy}, we initiated a program to stabilize the relative trace formula associated to these periods by developing the local theory of endoscopy for the ``Lie algebra'' of the symmetric variety 
\[
\calq=U(W)/U(W_1)\times U(W_2).
\]
Using \cite{Lesliedescent}, the full stabilization of the elliptic part of the relative trace formula ultimately reduces to this infinitesimal case.
Let us recall the basic notions.

\begin{Rem}
    Since the appearance of this article, we have formulated a general theory of endoscopy for symmetric varieties in \cite{LeslieEndodata}, and show that the notions introduced here and \cite{Leslieendoscopy} are compatible with the general notion of endoscopic symmetric variety. 
\end{Rem}
The  $2n$-dimensional Hermitian space $W=W_1\oplus W_2$ is naturally equipped with an involutive linear map: $\ep(w_1+w_2) = w_1-w_2$ for $w_i\in W_i$. This induces an involution on the unitary group $U(W)$ with the fixed-point subgroup $U(W)^\ep=U(W_1)\times U(W_2)$. Letting $\fu(W)$ denote the Lie algebra of $U(W)$, the differential of $\ep$ induces a $\zz/2\zz$-grading
\[
\fu(W)=\fu(W)_0\oplus \fu(W)_1,
\]
where $\fu(W)_i$ is the $(-1)^i$-eigenspace of $\ep$. Then $\fu(W)_1$ is  is the tangent space to the symmetric variety ${\calq}$ at the distinguished $U(W_1)\times U(W_2)$- fixed point and the subgroup $U(W_1)\times U(W_2)$ acts on $\fu(W)_1$ via restriction of the adjoint action. 

Section \ref{Section: relative endoscopy} reviews the notions of relative endoscopic data, endoscopic symmetric varieties, orbital integrals, and transfer. We postpone the details until then and content ourselves with the following special case: suppose that the extension $E/F$ is unramified and that $W_1=W_2=V_n$ is a split Hermitian space, so that there is a lattice $\Lam_n\subset V_n$ that is self-dual with respect to the Hermitian form. There is a natural identification in this case
\[
\fu(W)_1=\End(V_n),
\]
where the $U(V_n)\times U(V_n)$-action is given by pre- and post-composition. An elliptic endoscopic datum $\Xi$ determines positive integers $a,b$ such that $n=a+b$. To such a datum, we associate the \emph{endoscopic symmetric variety}
\[
\End(V_a)\oplus\End(V_b),
\]
where $V_a$ denotes a split Hermitian space of dimension $a$ and similarly with $V_b.$ Let $\Lam_a\subset V_a$ and $\Lam_b\subset V_b$ be self-dual lattices.

For a regular semi-simple element $\de\in \End(V_n)$, the endoscopic datum determines a character $\ka$, with respect to which we define the relative $\ka$-orbital integral
\[
\RO^\ka(f,\de)=\sum_{\de'\sim_{st}\de}\ka(\de')\RO(f,\de'),
\]
where $\de'$ runs over rational $U(V_n)\times U(V_n)$-orbits that lie in the same stable orbit of $\de.$ We show that there is a good notion of the matching of regular semi-simple elements
\[
\de\in\End(V_n)^{rss}\:\:\text{and}\:\:(\de_a,\de_b)\in\left(\End(V_a)\oplus\End(V_b)\right)^{rss},
\]
and transfer factors 
\[
\De_{rel}:\left(\End(V_a)\oplus\End(V_b)\right)^{rss}\times \End(V_n)^{rss}\to \cc.
\]
With these definitions, we say that $$f\in C_c^\infty(\End(V_n))\:\text{ and }\:f_{a,b}\in C_c^\infty(\End(V_a)\oplus\End(V_b))$$ are smooth transfers (or match) if 
\[
\SRO(f_{a,b},(\de_a,\de_b))=\De_{rel}((\de_a,\de_b),\de)\RO^\ka(f,\de)
\]
whenever $(\de_a,\de_b)$ and $\de$ match. Here $\SRO=\RO^\ka$ when $\ka=1$ is the trivial character. Our main result establishes the following matching of test functions. 
\begin{Thm}\label{Thm: main result intro}
Let $\End(\Lam_n)\subset \End(V_n)$ be the compact-open subring of endomorphisms of the lattice $\Lam_n,$ and let $\End(\Lam_a)\oplus\End(\Lam_b)$ be the analogous subring of $\End(V_a)\oplus \End(V_b)$.

The characteristic functions $\bfun_{\End(\Lam_n)}$ and $\bfun_{\End(\Lam_a)}\otimes \bfun_{\End(\Lam_b)}$ are smooth transfers of each other.
\end{Thm}

This is the \emph{endoscopic fundamental lemma} referred to in the title. It was conjectured in \cite{Leslieendoscopy}, where we proved the special case $n=2$ and $a=b=1$ via explicit computation.

\begin{Rem} We expect to show that the entire stabilization of the elliptic part of the relative trace formula follows from this result. Indeed, in the subsequent article \cite{Lesliedescent} we succeeded in developing tools to deduce the ``group-version'' of the fundamental lemma from Theorem \ref{Thm: main result intro}, and work-in-progress deals with the full transfer conjecture.

This expectation is entirely analogous to the Arthur-Selberg trace formula: work of Waldspurger \cite{Waldspurgerlocal,Waldstransfert} and Hales \cite{halesfundamental} reduced both the smooth transfer and fundamental lemma for the entire Hecke algebra to the fundamental lemma for the Lie algebra. This final statement was further reduced to the case of positive characteristic local fields in \cite{WaldsCharacteristic}. Famously, Ng\^{o} utilized the geometry of the Hitchin fibration to prove this last form in \cite{NgoFL}.
\end{Rem}

Our proof is firmly planted in characteristic-zero harmonic analysis. Drawing from several recent developments in a novel way, we show that this result follows from a \emph{new fundamental lemma for an entire modules of a Hecke algebra} for certain symmetric varieties.  This is already an example of ``relative endoscopy,'' but of a simpler sort.  This fundamental lemma in turn is reduced to an explicit transfer of orbital integrals in the context of the Lie algebra version of Jacquet--Rallis transfer from \cite{ZhangFourier}. We then introduce a new comparison of relative trace formulas to prove this fundamental lemma via global techniques. 

\subsection{Outline of the proof} The first part of our proof is a series of reductions, each one replacing an explicit statement of matching of orbital integrals for another. In each of these reductions, the varieties and groups involved in the orbital integrals change: the argument deals with no less than $6$ different types of orbital integrals! The goal is to obtain a statement to which global methods may be applied; this is the case for Theorem \ref{Thm: third reduction intro} below.

We outline these reductions in Figure \ref{datfvo} below, which indicates the relevant sections for each component of the argument. Beginning in the lower left-hand corner, we are in the context for Theorem \ref{Thm: main result intro}. We recall the \emph{contraction map} $r_n:\End(V_n)\to \Herm(V_n)$ introduced in \cite{Leslieendoscopy}, where
\[
\Herm(V_n)=\{y\in \End(V_n): \la yv,w\ra=\la v,yw\ra\text{ for any }v,w\in V_n\}
\]
is the \emph{twisted Lie algebra} for the quasi-split unitary group $U(V_n)$. The terminology ``twisted'' Lie algebra refers to the fact that $$\Lie(U(V_n)) = \Herm(V_n)\cdot \varepsilon,$$ where $\varepsilon\in E=F(\varepsilon)$ is a generator such that $\overline{\varepsilon} = -\varepsilon.$ In Section \ref{Section: Initial reduction}, we consider the Hermitian symmetric variety 
\[
X_n=\left(\Res_{E/F}\GL_n/\U(V_n)\right)(F)=\{y\in \Herm(V_n):\det(y)\neq0\}.
\]
The contraction map translates Theorem \ref{Thm: main result intro} into a matching of orbital integrals for \emph{non-standard test functions} on $X_n$ that are not compactly supported. These functions possess additional symmetries due to invariance properties of the endomorphism ring $\End(\Lam_n)$, allowing us to study them in terms of the \emph{spherical Hecke algebra} of the symmetric variety  $\calh_{K_{n,E}}(X_n)$ (see Section \ref{Section: spherical hecke} for details). Here $K_{n,E}=\GL_n(\calo_E)$ is a maximal compact subgroup and 
\[
\calh_{K_{n,E}}(X_n) := C_c^\infty(X_n)^{K_{n,E}}.
\]
A theorem of Hironaka \cite{hironaka1999spherical} shows that this ring is a free $\calh_{K_{n,E}}(\GL_n(E))$-module of rank $2^n$; in particular, there is a distinguished rank $1$ sub-module given by the embedding (see Section \ref{Section: spherical hecke} for the notation)
\[
-\ast \bfun_0:\calh_{K_{n,E}}(\GL_n(E))\lra \calh_{K_{n,E}}(X_n).
\]
Extension-by-zero gives an embedding of $\calh_{K_{n,E}}(X_n)\hra C_c^\infty(\Herm(V_n))$. Our first reduction relies on a morphism of Hecke algebras related to a \emph{non-tempered} version of parabolic induction (see Section \ref{Section: morphism Satake}) to show that Theorem \ref{Thm: main result intro} follows from the following result.
\begin{Prop}\label{Thm: first reduction intro}
There is a morphism of Hecke algebras
\[
\xi_{(a,b)}:\calh_{K_{n,E}}(\GL_n(E))\to \calh_{K_{a,E}}(\GL_a(E))\otimes \calh_{K_{b,E}}(\GL_b(E))
\]
 such that for any $\varphi\in \calh_{K_{n,E}}(\GL_n(E))$, the functions
 \[
 \varphi\ast \bfun_0\:\text{ and }\: \xi_{(a,b)}(\varphi)\ast \bfun_0
 \]
 are smooth transfers with respect to endoscopic transfer for the twisted Lie algebra. Here, $\xi_{(a,b)}(\varphi)\ast \bfun_0$ denotes the image of $\xi_{(a,b)}(\varphi)$ in $\calh_{K_{a,E}}(X_a)\otimes \calh_{K_{b,E}}(X_b)$ under the analogous embedding. 
\end{Prop}

This result  implies Theorem \ref{Thm: main result intro} (see Proposition \ref{Prop: initial reduction}) and gives new explicit endoscopic transfers of test functions on the twisted Lie algebra, generalizing the fundamental lemma of Laumon and Ng\^{o} \cite{LaumonNgo}. Moreover, it plays the role of the fundamental lemma for the Hecke algebra for the relative trace formula associated to the Galois symmetric pair $(\Res_{E/F}(\GL_n),\U_n)$; see \cite[Section 10]{LeslieEndodata}.

In order to establish Proposition \ref{Thm: first reduction intro}, we utilize a recent alternative proof of the existence of smooth transfer for the twisted Lie algebra due to Xiao \cite{Xiao}. This argument is indicated by the rectangle in the lower right of Figure \ref{datfvo}.  The arrows denote the following relationships:
\begin{itemize}
\item \underline{$ev_0$}: this arrow indicates the evaluation-at-$0$ map $ev_0(F)(-) =F(-,0)$;
\item \underline{JR}: this arrow indicates the Jacquet--Rallis transfer between the spaces
\[
\Herm(V_n)\times V_n\:\:\text{ and }\:\: \fgl_n(F)\times F^n\times F_n,
\]where $F_n=(F^n)^\ast$ is the space of $1\times n$ row vectors;
\item\underline{PD}: this arrow indicates Lie-algebraic parabolic descent of relative orbital integrals.
\end{itemize}

Roughly speaking, the matching of orbital integrals comprising the endoscopic transfer between $\Herm(V_n)$ and $\Herm(V_a)\oplus \Herm(V_b)$ may be obtained from parabolic descent of orbital integrals from
$\fgl_{n}(F)\times F^n\times F_n${ to the Levi factor }$\prod_{i=a,b}\fgl_{i}(F)\times F^i\times F_i$ by applying the Jacquet--Rallis transfer of
\newpage
\begin{landscape}
\begin{figure}  
\begin{tikzcd}
&&&&\text{\framebox{Part \ref{Part 2}}}&\\
&&&\GL_{n-1}(E)\times \GL_n(E)\ar[rr,"BC\text{ (Thm \ref{Thm: base change Twisted Jacquet--Rallis})}"]\ar[d]&&\GL_{n-1}(F)\times\GL_n(F)\ar[dd]\\
&&&\GL_n(E)\ar[d,"-\ast \bfun_0"]&&\\
&&\text{\framebox{Section \ref{Section: Initial reduction}}}&X_n\ar[d]&\text{\framebox{Section \ref{Section: Weil rep}}}&\GL_n(F)\ar[d]\\
\End(V_n)\ar[rr,"r_n"]\ar[dd,dotted,"\text{Thm \ref{Thm: full fundamental lemma}}"]&&\Herm(V_n)\ar[dd,dotted,"\text{Thm \ref{Thm: endoscopic base change}}"]&\Herm(V_n)\times V_n\ar[rr,"JR"]\ar[l,"ev_0"]&&\fgl_{n}(F)\times F^n\times F_n\ar[dd,"PD"]\ar[ll]\\
&&&&\text{\framebox{Section \ref{Section: nilpotent Weil}}}&\\
\End(V_a)\oplus \End(V_b)\ar[rr,"r_a\oplus r_b"] &&\Herm(V_a)\oplus\Herm(V_b)&\prod_{i=a,b}\Herm(V_i)\times V_i\ar[l,"ev_0"]\ar[rr,"JR"]&&\prod_{i=a,b}\fgl_{i}(F)\times F^i\times F_i\ar[ll]\\
\end{tikzcd}
\caption{Various spaces and the relations between their orbital integrals. While the notations on the two lower rows are the same, the bottom row deals with \emph{stable} orbital integrals, while the middle row deals with $\ka$-orbital integrals.}
\label{datfvo}
\end{figure}
 \fillandplacepagenumber
\end{landscape}
\newpage
\noindent
\cite{ZhangFourier} and then taking a limit to certain non-regular orbits. We outline this argument in greater detail in Section \ref{Section: first reduction}. The new tool for this proof is Xiao's analysis of certain \emph{generalized nilpotent orbital integrals} in the context of the Jacquet--Rallis transfer. We review these notions in Sections \ref{Section: JR theory} and \ref{Section: Xiao}. 

The upshot is that Proposition \ref{Thm: first reduction intro} follows if we can construct sufficiently many explicit pairs of functions that are smooth transfers of each other with respect to the Jacquet--Rallis transfer. To this end, we prove the following new fundamental lemma in the context of Jacquet--Rallis transfer generalizing the Jacquet--Rallis fundamental lemma of Yun \cite{YunJR}. 
\begin{Prop}\label{Thm: second reduction intro}
Let $\Lam_n\subset V_n$ be our self-dual lattice and set $\call_n=\calo_F^n\times {\calo_F}_n$. Let
\[
BC:\calh_{K_{n,E}}(\GL_n(E))\lra \calh_{K_{n,F}}(\GL_n(F))
\]
be the base change homomorphism of Hecke algebras. Then for any $\varphi\in \calh_{K_{n,E}}(\GL_n(E))$, the functions
\[
\{(\varphi\ast \bfun_0)\otimes\bfun_{\Lam_n},0\} \:\text{ and }\:BC(\varphi)\otimes\bfun_{\call_n}
\]
are smooth transfers of each other with respect to the Jacquet--Rallis transfer (\ref{eqn: JR matching}). 
\end{Prop}
This proposition implies Proposition \ref{Thm: first reduction intro}, hence Theorem \ref{Thm: main result intro} (see Proposition \ref{Prop: first reduction}).

From a spectral perspective, the presence of the characteristic functions $\bfun_{\Lam_n}$ and $\bfun_{\call_n}$ in the above comparison is artificial and ought to be remedied if we hope to apply global techniques to prove the result. Strikingly, the recently-explicated Weil representation (\cite{beuzart2019new}; see also \cite{ZhangAFL}) of $\SL_2(F)$ on the function spaces
\[
C_c^\infty(\Herm(V)\times V)\:\:\text{ and }\:\:C_c^\infty(\fgl_n(F)\times F^n\times F_n)
\]
allows us to do this. We recall the details of these representations in Section \ref{Section: Weil rep}. Beuzart-Plessis recently used this structure to give a new proof of the Jacquet--Rallis fundamental lemma \emph{for any residual characteristic}. We carry out a similar computation to reduce Proposition \ref{Thm: second reduction intro} to the final form of the fundamental lemma.
\begin{Thm}\label{Thm: third reduction intro}
Consider the Jacquet--Rallis transfer between the spaces
\[
C_c^\infty(\Herm(V))\:\:\text{ and }\:\:C_c^\infty(\fgl_n(F));
\]
see Section \ref{Section: JR theory} for details. Then for any $\varphi\in \calh_{K_{n,E}}(\GL_n(E))$, the functions
\[
\{(\varphi\ast \bfun_0),0\} \:\text{ and }\:BC(\varphi)
\]
are transfers of each other with respect to the matching (\ref{eqn: regular matching final}).
\end{Thm}
This result implies Proposition \ref{Thm: second reduction intro}, hence Theorem \ref{Thm: main result intro} (see Proposition \ref{Prop: second reduction}). This is our final reduction of Theorem \ref{Thm: main result intro}. Its proof is global, relying on a new comparison of trace formulas. We refer to these trace formulas as the \emph{twisted Jacquet--Rallis relative trace formulas} as they arise by ``switching the roles'' of the unitary group $U(V_n)$ and the linear group $\GL_n(F)$ in the original Jacquet--Rallis comparison. This switching is explained in terms of orbits at the beginning of Part \ref{Part 2}, and we refer the curious reader there. While several spectral consequences of this comparison are known by work of Feigon, Lapid, and Offen \cite{FLO} and Jacquet \cite{JacquetQuasisplit} on unitary periods of cusp forms, the resulting geometric comparison allows us to translate Theorem \ref{Thm: third reduction intro} into a spectral problem, despite being a statement in the \emph{Lie algebra version} of Jacquet--Rallis transfer with ostensibly no spectral content.

This argument is the content of Part \ref{Part 2}, which we have written to be essentially self-contained. To avoid making this introduction overlong, we refer the reader to the beginning of Part \ref{Part 2} for more details as the ideas and techniques used are rather different. We simply remark that the final piece is Theorem \ref{Thm: base change Twisted Jacquet--Rallis}, which establishes the fundamental lemma for the Hecke algebra for our comparison. This is the $BC$ arrow in Figure \ref{datfvo}, indicating that base change is the functoriality underlying this comparison.

Below we introduce notations and conventions which are in force throughout both Part \ref{Part 1} and Part \ref{Part 2}. We caution the reader that notations adopted within the two parts differ from one another in certain important aspects; we indicate these changes at the start of the second part. 
\subsection{Acknowledgements} First and foremost, we thank Jayce Getz for his mentorship, patience, and for asking the questions which led the author to consider relative notions of endoscopy.  We also want to thank Wei Zhang and Yiannis Sakellaridis for several illuminating conversations and for encouragement regarding this work. We also thank Rapha\"{e}l Beuzart-Plessis, Sol Friedberg, Ben Howard, Aaron Pollack, Ari Shnidman, Chen Wan, Jingwei Xiao, and Michal Zydor for helpful conversations. Finally, we thank the referees for several important comments which led to a better version of the paper.

This work was partially supported by an AMS-Simons Travel Award and NSF grants DMS-1902865 and DMS-2200852.

\subsection{Preliminaries}\label{Section: Prelim}

\subsubsection{Invariant theory}
For any field $F$ and any non-singular algebraic variety $\mathrm{Y}$ over $F$ with $\mathrm{G}$ an algebraic group over $F$ acting algebraically on $\mathrm{Y}$, we set $Y:=\mathrm{Y}(F)$. Let $\mathrm{Y}^{rss}$ denote the invariant-theoretic regular semi-simple locus. That is, $x\in Y^{rss}:=\mathrm{Y}^{rss}(F)$ if and only if its $\mathrm{G}$-orbit is of maximal possible dimension and is closed as a subset of $Y.$

For $x,x'\in Y^{rss}$, we say that $x'$ is in the \emph{rational $G$-orbit} of $x$ if there exists $g\in \mathrm{G}(F)$ such that
\[
g\cdot x= x'.
\]
Fixing an algebraic closure $\Fbar$, $x$ and $x'$ are said to lie in the same \emph{stable orbit} if $g\cdot x=x'$ for some $g\in \mathrm{G}(\Fbar)$ and such that the cocycle
\[
(\sig\mapsto g^{-1}g^\sig)\in Z^1(F,G)
\]
lies in $Z^1(F,G^0_x)$, where $G_x^0\subset G_x$ is the connected component of the identity of the stabilizer of $x$ in $G.$ A standard computation (see \cite[Lemma 2.1.5]{KalethaCharacter}) shows that the set $\calo_{st}(x)$ of rational orbits in the stable orbit of $x$ are in natural bijection with
\[
\cald(G^0_x/F):=\ker(H^1(F,G_x^0)\to H^1(F,G)).
\]
Here we ignore the dependence on $G$ in the notation. There is a natural abelianization of this pointed set 
\[
\frakc(G^0_x/F):=\ker(H^1_{ab}(F,G_x^0)\to H^1_{ab}(F,G)),
\]
where $H^1_{ab}$ is abelianized cohomology in the sense of \cite{Borovoi}. If $G_x^0$ is abelian (as will be the case for us), then 
\[
H^1(F,G_x^0)\cong H^1_{ab}(F,G_x^0),
\]
and there is an injective map
\[
\cald(G^0_x/F)\hra \frakc(G^0_x/F).
\] 
Finally, if $F$ is non-archimedean, this injection is a bijection and $\calo_{st}(x)$ is naturally a torsor over the abelian group $\frakc(G^0_x/F)$.
\subsubsection{Local fields} When $F$ is a non-archimedean field, we set $|\cdot|_F$ to be the normalized valuation so that if $\vp$ is a uniformizer, then
\[
|\vp|^{-1}_F= \#(\calo_F/\fp) = :q
\]
is the cardinality of the residue field. Here $\fp$ denotes the unique maximal ideal of $\calo_F.$ 

For any quadratic \'{e}tale algebra $E$ of a local field $F$, we set $\eta_{E/F}: F^\times \to \cc^\times$ for the character associated to the extension by local class field theory. In particular, if $E$ is not a field, then $\eta_{E/F}$ is the trivial character.

Throughout the article, all tensor products are over $\cc$ unless otherwise indicated.

\subsubsection{Groups and Hermitian spaces}
For a field $F$ and for $n\geq1$, we consider the algebraic group $\GL_n$ of invertible $n\times n$ matrices. Suppose that $E/F$ is a quadratic \'{e}tale algebra and consider the restriction of scalars $\Res_{E/F}(\GL_n)$. For any $F$-algebra $R$ and $g\in \Res_{E/F}(\GL_n)(R)$, we set \[
g\mapsto \overline{g}
\]
to be the Galois involution associated to the extension $E/F$; we also denote this involution by $\sig$. We denote by ${T}_n\subset \GL_n$ the diagonal maximal split torus, ${B_n}={T_nN_n}$ the Borel subgroup of upper triangular matrices with unipotent radical ${N_n}$. Set
\[
X_n:=\mathrm{X}_n(F)=\{x\in \GL_n(E): {}^t\overline{x}=x\}.
\]
Note that $\GL_n(E)$ acts on $X_n$ via
\[
g\ast x=gx{}^t\overline{g},\quad x\in X_n,\: g\in \GL_n(E),
\]
where ${}^tg$ denotes the transpose. We let $\calv_n$ be a fixed set of orbit representatives. For any $x\in X_n,$ set $\la\cdot,\cdot\ra_x$ to be the Hermitian form on $E^n$ associated to $x$. Denote by $V_x$ the associated Hermitian space and $\U(V_x)$ the corresponding unitary group. Note that if $g\ast x=x'$ then 
\[
V_x\xrightarrow{{}^t\overline{g}}V_{x'}
\]
is an isomorphism of Hermitian spaces. Thus, $\calv_n$ gives a set of representatives $\{V_x: x\in \calv_n\}$ of the equivalence classes of Hermitian vector space of dimension $n$ over $E$. When convenient, we will abuse notation and identify this set with $\calv_n$. If we are working with a fixed but arbitrary Hermitian space, we often drop the subscript. For any Hermitian space, we set
\[
U(V)=\U(V)(F).
\]

\subsubsection{Measures and centralizers}\label{measures} Suppose now that $E/F$ is an extension of local fields and fix an additive character $\psi: F\to \cc^\times$. By composing with the trace $\Tr_{E/F}$, we also obtain an additive character for $E$. We fix here our choice of Haar measures on the groups involved, choosing to follow \cite{FLO} closely. This is primarily to aid in Part \ref{Part 2} of the paper; the main point for Part \ref{Part 1} is that our choices are normalized to give the appropriate maximal compact subgroup volume $1$ in the unramified setting.

For any non-singular algebraic variety $\mathbf{Y}$ over $F$ of dimension $d$ and gauge form $\boldsymbol{\omega}_\mathbf{Y},$ the Tamagawa measure $dy_{Tam}$ of $Y=\mathbf{Y}(F)$ is defined by transferring the standard Haar measure on $F^d$ to $Y$ by $\boldsymbol{\omega}_\mathbf{Y}$.

For the varieties we consider, we set our measure to be of the form $dy=c(\psi)^{d/2}\boldsymbol{\lam}_\mathbf{Y}dy_{Tam}$, where 
\[
c(\psi) = \begin{cases}q^m&:\text{$F$ non-archimedean and } \mathrm{cond}(\psi)=\vp^m\calo_F,\\|a|_F&: \text{$F$ archimedean and }\psi(x) = e^{2\pi i\Tr_{E/\rr}(ax)}. \end{cases}
\]
For the other terms, we impose the choice that for any $\mathbf{Y},$
\[
\boldsymbol{\omega}_{\Res_{E/F}\mathbf{Y}}=p^\ast(\boldsymbol{\omega}_\mathbf{Y}),
\]
where $p^\ast$ is given in \cite[pg. 22]{Weiladeles}. We now fix $\boldsymbol{\omega}_\mathbf{Y}$:
\begin{itemize}
\item
For $\mathbf{Y}=\GL_n$, we take $\boldsymbol{\omega}_{\GL_n}=\frac{\prod_{i,j}dg_{i,j}}{\det(g)^n}$ and take $\boldsymbol{\lam}_{\GL_n} = \prod_{i=1}^nL(i,\bfun_{F^\times})$, where for any character $\chi:F^\times \to \cc^\times,$ $L(s,\chi)$ is the local Tate $L$-factor. We also set $\boldsymbol{\lam}_{\Res_{E/F}(\GL_n)} = \prod_{i=1}^nL(i,\bfun_{E^\times})$.
\item 
For $\mathbf{Y}=\mathrm{N}$ for any unipotent subgroup of $\GL_n,$ we set $\boldsymbol{\omega}_{\mathrm{N}}= \prod_idx_i$, where the product ranges over the non-constant coordinate functions on $N.$ We set $\boldsymbol{\lam}_{\mathrm{N}}=1.$
\item
For $\mathbf{Y}=\mathrm{X}_n$, set $\boldsymbol{\omega}_{\mathrm{X}_n}=\frac{\prod_{i\leq j}dx_{i,j}}{\det(x)^n}$, and take $\boldsymbol{\lam}_{\mathbf{X}_n} = \prod_{i=1}^nL(i,\eta^{i+1})$, where $\eta=\eta_{E/F}$ is the quadratic character associated to $E/F.$
\item
For $\mathbf{Y}=\U(V)$, we take $\boldsymbol{\omega}_{\U(V)}$ to be compatible with $\boldsymbol{\omega}_{\Res_{E/F}(\GL_n)}$ and $\boldsymbol{\omega}_{\mathrm{X}_n}.$ Finally, we take $\boldsymbol{\lam}_{\U(V)} = \prod_{i=1}^nL(i,\eta^i)$. In particular, the isomorphism
\[
X_n \cong \bigsqcup_{x\in \calv_n}\GL_n(E)/U(V_x)
\]
is compatible with these measures.
\end{itemize}
When $F$ is $p$-adic and $\psi$ of conductor $\calo_F,$ our choice of measure gives $K_n:=\GL_n(\calo_F)$ volume $1.$ When $E/F$ is also unramified, the same holds for the maximal compact subgroups $K_{n,E}:=\GL_n(\calo_E)\subset\GL_n(E)$ as well as $X_n(\calo_F):=\GL_n(\calo_E)\ast I_n$.

Finally, we consider the measures on regular semi-simple centralizers. Fix a Hermitian form $x$ and consider $U(V)=U(V_x)$. We will be interested in the \emph{twisted Lie algebra}
\[
\Herm(V)=\{\de\in \End(V): \la \de v, u\ra=\la v,\de u\ra\}.
\]
The group $U(V)$ acts on this space by the adjoint action, and an element $\de$ is regular semi-simple if its centralizer is a maximal torus $T_\de\subset U(V)$. To construct $T_\de$ note that there is a natural decomposition
$$F[\de]:=F[X]/(char_{\de}(X))=\prod_{i=1}^mF_i,$$ where $F_i/F$ is a field extension and $char_{\de}(X)$ denotes the characteristic polynomial of $\de$.
Setting $E_i=E\otimes_F F_i$, we have
\[
E[\de]=\prod_i E_i=\prod_{i\in S_1}E_i\times\prod_{i\in S_2}F_i\oplus F_i,
\]
where $S_1=\{i: F_i\nsupseteq E\}$.
\begin{Lem}\label{Lem: centralizers}
Let $\de\in \Herm(V)$ be regular semi-simple, let $T_\de$ denote the centralizer of $\de$ in $U(W)$. Then 
\[
T_\de\cong  Z_{\U(V)}(F)E[\de]^\times/F[\de]^\times,
\]
where $Z_{U(V)}(F)$ denotes the center of $U(V)$. Moreover, $H^1(F,T_\de)=\prod_{S_1}\zz/2\zz$ and 
\[
\frakc(T_\de/F)=\ker\left(H^1(F,T_\de)\to H^1_{ab}(F,U(V))\right)=\ker\left(\prod_{S_1}\zz/2\zz\to \zz/2\zz\right),
\]
where the map on cohomology is the summation of the factors.
\end{Lem}
\begin{proof}
This is proved, for example, in \cite[3.4]{RogawskiBook}.
\end{proof}

Set $T_{S_1}\cong Z_{\U(V)}(F)\prod_{i\in S_1}E_i^\times/F_i^\times\times \prod_{i\in S_2}\calo_{F_i}^\times$ for the unique maximal compact subgroup of $T_\de$. We choose the measure $dt$ on $T_\de$ giving this subgroup volume $1$. We will study orbital integrals over regular semi-simple orbits on several different varieties. We will always use the measures introduced here to define invariant measures on these orbits. By a slight abuse of notation, we will not acknowledge this in our notation.

\part{Endoscopic theory and reduction}\label{Part 1}
In this first part, we recall the basic theory of endoscopy for the infinitesimal symmetric variety from \cite{Leslieendoscopy}. We then state our main result in Theorem \ref{Thm: full fundamental lemma}. In Section \ref{Section: Initial reduction}, we show that the main theorem follows from a fundamental lemma for an entire Hecke algebra on the symmetric variety $X_n$. In Section \ref{Section: nilpotent Weil}, we use recent results relating endoscopic transfer for unitary Lie algebras and Jacquet--Rallis transfer to translate the problem into a statement about Jacquet--Rallis transfer. Finally, we use the Weil representation on certain spaces of orbital integrals to reduce the statement to its final form in Theorem \ref{Thm: JR fundamental lemma for algebra}. The proof of this final reformulation is the content of Part \ref{Part 2}.

For the entirety of this part, $F$ is a non-archimedean local field and $E/F$ is a quadratic \'{e}tale $F$-algebra. For the identity form $I_n\in X_n$, we set $V_n:=V_{I_n}$ and note that when $E/F$ is unramified, then $V_n$ is a split Hermitian space and $U(V_n)$ is the quasi-split unitary group.

\section{The relative endoscopic fundamental lemma}

In this section, we recall the basics of the theory of endoscopy for the infinitesimal symmetric variety; our reference  is \cite{Leslieendoscopy}. We then state our main result in Theorem \ref{Thm: full fundamental lemma}.
\subsection{The Lie algebra of the symmetric variety} \label{Section: pushforward}
 Recall that $W_1$ and $W_2$ denote two Hermitian spaces of dimension $n$ over $E$. Setting $W=W_1\oplus W_2$, we consider the Lie algebra $\fu(W)$ of the rank $2n$ unitary group $U(W)$. As mentioned in the introduction, this Lie algebra possesses a natural $\zz/2\zz$-grading
\[
\fu(W)= \fu(W)_0\oplus \fu(W)_1,
\]
where we have the natural identifications 
\[
\fu(W)_0=\fu(W_1)\oplus \fu(W_2),\text{  and   }\fu(W)_1=\Hom_{E}(W_2,W_1).
\]
Here $U(W_1)\times U(W_2)$ acts on $\fu(W)_1$ by the restriction of the adjoint action. In terms of $W_1$ and $W_2$, the action is given by $(g,h)\cdot \varphi = g\circ \varphi \circ h^{-1}$. 

In particular, any element $\de\in \fu(W)_1$ may be uniquely written
\[
\de=\de(X)= \left(\begin{array}{cc}&X\\-X^\tau&\end{array}\right),
\]
where $X\in \Hom_{E}(W_2,W_1)$ and where for any $w_i\in W_i$
\[
\la Xw_2,w_1\ra_1=\la w_2,X^\tau w_1\ra_2.
\]
For any such $\de$, we denote by $$H_\de=\{(h,g)\in U(W_1)\times U(W_2): h^{-1}\de g=\de\}\subset U(W_1)\times U(W_2)$$ the stabilizer of $\de$.

Define the regular semi-simple locus $\fu(W)_1^{rss}$ to be the set of $\de\in \fu(W)_1$ whose orbit under $U(W_1)\times U(W_2)$ is closed and of maximal dimension. In our present case, we have 
\[\fu(W)_1^{rss}=\fu(W)_1\cap \fu(W)^{rss},\] where $\fu(W)^{rss}$ is the classical regular semi-simple locus of the Lie algebra. This is due to the fact that the symmetric pair $(U(W),U(W_1)\times U(W_2))$ is geometrically quasi-split. See \cite[Section 1.2]{Lesliespringer} for more details on quasi-split symmetric varieties. In particular, if $\de\in \fu(W)_1^{rss},$ then $H_\de$ is a torus of rank $n$.

 There are natural \emph{contraction maps} $r_i:\fu(W)_1\to \Herm(W_i)$ given by 
\begin{equation}\label{eqn: choice of contraction}
r_i(\de(X))= \begin{cases} -XX^\tau :\quad i=1\\ -X^\tau X:\quad i=2.\end{cases}
\end{equation}

\begin{Prop}\label{Lem: cat quotient Lie algebra 1}
The map $r:= r_1$ intertwines the $U(W_1)$ action on $\fu(W)_1$ and the adjoint action on $\Herm(W_1)$. Moreover, the pair $(\Herm(W_1), r)$ is a categorical quotient for the $U(W_2)$-action on $\fu(W)_1$.\qed
\end{Prop}

\begin{proof} The equivariance statement is obvious. 
As the categorical quotient assertion is geometric, we may assume without loss that $F=\Fbar$. The action we consider is following action of $\GL_n\times \GL_n$ on $\fgl_n\times \fgl_n$:
\[
(g,h)\cdot (X,Y) = (gXh^{-1},hYg^{-1}).
\]
The map $r$ becomes the product map
\begin{align*}
\fgl_n\times \fgl_n&\to \fgl_n\\
					(X,Y)&\mapsto  XY.
\end{align*}We make use of Igusa's criterion \cite[Section 3]{ZhangFourier}: let a reductive group $H$ act on an irreducible affine variety $X$. Let $Q$ be a normal irreducible variety, and let $\pi:X\to Q$ be a morphism that is constant on $H$ orbits such that
\begin{enumerate}
\item $Q-\pi(X)$ has codimension at least two,
\item there exists a nonempty open subset $Q'\subset Q$ such that the fiber $\pi^{-1}(q)$ of $q\in Q'$ contains exactly one orbit.
\end{enumerate}
Then $(Q,\pi)$ is a categorical quotient of $(H,X)$.  Note that it is clear that $r$ is surjective as $X\to (X,I_n)$ provides a section, so that the first criterion is satisfied. For the second criterion, we note that the open set $Q'=\GL_n(F)$ works.
\end{proof}

Note that a similar argument gives the following lemma for the quotient by both unitary actions.

\begin{Lem}\label{Lem: cat quotient Lie algebra}
Let $\A^n$ denote a $n$-dimensional affine space and let $\pi: \fu(W)_1\to \A^n$ be the morphism sending $\de(X)$ to the coefficients of the characteristic polynomial of $r(\de(X))=-XX^\tau$. Then the pair $(\A^n,\pi)$ is a categorical quotient for the $U(W_1)\times U(W_2)$ action on $\fu(W)_1$.
\end{Lem}

 Let 
\begin{equation*}
\fu(W)_1^{iso}\cong\mathrm{Iso}_E(W_2,W_1)
\end{equation*} be the open subvariety of elements $\de(X)$ where $X:W_2\to W_1$ is a linear isomorphism; we refer to this open subvariety as the \emph{non-singular locus}. The next lemma shows that the contraction map $r$ preserves centralizers over the non-singular locus. 
\begin{Lem}\label{Lem: centralizer contraction}
The restriction of $r$ to $\fu(W)_1^{iso}$ gives an (algebraic) $\U(W_2)$-torsor. Moreover, for $\de\in \fu(W)^{iso}_1$, we have an isomorphism
\[
\phi_\de:H_\de\iso T_{r(\de)}
\]
given by $(h_1,h_2)\mapsto h_1$, where $T_{r(\de)}\subset \U(W_1)$ is the centralizer of $r(\de)$. Finally, $\phi_\de$ induces an isomorphism between
\begin{equation}\label{eqn: iso kernels}
\cald(H_\de/F)\iso \cald(T_{r(\de)}/F)
\end{equation}
where $$\cald(H_\de/F)=\ker\left(H^1(F,H_\de)\to H^1(F, U(W_1)\times U(W_2))\right)$$ and \begin{equation*}\cald(T_{r(\de)}/F)=\ker\left(H^1(F,T_{r(\de)})\to H^1(F,U(W_1))\right).\end{equation*}
\end{Lem}
\begin{proof}
This is proved as Lemma 3.9 of \cite{Leslieendoscopy} for a general local field. As we are currently restricting to the non-archimedean setting and identifying
\[
\cald(H_\de/F)\cong\frakc(H_\de/F),
\]
a more direct argument is given in \cite[Lemma 5.12]{Lesliedescent}.
\end{proof}

The isomorphism (\ref{eqn: iso kernels}) implies that there is a bijection of rational orbits $\calo_{st}(\de)$ of $U(W_1)\times U(W_2)$  inside the stable orbit of $\de$ and rational conjugacy classes of $\Herm(W_1)$ inside the stable conjugacy class of $r(\de)$.

For $f\in C_c^\infty(\fu(W)_1)$, and $\de\in \fu(W)_1$ a semi-simple element, we define the \textbf{relative orbital integral} of $f$ by
\begin{equation}\label{eqn: orbital int def}
\RO(f,\de) = \displaystyle\iint_{H_\de\backslash U(W_1)\times U(W_2)}f(h_1^{-1}\de h_2) {d{h}_1d{h}_2}.
\end{equation} 
Our primary tool for studying relative orbital integrals is to relate them via the contraction map to orbital integrals of \emph{non-standard test functions} on the twisted Lie algebra $\Herm(W_1)$. The next lemma explains why this is effective for regular semi-simple orbits.

\begin{Lem}\label{Lem: regular semi-simple locus} There is an inclusion $\fu(W)_1^{rss}\subset\fu(W)_1^{iso}$.
\end{Lem}
\begin{proof}As in the proof of Proposition \ref{Lem: cat quotient Lie algebra 1}, we again pass to the algebraic closure $F=\Fbar$ and consider the action of $\GL_n\times \GL_n$ on $\fgl_n\times \fgl_n$. The invariant of this action is $\pi(X,Y)(t)=\det(tI-XY)$ as in Lemma \ref{Lem: cat quotient Lie algebra}. 

Recalling that the infinitesimal symmetric variety $\fgl_n\times \fgl_n$ is quasi-split, the element $(X,Y)$ is regular semi-simple if and only if the element
\[
Z=\left(\begin{array}{cc}&X\\Y&\end{array}\right)\in \fgl_{2n}(F)
\] is regular semi-simple. But $Z$ is regular semi-simple if and only if $\det(tI_{2n}-Z)$ has distinct roots. Now a simple exercise in linear algebra shows that
\[
\pi(X,Y)(t^2)=\det(tI_{2n}-Z).
\]
Thus, $Z\in \fgl_{2n}(F)^{rss}$ is possible only if $0$ is not a root of $\pi(X,Y)$, implying the lemma.
\end{proof}

This inclusion allows us to express relative orbital integrals at regular semi-simple points in terms of classical orbital integrals. Setting $\Omega:=r(\fu(W)_1^{iso})$, we see that $r$ gives a submersion from $\fu(W)_1^{iso}$ to $\Omega$. As in \cite[Section 1]{shalikagerms}, this implies that for $f\in C_c^\infty(\fu(W)^{iso}_1)$ and for $x\in \fu(W)_1$ regular semi-simple, the integral
\begin{equation}\label{eqn: pushforward operation}
    r_!(f)(r(x)) := \displaystyle\int_{U(W_2)}f(xu)du,
\end{equation}
converges and gives an element of $C^\infty_c(\Omega)$, and the induced map $C_c^\infty(\fu(W)_1^{iso})\to C^\infty_c(\Omega)$ is surjective. By Lemma \ref{Lem: regular semi-simple locus}, if we set $\Omega^{reg}=r(\fu(W)_1^{rss}),$ we get an induced (surjective) operator $C_c^\infty(\fu(W)_1^{rss})\to C^\infty_c(\Omega^{reg})$.

While the functions we will consider are not supported in $\fu(W)^{rss}_1$, each regular semi-simple orbit gives a closed subset of $\fu(W)^{rss}_1$, so that we may truncate any such function for the purpose of computing a particular orbital integral. Analyzing the behavior of the contraction of a particular function not supported in $\fu(W)^{rss}_1$ occupies Section \ref{Section: initial red}.
\begin{Lem}\label{Lem: orbital reduction} Suppose that $f\in C^\infty_c(\fu(W)_1)$ and $x\in \fu(W)_1^{rss}$. Then the relative orbital integral \eqref{eqn: orbital int def} converges, and we have the equality $$\RO(f,x) =\int_{T_{r(x)}\backslash U(W_1)}r_!(f)(g^{-1}r(x)g)d{g}=: \Orb(r_!(f),r(x)).$$
\end{Lem}
\begin{proof}
If $x$ is a regular semi-simple element, then everything is clearly absolutely convergent. By Lemma \ref{Lem: regular semi-simple locus}, we know that $x\in \fu(W)_1^{iso}$, so that replacing $f$ by $f\cdot \bfun_{U}$ for a open neighborhood $x\in U\subset \fu(W)_1^{rss}$ containing the $U(W_1)\times U(W_2)$-orbit of $x$, we see that $r_!(f)$ is well-defined on $U(W_1)\cdot r(x)$.  Lemma \ref{Lem: centralizer contraction} now implies that
\[
\RO(f,x) = \displaystyle\int_{T_{r(x)}\backslash U(W_1)}r_!(f)(g^{-1}r(x)g)d{g}.\qedhere
\]
\end{proof}
\subsection{Endoscopy for the twisted Lie algebra}
Lemmas \ref{Lem: centralizer contraction} and \ref{Lem: orbital reduction} allow us to utilize the contraction map to define endoscopic symmetric varieties for $\fu(W)_1$ and the associated transfer factors in terms of those for the twisted Lie algebra $\Herm(W_1).$ We briefly recall the necessary facts from this theory.  We refer the reader to \cite[Chapter 3]{RogawskiBook} or \cite{Xiao} for proofs of these facts.
\subsubsection{Matching}\label{Section: matching}
An elliptic endoscopic datum for $\Herm(W_1)$ is the same as a datum for the group $U(W_1),$ namely a triple $(U(V_a)\times U(V_b),s,\xi)$  where $a+b=n$, with $s\in \hat{U}(W_1)$ a semi-simple element of the Langlands dual group of $U(W_1)$, and an embedding $$\xi:\hat{U}(V_a)\times\hat{U}(V_b)\hra \hat{U}(W_1)$$ identifying $\hat{U}(V_a)\times\hat{U}(V_b)$ with the neutral component of the centralizer of $s$.

Fixing such a datum, we consider the endoscopic Lie algebra $\Herm(V_a)\oplus \Herm(V_b)$. Let $y\in \Herm(W_1)$ and $(y_a,y_b)\in\Herm(V_a)\oplus \Herm(V_b)$ be regular semi-simple. We recall the notion of matching orbits. For this, we first recall the notion of Jacquet--Langlands transfer between two non-isomorphic Hermitian spaces $W$ and $W'$. If we identify the underlying vector spaces (but not necessarily the Hermitian structures)
\begin{equation*}
  W\cong E^n\cong W',  
\end{equation*}
we have embeddings 
\[
\Herm(W),\:\Herm(W')\hra\fgl_n(E).
\]
Then $\de\in \Herm(W)$ and $\de'\in \Herm(W')$ are said to be \textbf{Jacquet--Langlands transfers} if they are $\GL_n(E)$-conjugate in $\fgl_n(E)$.  This is well defined since the above embeddings are determined up to $\GL_n(E)$-conjugacy. Note that if $\de$ and $\de'$ are Jacquet--Langlands transfers, then
\[
\de'=\Ad(g)(\de)
\]
for some $g\in \GL(W)$ and we obtain a well-defined cohomology class
\[
\inv(\de,\de')=[\sig\in \Gal(\Fbar/F)\mapsto g^{-1}\sig(g)]\in H^1(F,T_\de)
\]
extending the invariant map on $\cald(T_\de/F)$.

\begin{Def}\label{Def: endoscopic matching}
In the case that $W'=W_{a,b}:=V_a\oplus V_b$, we have an embedding
 $$\phi_{a,b}:\Herm(V_a)\oplus\Herm(V_b)\hra \Herm(W_{a,b}),$$ well defined up to conjugation by $U(W_{a,b})$. We say that $\de$ and $(\de_a,\de_b)$ are \textbf{transfers (or are said to match)} if $\de$ and $\phi_{a,b}(\de_a,\de_b)$ are Jacquet--Langlands transfers in the above sense.
\end{Def}

For later purposes, if $W\cong W_{a,b}$, we say that a matching pair $y$ and $(\de_a,\de_b)$ are a \textbf{\emph{nice matching pair}} if we may choose $\phi_{a,b}$ so that
 \[
 \phi_{a,b}(\de_a,\de_b) = \de.
 \]

\subsubsection{Orbital integrals}\label{Section:endoscopy character}
For $y\in \Herm(W_1)^{rss}$ and $f\in C_c^\infty(\Herm(W_1))$, we define the orbital integral
\[
\Orb(f,y)=\int_{T_y\backslash U(W_1)}f(g^{-1}y g)d{g},
\]
To an elliptic endoscopic datum $(U(V_a)\times U(V_b),s,\xi)$ and regular semi-simple element $y\in \Herm(W_1)$, there is a natural character (see \cite[Chapt. 3]{RogawskiBook}, for example)
\[
\ka:\cald(T_y/F)\to \cc^\times.
\]
Since we are in the non-archimedean setting, the set of rational conjugacy classes $\calo_{st}(y)$ in the stable conjugacy class of $y$ form a $\cald(T_y/F)$-torsor, and we have a map
\begin{equation}\label{invariant cohom}
\inv(y,-):\calo_{st}(y)\iso \cald(T_y/F)
\end{equation}trivializing the torsor by fixing the base point $y$. We then form the $\ka$-orbital integral
\[
\Orb^\ka(f,y) = \sum_{y'\sim_{st}y}\ka(\inv(y,y'))\Orb(f,y').
\]
 When $\ka=1$ is trivial, write $\SO=\Orb^\ka.$

In our case, the character $\ka$ is easy to describe. For matching elements $y$ and $(y_a,y_b)$,
\begin{equation}\label{eqn: cohom decomp}
H^1(F,T_y)=\prod_{S_1}\zz/2\zz=\prod_{S_1(a)}\zz/2\zz\times\prod_{S_1(b)}\zz/2\zz=H^1(F,T_{y_a}\times T_{y_b}),
\end{equation}
where the notation indicates which elements of $S_1$ arise from the torus $T_{y_a}$ or $T_{y_b}$.
\begin{Lem}
 Consider the character $\tilde{\ka}: H^1(F,T_y)\to \cc^\times$ such that on each $\zz/2\zz$ factor arising from $S_1(a)$, $\tilde{\ka}$ is the trivial map, while it is the unique nontrivial map on each $\zz/2\zz$-factor arising from $S_1(b)$. Then $$\ka=\tilde{\ka}|_{\cald(T_y/F)}.$$ 
\end{Lem}
\subsubsection{Smooth transfer}\label{Section: transfer factor}
The final notion is the \emph{transfer factor} of Langlands--Shelstad and Kottwitz. This is a function
\[
\De:[\Herm(V_a)\oplus \Herm(V_b)]^{rss}\times  \Herm(W_1)^{rss}\to \cc.
\]
The two important properties are
\begin{enumerate}
    \item $\De((\de_a,\de_b),\de) = 0$ if $\de$ does not match $(\de_a,\de_b),$ and
    \item if $\de$ is stably conjugate to $\de'$, then 
    \[
    \De((\de_a,\de_b),\de)\Orb^\ka(\de,f) = \De((\de_a,\de_b),\de')\Orb^\ka(\de',f).
    \]
\end{enumerate}
While the general definition, given in \cite{LanglandsShelstad1} for the group case and \cite{kottwitztransfer} in the quasi-split Lie algebra setting, is subtle, our present setting enjoys the following simplified formulation (cf. \cite[Appendix A]{Lesliedescent}). While our analysis of orbital integrals depend only on the formal properties above and Theorem \ref{Thm: smooth transfer lie algebra} below, we include this formulation for the convenience of the reader.

When $\de\in \Herm(W)$ and $(\de_a,\de_b)\in \Herm(V_a)\oplus \Herm(V_b)$ do not match, we set
\[
\De((\de_a,\de_b),\de)=0.
\]
Now suppose that $\de$ and $(\de_a,\de_b)$ match. We define the relative discriminant
\[
D(\de)=\prod_{x_a,x_b}(x_a-x_b),
\]
where $x_a$ (resp. $x_b$) ranges over the eigenvalues of $\de_a$ (resp. $\de_b$) in $\Fbar$. 
\begin{Rem}
This is precisely the quotient of the standard Weyl discriminants that occurs in the factor $\De_{IV}$ in \cite{LanglandsShelstad1}.
\end{Rem}
Recall our notation $W_{a,b}=V_a\oplus V_b$ and first assume that $W\cong W_{a,b}$ and that $\de$ and $(\de_a,\de_b)$ are a nice matching pair. In this case, the transfer factor is then given by
\begin{equation}\label{nice matching}
\De((\de_a,\de_b),\de):=\eta_{E/F}(D(\de))|D(\de)|_F,
\end{equation}
where $\eta_{E/F}$ is the quadratic character associated to $E/F$.

Now for any matching pair $\de$ and $(\de_a,\de_b)$, let
\[
\de'=\phi_{a,b}(\de_a,\de_b)\in \Herm(W_{a,b}).
\]
As discussed in Section \ref{Section: matching}, $\de$ and $\de'$ are Jacquet--Langlands transfers of each other and we set
\[
\De((\de_a,\de_b),\de) = \kappa(\inv(\de,\de'))\eta_{E/F}(D(\de))|D(\de)|_F,
\]
where $\ka:H^1(F,T_\de)\to \cc^\times$ is the character arising from the datum $(U(V_a)\times U(V_b),s,\eta)$ and $\inv$ is the extension of the invariant map discussed in Section \ref{Section: matching}. 

A pair of functions 
\[f\in C_c^\infty(\Herm(W_1))\:\text{ and }\:f_{a,b}\in C^\infty_c(\Herm(V_a)\oplus \Herm(V_b))
\]
are said to be smooth transfers (or matching functions) if the following conditions are satisfied:
\begin{enumerate}
\item for any matching regular semi-simple elements $y$ and $(y_a,y_b)$,
\begin{equation*}
\SO(f_{a,b},(y_a,y_b))= \Delta((y_a,y_b),y)\Orb^\kappa(f,y);
\end{equation*}
\item if there does not exist $y$ matching $(y_a,y_b)$, then
\begin{equation*}
\SO(f_{a,b},(y_a,y_b))= 0.
\end{equation*}
\end{enumerate}
The following theorem was first shown by combining \cite{LaumonNgo}, \cite{WaldsCharacteristic}, and \cite{Waldstransfert}; we will outline an alternative proof due to \cite{Xiao} in Section \ref{Section: nilpotent Weil}.
\begin{Thm}\label{Thm: smooth transfer lie algebra}
For any $f\in C_c^\infty(\Herm(W_1))$, there exists a smooth transfer $f_{a,b}\in C^\infty_c(\Herm(V_a)\oplus \Herm(V_b))$.
\end{Thm}

\subsection{Relative endoscopy for $(U(W),U(W_1)\times U(W_2))$}\label{Section: relative endoscopy}
 Recall that $\calv_n$ denotes our fixed set of representatives of the $\GL_n(E)$-orbits on $X_n$. Since we only consider the non-archimedean setting, $|\calv_n|=2$ for any $n$; we always assume that $I_n\in \calv_n$.

In \cite{Leslieendoscopy}, we defined a \emph{relative elliptic endoscopic datum} of $\fu(W)_1$ to be a quintuple
\[
\Xi=(U(V_a)\times U(V_b), s, \xi,\al,\be),
\]
where $(U(V_a)\times U(V_b), s, \xi)$ is an elliptic endoscopic datum for $U(W_1)$ and $\al\in\calv_a$ and $\be\in\calv_b$ are Hermitian forms on $E^a$ and $E^b$ respectively. We denote $V_\al = (E^a,\al)$ and $V_\be=(E^b,\be)$. For such a datum, we consider the Lie algebras
\[
\fu(V_a\oplus V_\al)\text{   and    }\fu(V_b\oplus V_\be),
\] and associated symmetric pairs
\[
\left(U(V_a)\times U(V_\al),\fu(V_a\oplus V_\al)_1\right) \text{ and }\left(U(V_b)\times U(V_\be),\fu(V_b\oplus V_\be)_1\right).
\]
  The direct sum of these symmetric pairs gives an \textbf{endoscopic (infinitesimal) symmetric pair} associated to the datum. This space comes equipped with the contraction map
\begin{align*}
r_{\al,\be}:\fu(V_a\oplus V_\al)_1\oplus \fu(V_b\oplus V_\be)_1&\lra \Herm(V_a)\oplus \Herm(V_b)\\
(\de_a,\de_b)&\longmapsto (r(\de_a),r(\de_b))
\end{align*}

We say that a regular semi-simple element $\de\in \fu(W)_1^{rss}$  \textbf{matches} the pair $$(\de_a,\de_b)\in [\fu(V_a\oplus V_\al)_1\oplus\fu(V_b\oplus V_\be)_1]^{rss}$$ if $r(\de)\in \Herm(W_1)$ and $r_{\al,\be}(\de_a,\de_b)\in \Herm(V_a)\oplus\Herm(V_b)$ match in the sense of Section \ref{Section: matching}.

 For matching elements $(\de_a,\de_b)$ and $\de$, we define the transfer factor
\begin{equation*}
\De_{rel}((\de_a,\de_b),\de):=\De(r_{\al,\be}(\de_a,\de_b), r(\de)),
\end{equation*}
where the right-hand side is the Langlands--Shelstad--Kottwitz transfer factor for the twisted Lie algebra from Section \ref{Section: transfer factor}. 
\subsubsection{Smooth transfer}

 Fix $\de\in \fu(W)_1^{rss}$ and let $\Xi$ be a relative endoscopic datum. Combining Lemma \ref{Lem: centralizer contraction} with the construction of Section \ref{Section:endoscopy character} gives a character $$\kappa:\cald(H_\de/F)\to \cc^\times,$$ with which we define the relative $\kappa$-orbital integral to be
\begin{equation*}\label{eqn: kappa orbital integral}
\RO^\kappa(f,\de) := \sum_{\de'\sim_{st}\de}\kappa(\inv(\de,\de'))\RO(f,\de'),
\end{equation*}
where $\de'$ runs over the set of rational orbits in $\fu(W)_1$ in the stable orbit of $\de$ and $$\inv(\de,\de'):=\inv(r(\de),r(\de')).$$ Here, $\inv(r(\de),-)$ is the invariant map (\ref{invariant cohom}). When $\ka=1,$ this is called the stable relative orbital integral and denoted by $\SRO=\RO^1$.

\begin{Def}\label{Def: transfer}
We say that $f\in C_c^\infty(\fu(W)_1)$ and $f_{\al,\be}\in C^\infty_c(\fu(V_a\oplus V_\al)_1\oplus \fu(V_b\oplus V_\be)_1)$ match (or are smooth transfers) if the following conditions are satisfied:
\begin{enumerate}
\item For any matching orbits $\de\in \fu(W)_1^{rss}$ and $(\de_a,\de_b)\in [\fu(V_a\oplus V_\al)_1\oplus \fu(V_b\oplus V_\be)_1]^{rss}$, we have an identify
\begin{equation}\label{relative transfer}
\SRO(f_{\al,\be},(\de_a,\de_b))= \Delta_{rel}((\de_a,\de_b),\de)\RO^\kappa(f,\de).
\end{equation}
\item If there does not exist  $\de$ matching $(\de_a,\de_b)$, then 
$
\SRO(f_{\al,\be},(\de_a,\de_b))= 0.
$
\end{enumerate}
\end{Def}

We conjectured that smooth transfers always exist in \cite[Conjecture 4.4]{Leslieendoscopy}, and showed that transfers exist for many test functions.  
\begin{Rem}
Recall that $\fu(W)_1$ has two natural contraction maps (\ref{eqn: choice of contraction}). For the reader concerned with canonicity, we remark that it is straightforward to show using the properties of the endoscopic transfer, Jacquet--Rallis transfer, and the Langlands--Shelstad--Kottwitz transfer factors  that these definitions are independent of our choice of contraction $r=r_1.$
\end{Rem}

\subsection{The endoscopic fundamental lemma}
We now assume that $E/F$ is an unramified extension of non-archimedean local fields of characteristic zero. Suppose that $V_n=W_1=W_2$ is split, and let $\Lam_n\subset V_n$ be a self-dual lattice. In this case,
\[
\fu(W)_1= \Hom_E(V_n,V_n) =\End(V_n)
\] 
and the ring of endomorphisms $\End(\Lam_n)\subset \End(V_n)$ of the lattice $\Lam_n$ is a compact open subset. Let $\bfun_{\End(\Lam_n)}$ denote the indicator function for this subset. This also induces a hyperspecial maximal compact subgroup $U(\Lam_n)\subset U(V_n)$.

 Now suppose that $\Xi$ is an elliptic relative endoscopic datum. Under our assumptions, we have $V_n\cong V_a\oplus V_b$ and we fix an isomorphism by imposing $\Lam_n=\Lam_a\oplus \Lam_b$ for fixed self-dual lattices $\Lam_a\subset V_a$ and $\Lam_b\subset  V_b$. Our measures conventions in Section \ref{measures} ensure that the given hyperspecial maximal subgroups of $U(V_n)\times U(V_n)$ and 
 \[
 (U(V_a)\times U(V_a))\times (U(V_b)\times U(V_b))
 \]each have volume $1$.

The following was conjectured in \cite{Leslieendoscopy}, and is the main result of this paper.
\begin{Thm}\label{Thm: full fundamental lemma}(Relative fundamental lemma)
 If $(\al,\be) = (I_a,I_b)$, the functions $\bfun_{\End(\Lam_n)}$ and $\bfun_{\End(\Lam_a)}\otimes\bfun_{\End(\Lam_b)}$ match. Otherwise, $\bfun_{\End(\Lam_n)}$ matches $0$.
\end{Thm}
 The proof of this statement follows a series of reductions, each of which changes the orbital integrals involved and the comparison needed. These reductions take up the rest of Part \ref{Part 1}, culminating in Theorem \ref{Thm: JR fundamental lemma for algebra}. %

\section{A relative fundamental lemma for the Hecke algebra}\label{Section: Initial reduction}
The goal of this section is to reduce the proof of Theorem \ref{Thm: full fundamental lemma} to Theorem \ref{Thm: endoscopic base change}, which states an explicit endoscopic transfers for certain modules of spherical Hecke algebras.

More precisely, note that Lemma \ref{Lem: regular semi-simple locus} implies that for any $x\in \fu(W)_1^{rss}$, 
\[
r(x)\in X_n=\{y\in \Herm(V_n): \det(y)\neq0\};
\]
in fact, using the notation from \eqref{eqn: pushforward operation}, $\Omega=r(\GL_n(E))\subset X_n$.
This motivates the study of orbital integrals of special functions on the Hermitian symmetric variety $X_n.$ We make a detailed analysis of the relevant part of the \emph{module of spherical functions of the symmetric variety} $X_n$ studied by Hironaka \cite{hironaka1999spherical}, expressing $r_!(\bfun_{\End(\Lam_n)})$ as an infinite sum of elements of this module (see Lemma \ref{Lem: densities}).

On the other hand, we may view functions on $X_n$ as elements of $C_c^\infty(\Herm(V_n))$ via extension-by-zero, where we can formulate a statement on endoscopic transfer for spherical functions on $X_n$. The precise statement is Theorem \ref{Thm: endoscopic base change}, the proof of which is takes up the rest of the paper. That this implies  Theorem \ref{Thm: full fundamental lemma} is Proposition {Prop: initial reduction}.

\subsection{A morphism of Hecke algebras}\label{Section: morphism Satake} Assume for the remainder of the section that $E/F$ is unramified. Recall that $q_E=q^2$ is the cardinality of the residue field of $E$.

We now construct the map of Hecke algebras which arises in the fundamental lemma of Hecke algebras stated below.
Let $\calh_{K_{n,E}}(\GL_n(E))$ denote the spherical Hecke algebra of $\GL_n(E)$. For any $(s_1,\ldots,s_n)\in \cc^n$, we recall the Satake transform
\begin{equation}\label{eqn: satake transform}
Sat(f)(s_1,\ldots,s_n) = \int_{\GL_n(E)}f(g)\prod_{i=1}^n|a_i|_E^{s_i-\frac{1}{2}(n+1-2i)}dg,
\end{equation}
where $g=nak$ is the Iwasawa decomposition of $g$, $dg$ is our chosen measure from Section \ref{measures}, and 
\[
a=\left(\begin{array}{ccc}
a_1 &    &  \\
&\ddots     & \\
&&a_n
\end{array}\right)\in T_n(E).
\]
This gives an algebra isomorphism 
\begin{equation*}
    Sat:\calh_{K_{n,E}}(\GL_n(E))\iso \cc[q^{\pm 2s_1},\ldots,q^{\pm 2s_n}]^{S_n}.
\end{equation*}
 Setting $t_i=q^{-2s_i}$, $ t=\diag(t_1,\ldots,t_n)\in \hat{T}_n\subset \GL_n(\cc)$ is an element of the diagonal split torus in the dual group of $\GL_n(E)$, and $$\cc[q^{\pm 2s_1},\ldots,q^{\pm 2s_n}]\cong\cc[\hat{T}_n]\cong \cc[Z_1^{\pm1},\ldots,Z_n^{\pm1}],$$ where 
\begin{equation}\label{eqn: weird normalization}
Z_i\left(\begin{array}{ccc}t_1&&\\&\ddots&\\&&t_n\end{array}\right)=t_i.
\end{equation}

Suppose now that $n=a+b$. Let $P_{(a,b)}=M_{(a,b)}N_{(a,b)}\subset \GL_n$ be the standard parabolic subgroup of $\GL_n$ such that $M_{(a,b)}\cong \GL_a\times \GL_b$ is realized as block diagonal matrices with $\GL_a$ appearing as the upper-left block. On the dual group side, consider the embedding
\begin{align*}
    \GL_a(\cc)\times \GL_b(\cc)&\lra \GL_n(\cc)\\ (m_1,m_2)&\longmapsto\left(\begin{array}{cc}
    \mu_b(\vp)m_1&\\&\mu_a(\vp)m_2
    \end{array}\right),
\end{align*}
where $\mu_s(t)=|t|_E^{s/2}$ for any $t\in E^\times$ and $s\in \cc.$
If $\pi_1\boxtimes \pi_2$ is a smooth irreducible representation of $M_{(a,b)}(E)$, this map of dual groups corresponds to the parabolic induction
\[
\pi_1\boxtimes \pi_2\mapsto \Ind_{P_{(a,b)}(E)}^{\GL_n(E)}(\pi_1(\mu_b\circ\det)\boxtimes \pi_2(\mu_a\circ\det))
\]
where $\Ind_{P_{(a,b)}(E)}^{\GL_n(E)}$ is normalized induction. Note that this induction functor does not send tempered representations to tempered representations. 

Restricting to unramified representations, this induces a dual map on Hecke algebras 
\begin{equation*}
\xi_{(a,b)}:\calh_{K_{n,E}}(\GL_n(E))\lra\calh_{K_{a,E}}(\GL_a(E))\otimes \calh_{K_{b,E}}(\GL_b(E)).
\end{equation*}
The following lemma makes this map explicit.
\begin{Lem}\label{Lem: parabolic descent}
 Define the parabolic descent $f^{P_{(a,b)}}\in C_c^\infty(M_{(a,b)}(E))$ to be
\[
f^{P_{(a,b)}}(m_1,m_2)=\de^{1/2}_{P_{(a,b)}}(m_1,m_2)\int_{N_{(a,b)}(E)}\int_{K_{n,E}}f\left(k\left(\begin{array}{cc}m_1&\\&m_2\end{array}\right)nk^{-1}\right)dk dn,
\]
where the measures are normalized as in Section \ref{measures} and 
\[
\de_{P_{(a,b)}}(m_1,m_2) = |\det(m_1)|_F^b|\det(m_2)|_F^{-a}.
\]
is the modular character of ${P_{(a,b)}}(E)$. 

Then the morphism $\xi_{(a,b)}$ of spherical Hecke algebras is given as follows: let $f\in \calh_{K_E}(\GL_n(E))$
\begin{align*}
\xi_{(a,b)}(f)(m_1,m_2)&= \mu_b(\det(m_1))\mu_a(\det(m_2))f^{P_{(a,b)}}(m_1,m_2).
\end{align*}
\end{Lem}
\begin{proof}
This expression is a direct consequence of the Satake isomorphism (see \cite{MinguezUnramified}, for example). 
\end{proof}

Using the Satake transform, this morphism gives a morphism
\[
\hat{\xi}_{(a,b)}:\cc[\hat{T}_n]^{S_n}\lra \cc[\hat{T}_{a}]^{S_a}\otimes\cc[\hat{T}_{b}]^{S_b}
\]that fits into a commutative diagram
\begin{equation}\label{eqn: satake diagram}
\begin{tikzcd}
\calh_{K_{n,E}}(\GL_n(E))\ar[d,"\xi_{(a,b)}"]\ar[r,"\mathrm{Sat}"]&\cc[\hat{T}_n]^{S_n}\ar[d,"\hat{\xi}_{(a,b)}"]\\
\calh_{K_{a,E}}(\GL_a(E))\otimes \calh_{K_{b,E}}(\GL_b(E))\ar[r,"\mathrm{Sat}"]&\cc[\hat{T}_{a}]^{S_a}\otimes\cc[\hat{T}_{b}]^{S_b}.
\end{tikzcd}
\end{equation}
We choose variables $\{X_i\}$ and $\{Y_j\}$ normalized analogously to (\ref{eqn: weird normalization}) so that $$\cc[\hat{T}_{a}]^{S_a}\otimes\cc[\hat{T}_{b}]^{S_b}\cong\cc[X_1^{\pm1},\ldots X_a^{\pm1}]^{S_a}\otimes \cc[Y_1^{\pm1},\ldots,Y_b^{\pm1}]^{S_b}.$$

\begin{Lem}\label{Lem: please be true}
The morphism $\hat{\xi}_{(a,b)}$ is the restriction to symmetric polynomials of the morphism 
\begin{align*}
     \cc[Z_1^{\pm1},\ldots,Z_n^{\pm1}]&\lra \cc[X_1^{\pm1},\ldots X_a^{\pm1}]\otimes \cc[Y_1^{\pm1},\ldots,Y_b^{\pm1}]\\
                        Z_i\qquad&\longmapsto\qquad \begin{cases}q^{-b}X_i &:i\leq a\\ q^{-a}Y_{i-a}&: i\geq a+1.
\end{cases}
\end{align*}
\end{Lem}
In order to prove the lemma, we introduce the some notation for partitions. For any $n\in \zz_{\geq0}$ we set
\[\mathrm{P}_n=\{\lam=(\lam_1,\lam_2,\ldots,\lam_n)\in \zz^n: \lam_1\geq\cdots\geq\lam_n\}.
\] 
and for any pair $n,d\in \zz_{\geq0}$ we set
\begin{equation}\label{eqn: partitions}
    \mathrm{P}^+_{n,d}=\left\{\lam=(\lam_1,\ldots,\lam_n)\in \mathrm{P}_n: \lam_i\geq0,\:\text{and}\: \sum_i\lam_i=d\right\}.
\end{equation}
Finally, for each $\lam\in \mathrm{P}_n$,
\begin{equation}\label{eqn: diagonal guy}
\varpi^{\lam}=\left(\begin{array}{ccc}\varpi^{\lam_1}&&\\
&\ddots&\\&&\varpi^{\lam_n}\end{array}\right)\in T_n(E).
\end{equation}
\begin{proof}
Let $\lam$ be a dominant coweight of $T_n(E)\subset \GL_n(E)$ and recall (see \cite[pg. 299]{macdonald}) that 
\[
Sat(\bfun_{K_{n,E}\vp^\lam K_{n,E}}) = q^{\la\lam,2\rho\ra}P_\lam(Z_1,\ldots,Z_n;q^{-2}),
\]
where we remind the reader that $q$ is the cardinality of the residue field of $F$. Here,
\[
P_\lam(x_1,\ldots,x_n;t)= V(t) \sum_{\sigma\in S_n}\sig\left(x_1^{\lam_1}\cdots x_n^{\lam_n}\prod_{\lam_i>\lam_j}\frac{x_i-tx_j}{x_i-x_j}\right)
\]is the $\lam$-th Hall--Littlewood polynomial, where $V(t)$ is an explicit rational function in $t$ \cite[pg. 208]{macdonald}. It is well known that as $\lam$ ranges over $\mathrm{P}_{n,d}^+$ for all $d\geq0$, these polynomials give a $\zz$-basis for $\zz[t][x_1,\ldots,x_n]^{S_n}$, so it suffices to compute $\hat{\xi}_{(a,b)}$ on these polynomials. A key point is that $P_\lam$ is homogeneous of degree $|\lam|$.

By \cite[Proposition 4.6 (2)]{MinguezUnramified}, parabolic descent on the spherical Hecke algebra $\calh_{K_{n,E}}(\GL_n(E))$ is dual to restriction to the Levi subgroup 
\[
(m_1,m_2)\longmapsto\left(\begin{array}{cc}
    m_1 &  \\
     & m_2
\end{array}\right).
\] We see that the parabolic descent $f\mapsto f^{P_{(a,b)}}$ in Lemma \ref{Lem: parabolic descent} corresponds under the Satake transform to
\[
P_\lam(Z_1,\ldots,Z_n;q^{-2})\longmapsto P_\lam(X_1,\ldots X_a,Y_1,\ldots,Y_b;q^{-2}).
\] This latter polynomial lies in the span of the products 
\[
P_\al(X_1,\ldots X_a;q^{-2})\cdot P_\be(Y_1,\ldots,Y_b;q^{-2}),
\]
where $\al$ is a partition of length $a$, $\be$ is of length $b$, and $|\lam|=|\al|+|\be|$. The coefficients of this expansion are well known (these may be derived from \cite[III (5.5)]{macdonald}, for example). To simplify notation for the moment, we will write $$P_\al(X_i;q^{-2}):=P_\al(X_1,\ldots X_a;q^{-2})$$ and similarly for other polynomials.

For $\lam\in P_{n,d}^+$, we may write $P_\lam(X_i,Y_j;q^{-2})$ as a sum
\begin{equation}\label{eqn: graded decomp}
    \sum_{d_a+d_b=d}P_{d_a,d_b,\lam}(X_i,Y_j;q^{-2})
\end{equation}

where
\[
P_{d_a,d_b,\lam}(X_i,Y_j;q^{-2})=\sum_{\al\in \mathrm{P}^+_{a,d_a}}\sum_{\be\in \mathrm{P}^+_{b,d_b}}c_{\al,\be}(\lam)P_\al(X_i;q^{-2})\cdot P_\be(Y_j;q^{-2}),
\]
for certain coefficients $c_{\al,\be}(\lam)\in \cc$, is the $(d_a,d_b)$-homogeneous part of $P_\lam(X_i,Y_j;q^{-2})$. The inverse Satake transform takes this decomposition to an expression
\[
\bfun_{K_{n,E}\vp^\lam K_{n,E}}^{P_{(a,b)}}=\sum_{d_a+d_b=d}f_{d_a,d_b}
\]
for some $f_{d_a,d_b}\in \calh_{K_{a,E}}(\GL_a(E))\otimes\calh_{K_{b,E}}(\GL_b(E))$. In particular, for any pair $(\al,\be)$, we have
\[
\bfun_{K_{n,E}\vp^\lam K_{n,E}}^{P_{(a,b)}}(\vp^\al,\vp^\be) = f_{|\al|,|\be|}(\vp^\al,\vp^\be).
\]
By Lemma \ref{Lem: parabolic descent}, it follows that 
\[
\xi_{(a,b)}(\bfun_{K_{n,E}\vp^\lam K_{n,E}})(\vp^\al,\vp^\be)=q^{-|\al|b-|\be|a}f_{|\al|,|\be|}(\vp^\al,\vp^\be).
\]
The commutativity of  (\ref{eqn: satake diagram}) thus implies
\begin{align*}
\hat{\xi}_{(a,b)}(P_\lam(Z_i;q^{-2}))&= \sum_{d_a+d_b=d}q^{-d_ab}q^{-d_ba}P_{d_a,d_b,\lam}(X_i,Y_j;q^{-2})\\
                                    &= \sum_{d_a+d_b=d}P_{d_a,d_b,\lam}(q^{-b}X_i,q^{-a}Y_j;q^{-2})
\end{align*}
Comparing with (\ref{eqn: graded decomp}), this proves the claim.
\end{proof}
\subsection{The spherical Hecke algebra for $X_n$}\label{Section: spherical hecke}
Set 
\[
X_n^{rss}=X_n\cap \Herm(V_n)^{rss};
\]this agrees with the invariant-theoretic notion of regular semi-simple locus of $X_n$ as a $U(V_n)$-variety.

Fix an elliptic endoscopic datum $(U(V_a)\times U(V_b), s, \xi)$ for $\Herm(V_n)$ and let $y\in X_n^{rss}$. Note that any element $(y_a,y_b)\in \Herm(V_a)\times \Herm(V_b)$ matching $y$ necessarily lies in $X_a\times X_b$. 

\begin{Rem}
    It is reasonable to view $X_a\times X_b$ as an \emph{endoscopic symmetric variety} for $X_n$.  In this way, an elliptic endoscopic datum of the symmetric variety $X_n$ is just an elliptic endoscopic datum $(U(V_a)\times U(V_b), s, \xi)$ for $\Herm(V_n)$. This is compatible with the theory of \cite{LeslieEndodata}, and Theorem \ref{Thm: endoscopic base change} plays the role of a fundamental lemma for the relative trace formula for the Galois pair $(\GL_n(E),U(V_n))$.
\end{Rem}

 In addition to the $U(V_n)$-action, the group $\GL_n(E)$ acts on $X_n$ via twisted conjugation: for any $g\in \GL_n(E)$ and $y\in X_n$
\[
g\ast y= gy{}^t\overline{g}.
\]
It follows from \cite{jacobowitz1962hermitian} that the $K_{n,E}$-orbits on $X_n$ are
\begin{equation}\label{eqn: orbits everywhere}
X_n=\bigsqcup_{\lam\in\mathrm{P}_{n}}K_{n,E}\ast\vp^\lam,
\end{equation}
where $\vp^\lam$ is defined in \eqref{eqn: diagonal guy}. The $\GL_n(E)$-action on $X_n$ induces an action of $\GL_n(E)$ on $C_c^\infty(X_n)$ given by 
\[
g\ast f(y) = f(g^{-1}\ast y), \quad \text{for any}\quad f\in C_c^\infty(X_n), \: g\in \GL_n(E)\;\text{ and }\; y\in X_n.
\] Set $\calh_{K_{n,E}}(X_n):=C^\infty_c(X_n)^{K_{n,E}}$ to be the vector space of $K_{n,E}$-invariant functions. This is known as the spherical Hecke algebra of the symmetric variety $X_n$.  Set $\bfun_\lam$ to be the indicator function of the orbit $K_{n,E}\ast \vp^\lam.$ The above orbit decomposition implies that $\{\bfun_\lam\}_{\lam\in \mathrm{P}_n}$ is a $\cc$-basis for $\calh_{K_{n,E}}(X_n)$. Note that with this notation 
 \[
 \bfun_0=\bfun_{X_n(\calo_F)}.
 \]
 The spherical Hecke algebra $\calh_{K_{n,E}}(\GL_n(E))$ acts on this space by
\[
f\ast  \phi(y)= \int_{\GL_n(E)}f(g^{-1})\phi(g\ast y)dg.
\]
The induced $\calh_{K_{n,E}}(\GL_n(E))$-module structure of $\calh_{K_{n,E}}(X_n)$ is well understood thanks to the work of Hironaka.

\begin{Prop}\cite[Theorem 2]{hironaka1999spherical}\label{Prop: Hironaka freeness}
As an $\calh_{K_{n,E}}(\GL_n(E))$-module, the spherical Hecke algebra $\calh_{K_{n,E}}(X_n)$ is free of rank $2^n$.
\end{Prop}
In particular, we have a distinguished rank $1$ sub-$\calh_{K_{n,E}}(\GL_n(E))$-module given by the embedding 
\begin{align*}
-\ast \bfun_0:\calh_{K_{n,E}}(\GL_n(E))&\varlonghookrightarrow \calh_{K_{n,E}}(X_n)\\f\qquad\qquad&\longmapsto\quad f\ast  \bfun_0.
\end{align*}

Suppose now that $(U(V_a)\times U(V_b),s,\xi)$ is an elliptic endoscopic datum of $X_n$. By a slight abuse of notation, we also denote the map
\begin{align*}
    \calh_{K_{a,E}}(\GL_a(E))\otimes \calh_{K_{b,E}}(\GL_b(E))&\lra \calh_{K_{a,E}}(X_a)\otimes \calh_{K_{b,E}}(X_b)\\
    f_a\otimes f_b\qquad\qquad\quad\:\:&\longmapsto\quad (f_a\ast \bfun_0)\otimes (f_b\ast \bfun_0)
\end{align*}
by $-\ast \bfun_0.$ Much of this paper consists of the proof of the following Theorem.

\begin{Thm}\label{Thm: endoscopic base change}
For any $\varphi\in \calh_{K_{n,E}}(\GL_n(E))$, the functions $\varphi\ast \bfun_0$ and $\xi_{(a,b)}(\varphi)\ast \bfun_0$ are smooth transfers of each other in the sense of Theorem \ref{Thm: smooth transfer lie algebra}.
\end{Thm}
\begin{Rem}
In proving Proposition \ref{Prop: Hironaka freeness}, Hironaka introduces the \emph{spherical Fourier transform} which she uses to give an isomorphism of $\calh_{K_{n,E}}(\GL_n(E))$-modules, 
\[
 H:\calh_{K_{n,E}}(X_n)\iso \calh_{K_{n}}(\GL_n(F)),
\]
where the module structure on the right is induced by the (injective) base change homomorphism 
\[
BC: \calh_{K_{n,E}}(\GL_n(E))\lra \calh_{K_{n}}(\GL_n(F)).
\]
The algebra structure on $ \calh_{K_{n,E}}(X_n)$ is given via transfer of the algebra structure of $ \calh_{K_{n}}(\GL_n(F))$ via $H$.
In particular, we have a commutative diagram of $\calh_{K_{n,E}}(\GL_n(E))$-modules, 
\begin{equation*}
\begin{tikzcd}
&\calh_{K_{n,E}}(\GL_n(E))\ar[dl,swap,"-\ast \bfun_0"]\ar[dr,"BC"]&\\
\calh_{K_{n,E}}(X_n)\ar[rr,"H"]&&\calh_{K_n}(\GL_n(F)).
\end{tikzcd}
\end{equation*}
\end{Rem}
\begin{Rem}
It is tempting to extend the statement of Theorem \ref{Thm: endoscopic base change} to the entire Hecke algebra $\calh_{K_{n,E}}(X_n).$ Indeed, using the spherical Fourier transform of Hironaka, we may extend the morphism $\xi_{(a,b)}$ to a module homomorphism
\begin{equation*}
    \xi_{(a,b)}:\calh_{K_{n,E}}(X_n)\lra \calh_{K_{a,E}}(X_a)\otimes \calh_{K_{b,E}}(X_b),
\end{equation*}
and conjecture that for any $\varphi\in\calh_{K_{n,E}}(X_n)$, $ \xi_{(a,b)}(\varphi)$ is a smooth transfer in the sense of Theorem \ref{Thm: smooth transfer lie algebra}. This should play the role of the full fundamental lemma for the relative trace formula for the Galois pair $(\GL_n(E),U(V_n))$.

To make this precise, we would need to deal with several complications not germane to our current discussion. For example, preliminary calculations suggest augmenting the Langlands--Shelstad--Kottwitz transfer factors in a precise way for such a generalization to hold. We plan to return to this in a future paper.
\end{Rem}
\subsection{The initial reduction}\label{Section: initial red}
We now show that Theorem \ref{Thm: endoscopic base change} implies the relative endoscopic fundamental lemma. Let $\bfun_{\End(\Lam_n)}$ and $\bfun_{\End(\Lam_a)}\otimes \bfun_{\End(\Lam_b)}$ be as in the statement of Theorem \ref{Thm: full fundamental lemma}.

Recalling the operator $r_!$ defined by \eqref{eqn: pushforward operation}, set $\Phi^n:=r_!\bfun_{\End(\Lam_n)}$. Using the notation from Section \ref{Section: pushforward}, this gives a function on $\Omega^{reg}=r(\fu(W)_1^{rss})\subset X_n$, which we view as a locally constant function on $X_n$ by extending-by-zero over the compliment of $\Omega^{reg}$. The resulting function is not compactly supported. It is \emph{almost-compactly supported} in the sense that if we decompose of $X_n$ into disjoint closed (in the Hausdorff topology) subsets
\[
X_n = \bigsqcup_{d\in \zz}X_{n,d},\quad X_{n,d} = \{h\in X_n: |\det(h)|_F = q^{-d}\}
\]
and set $\Phi^n_d = \Phi^n\cdot\bfun_{X_{n,d}}$, then $\Phi^n_d\in C_c^\infty(X_n)$ for all $d\in\zz$.

We now give a formula for $\Phi^n_d$. Suppose that $T\in X_n$ and set 
\begin{equation}\label{eqn: local rep density}
\mathfrak{m}_0(T)=\begin{cases}\displaystyle\int_{U(V_n)}1_{\End(\Lam_{n})}(xh)\,dh& \text{if } T=xx^\tau , x\in \GL_n(E),\\
0& \text{otherwise}. 
\end{cases}
\end{equation}
This is the (normalized) local representation density, denoted as $\mathrm{Den}(L)$ in \cite{LiZhang} when $L=x^\tau\Lam_n$ is the associated lattice (one may use formula \cite[(3.6.1.1)]{LiZhang}, for example). Note that $\mathfrak{m}_0(T)=0$ unless $T$ is integral, which holds if and only if $L\subset \Lam_n \subset L^\vee = x^{-1}\Lam_n$. 

{To see this, recall the formula \cite[(3.6.1.1)]{LiZhang}
\[
\mathrm{Den}(L) =\#\{{L\subset L'\subset L^\vee}:{L'\text{ self-dual}}\}.
\]
Now if $T=xx^\ast$, it is easy to see that
\[
\mathfrak{m}_0(T) = \#\{{{[h]\in U(V_n)/U(\Lam_n)}:{xh\in \End(\Lam_n)}}\}.
\]
The fact that any two self-dual lattices are conjugate by $U(V_n)$ implies that these two index sets are in bijection. Indeed, for such an $[h]$, 
\[
h^{-1}x^\tau\Lam_n\subset \Lam_n\subset h^{-1}x^{-1}\Lam_n\Leftrightarrow L\subset L':=h\Lam_n\subset L^\vee.
\]
On the other hand, each self-dual lattice $L'$ may be written $L' = h\Lam_n$ for some $[h]\in U(V_n)/U(\Lam_n)$ and the constraint that $L\subset L'\subset L^\vee$ is then equivalent to $xh\in \End(\Lam_n)$.}
\begin{Lem}\label{Lem: densities}
We have $\Phi^n _d\equiv 0$ if $d$ is odd or $d<0$. Moreover,
$$\Phi_{2d}^n = \sum_{\lam\in \mathrm{P}^+_{n,2d}}\mathfrak{m}_0(\vp^\lam) \bfun_{\lam},$$
 where $\mathrm{P}^+_{n,2d}$ is defined in \eqref{eqn: partitions} and $\bfun_\lam$ is the indicator function of the orbit $K_{n,E}\ast \vp^\lam.$
\end{Lem}
\begin{proof}
Since $\supp(\Phi^n )\subset r(\End(\Lam_n)),$ if $x\in \supp(\Phi^n )$, then $\det(x)\in \Nm_{E/F}(\calo_E)$. Our assumption that $E/F$ is unramified now implies the vanishing statement. Now for any $g\in K_{n,E}$,
\[
\Phi^n(gr(x){}^t\overline{g}) = \int_{U(V_n)} \bfun_{\End(\Lam_n)}(gxh) dh=\int_{U(V_n)}\bfun_{\End(\Lam_n)}(xh) dh=\Phi^n (r(x)).
\]
Thus, $\Phi^n$ is constant on $K_{n,E}$-orbits of $X_n$, with the value given by the formula in the statement. The lemma now follows from the $K_{n,E}$-orbit decomposition (\ref{eqn: orbits everywhere}).
\end{proof}
\begin{Rem}
While it is a striking fact that the coefficients in $\Phi$ are given by representation densities, we will not make use of this fact in this paper. It does play a role in the arithmetic aspects of a comparison of trace formulae designed to attack Conjecture \ref{Conj 1}, which relies on the results of this work.
\end{Rem}A corollary of this and Lemma \ref{Lem: orbital reduction} is the following restatement of Theorem \ref{Thm: full fundamental lemma}.
\begin{Cor}\label{Cor: initial restatement}
 Theorem \ref{Thm: full fundamental lemma} holds if and only if for every $d\in \zz_{\geq0}$, the functions
\[
\Phi^n_{2d} \text{ and }  \sum_{d_a+d_b=d}\Phi^a_{2d_a}\otimes\Phi^b_{2d_b}
\]  
match in the sense of Theorem \ref{Thm: smooth transfer lie algebra}.
\end{Cor}
\begin{proof}
This follows in a straightforward fashion from our previous discussion and Lemma \ref{Lem: orbital reduction}.
\end{proof}
To relate this corollary to Theorem \ref{Thm: endoscopic base change}, we record the following elementary lemma.
\begin{Lem}\label{Lem: elementary lemma}
For $\phi\in \calh_{K_{n,E}}(\GL_n(E))$, one has
\begin{equation*}
r_!(\phi)=\phi\ast \bfun_0.
\end{equation*}
\end{Lem}
\begin{proof}
First we prove the claim for $\phi=\bfun_{K_{n,E}}$. In this special case, it is immediate that $\bfun_{K_{n,E}}\ast \bfun_0=\bfun_0$. On the other hand, for any $x\in \GL_n(E)$
\[
r_!(\bfun_{K_{n,E}})(xx^\tau)=\int_{U(V_n)}\bfun_{K_{n,E}}(xu)du.
\]
The right-hand side is only non-zero if there exists $u\in U(V_n)$ such that $xu\in K_{n,E}$. This implies that the left-hand side is non-zero only if $xx^\tau\in K_{n,E}\ast I_n$. Since $r_!(\bfun_{K_{n,E}})\in \calh_{K_{n,E}}(X_n)$, we must have $r_!(\bfun_{K_{n,E}})=c\bfun_0$ for some constant $c\in \cc$. Since our measure conventions give $U(\Lam_n)=U(V_n)\cap K_{n,E}$ volume $1$, we check that
\[
c= r_!(\bfun_{K_{n,E}})(1) = \int_{U(V_n)}\bfun_{K_{n,E}}(u)du = \vol(U(\Lam_n)) =1,
\]
proving the claim for $\phi=\bfun_{K_{n,E}}$.

In general, if $\phi\in\calh_{K_{n,E}}(\GL_n(E))$, then for any other $\phi_1\in \calh_{K_{n,E}}(\GL_n(E))$,
\begin{align*}
\phi\ast  r_!(\phi_1)(xx^\tau) &= \int_{\GL_n(E)}\phi(g^{-1}) r_!(\phi_1)(gxx^\tau{}^t\overline{g})dg\\	
			&=\int_{\GL_n(E)}\int_{U(V_n)}\phi(g^{-1})\phi_1(gxu)dudg\\
			&=\int_{U(V_n)}(\phi\ast \phi_1)(xu)du=r_!(\phi\ast\phi_1)(xx^\tau),
\end{align*}
where $\phi\ast\phi_1$ denotes the convolution product.

Setting $\phi_1=\bfun_{K_{n,E}},$ and using the $\calh_{K_{n,E}}(\GL_n(E))$-module structure, we find that
\[
    r_!(\phi)=r_!(\phi\ast\bfun_{K_{n,E}})= \phi\ast r_!(\bfun_{K_{n,E}})=\phi\ast \bfun_0. 
 \qedhere
\]
\end{proof}

\begin{Prop}\label{Prop: initial reduction}
Theorem \ref{Thm: endoscopic base change} implies Theorem \ref{Thm: full fundamental lemma}.
\end{Prop}
\begin{proof}
Set $\bfun_{d} = \sum_{\lam\in \mathrm{P}^+_{n,d}}\bfun_{K_{n,E}\vp^\lam K_{n,E}}$. Combining the definition of $\Phi^n_{2d}$ with Lemma \ref{Lem: elementary lemma}, 
\[
\Phi^n_{2d}=r_!(\bfun_{d})=\bfun_{d}\ast \bfun_0.
\]
If we assume the statement of Theorem \ref{Thm: endoscopic base change}, Corollary \ref{Cor: initial restatement} implies that it suffices to show that 
\begin{equation}\label{eqn: local goal}
\xi_{(a,b)}(\bfun_{d})=\sum_{d_a+d_b=d}\bfun_{d_a}\otimes\bfun_{d_b}.
\end{equation}
We claim this follows if we can show that
\begin{equation}\label{eqn: sft special}
Sat(\bfun_{d})(Z_i) = q^{d(n-1)}\sum_{{\bf m}\in \zz_{d,+}^n}Z^{{\bf m}},
\end{equation}
where $ \zz_{d,+}^n =\{{\bf m}\in \zz^n_{\geq0}: \sum_i m_i=d\}$ and
\[
Z^{{\bf m}}=\prod_iZ_i^{m_i}.
\]

Indeed, (\ref{eqn: sft special}) implies that for each ${d_a+d_b=d}$,
\[
Sat\left(\bfun_{d_a}\otimes\bfun_{d_b}\right)(X_i,Y_j) =
q^{d_a(a-1)+d_b(b-1)}\sum_{{\bf a}\in \zz_{d_a,+}^a}\sum_{{\bf b}\in \zz_{d_b,+}^b}X^{{\bf a}}Y^{{\bf b}}.
\]

There is a bijection
\begin{align*}
\bigoplus_{d_a+d_b=d}\left(\zz_{d_a,+}^a\oplus\zz_{d_b,+}^b\right)&\iso  \zz_{d,+}^n\\
({\bf a},{\bf b})\quad\qquad&\mapsto {\bf a}\cup{\bf b},
\end{align*}
where $\cup$ denotes concatenation. Applying this and Lemma \ref{Lem: please be true} to 
\[
Sat(\bfun_d)\in \sspan_{\cc}\left\{P_\lam(Z_i;q^{-2}): \lam\in P_{n,d}^+\right\},
\]
the equality (\ref{eqn: sft special}) implies that
\begin{align*}
\hat{\xi}_{(a,b)}(Sat(\bfun_d))(X_i,Y_j)&=q^{d(n-1)}\sum_{d_a+d_b=d}q^{-d_ab}q^{-d_ba}\sum_{{\bf a}\in \zz_{d_a,+}^a}\sum_{{\bf b}\in \zz_{d_b,+}^b}X^{{\bf a}}Y^{{\bf b}}\\
&=\sum_{d_a+d_b=d} q^{d_a(a-1)+d_b(b-1)}\sum_{{\bf a}\in \zz_{d_a,+}^a}\sum_{{\bf b}\in \zz_{d_b,+}^b}X^{{\bf a}}Y^{{\bf b}}\\ &=\sum_{d_a+d_b=d}Sat\left(\bfun_{d_a}\otimes\bfun_{d_b}\right)(X_i,Y_j).
\end{align*}
By the commutativity of (\ref{eqn: satake diagram}), this is equivalent to (\ref{eqn: local goal}). 

To prove (\ref{eqn: sft special}), we use (\ref{eqn: satake transform}) and the Iwasawa decomposition on $\GL_n(E)$ to compute that
\[
Sat(\bfun_{d})(Z_1,\ldots,Z_n)=\sum_{\lam\in \zz^n}q^{\la\lam,2\rho\ra}Z^{\lam}\int_{N_n(E)}\bfun_d(u\vp^\lam)du,
\]
where $N_n(E)$ is the $E$-points of the unipotent radical of the Borel subgroup $B_n(E)\subset \GL_n(E)$ of upper triangular matrices.
A standard computation (see \cite[Section 6]{OffenJacquet}) shows that
\[
\int_{N_n(E)}\bfun_d(u\vp^\lam)du=\begin{cases}q^{\sum_{i}(i-1)2\lam_i}&: \lam\in \zz_{d,+}^n,\\0&:\text{otherwise}.\end{cases}
\]
 Therefore, since $\la\lam,2\rho\ra+\sum_{i}(i-1)2\lam_i=2d(n-1)/2=d(n-1)$, we obtain \eqref{eqn: sft special}.
\end{proof}

\section{Nilpotent orbital integrals and the second reduction}\label{Section: nilpotent Weil}
In this section, we reduce Theorem \ref{Thm: endoscopic base change} to a statement of explicit transfers in the context of the Lie algebra version of the Jacquet--Rallis transfer. This relies on recent results of Xiao relating endoscopic transfer for the twisted Lie algebra to \emph{germ expansion of orbital integrals} in the context of the Jacquet--Rallis transfer. We recall the fundamental notions and results in the next section, review the main result of Xiao in Section \ref{Section: Xiao}, and execute the reduction in Section \ref{Section: first reduction}.

\subsection{Jacquet--Rallis transfer and fundamental lemma}\label{Section: JR theory}
 Let $E/F$ be a quadratic extension of non-archimedean local fields and let $V$ be an arbitrary $n$-dimensional Hermitian space. The linear side of the Lie algebra version of the Jacquet--Rallis comparison is
\[
\fgl_n(F)\times F^n\times F_n,
\]
where $F_n=(F^n)^\ast$ is the vector space of $1\times n$ row vectors. We consider the diagonal action of $\GL_n(F)$ on this space. The unitary side of Jacquet--Rallis transfer considers the space
\[
\Herm(V)\times V,
\]
with the diagonal action of $U(V)$.
\subsubsection{Linear side}
We define the invariants of $(x,v,v^\ast)\in \fgl_n(F)\times F^n\times F_n$ to be $\chi(x,v,v^\ast) = (a,b)\in F^n\times F^n$ with 
\[
a_i=\text{ coefficient of $t^i$ in }\det(tI-x),\qquad\text{and  }b_i =  v^\ast(x^i v). 
\]
An element $(x,v,v^\ast)$ is \emph{regular semi-simple} if and only if $$\det\left(\left(\la v^\ast, x^{i+j}v\ra\right)_{i,j}\right)\neq 0.$$ Moreover, the stabilizer of a regular semi-simple element is trivial and two regular semi-simple elements have the same invariants if and only if they are in the same $\GL_n(F)$-orbit \cite{rallis2007multiplicity}.

For $f\in C_c^\infty(\fgl_n(F)\times F^n\times F_n)$, we consider the orbital integrals
\[
\Orb^{\GL_n(F),\eta}(f,(x,v,v^\ast))=\int_{\GL_n(F)}f(\Ad(g)x,gv,v^\ast g^{-1})\eta(\det(g))dg,
\]
where $\eta=\eta_{E/F}$ is the quadratic character associated to $E/F$. This gives a $(\GL_n(F),\eta)$-invariant distribution. To compare with unitary orbital integrals, we multiply by the transfer factor $\omega$ introduced in \cite[Section 3]{ZhangFourier}. This function is defined for any regular semi-simple $(x,v,v^\ast)$ as
\begin{equation}\label{eqn: transfer factor JR 1}
\omega(x,v,v^\ast) =\eta\left(\det[v|xv|\ldots|x^{n-1}v]\right),
\end{equation}
where $[v|xv|\ldots|x^{n-1}v]$ denotes the $n\times n$ matrix with columns $x^iv$ for $i=0,1,\ldots,n-1$.
\subsubsection{Unitary side}
We similarly associate to an element $(y,w)\in\Herm(V)\times V$ the invariants $\chi_V(y,w)= (a,b)$, where
\[
a_i=\text{ coefficient of $t^i$ in }\det(tI-y),\qquad\text{and  }b_i = \la w,x^i w\ra_V. 
\]
It is clear that these values lie in $F$. For $f\in C_c^\infty(\Herm(V)\times V)$, we consider the orbital integrals
\[
\Orb^{U(V)}(f,(y,w))=\int_{U(V)}f(\Ad(g)y,gw)dg.
\]
As in the linear case, the stabilizer of a regular semi-simple element is trivial and two regular semi-simple elements have the same invariants if and only if they are in the same $U(V)$-orbit. 
\subsubsection{Transfer}\label{Section: JR transfer}
Two regular semi-simple elements $(x,v,v^\ast)\in \fgl_n(F)\times F^n\times F_n $ and $(y,w)\in\Herm(V)\times V$ are said to \emph{match} if their invariants agree. It is helpful to view this matching invariant theoretically.

Suppose that $\cala$ is the categorical quotient of $\fgl_n\times \mathbb{G}_a^n\times (\mathbb{G}_a^n)^\ast$ by $\GL_n$. There is a natural isomorphism of affine varieties (see \cite[Proposition 2.2.2.1]{chaudouard2019relative})
\[
\cala=\fgl_n\times \mathbb{G}_a^n\times (\mathbb{G}_a^n)^\ast//\GL_n\cong\Herm(\mathrm{V})\times \mathrm{V}//\U(V)\cong \A^{2n},
\] where here $\A$ denotes the affine line over $F$. 

The image of the regular locus is an open subvariety $\cala^{rss}\subset\cala$, and for any $a\in \cala^{rss}(F)$, the inverse image of $a$ in $\fgl_n(F)\times F^n\times F_n$ consists of a single $\GL_n(F)$-orbit. On the other hand, the preimage of $a$ in $\Herm(V)\times V$ is either empty or a single $U(V)$-orbit. On $F$-points, this gives a bijection of regular semi-simple orbits \cite{rallis2007multiplicity}:

\begin{equation}\label{Thm: JR orbit transfer}
    \left[\GL_n(F)\backslash \fgl_n(F)\times F^n\times F_n\right]^{rss}\iso \bigsqcup_{V\in\calv_n}\left[U(V)\backslash \Herm(V)\times V\right]^{rss}.
\end{equation}
Here $V\in\calv_n$ runs through our representatives of the isomorphism classes of non-degenerate Hermitian spaces of dimension $n.$

We say that functions $f\in C_c^\infty(\fgl_n(F)\times F^n\times F_n)$ and $\{f_V\}_{V}$ with $f_V\in C_c^\infty(\Herm(V)\times V)$ are said to be Jacquet--Rallis transfers if for any matching regular semi-simple elements $(x,v,v^\ast)$ and $(y,w)$, the following identify holds
\begin{equation}\label{eqn: JR matching}
\omega(x,v,v^\ast)\Orb^{\GL_n(F),\eta}(f,(x,v,v^\ast))=\Orb^{U(V)}(f_V,(y,w)).
\end{equation}
 The existence of smooth transfer follows from Theorem \ref{Thm: JR transfer} below.

\subsubsection{A variant}
In Section \ref{Section: Weil rep} below, we will need to also consider a slight variant of the preceding set-up, which is the version of the Jacquet--Rallis transfer for the Lie algebra considered by \cite{ZhangFourier}.

To this end, note that there is a natural embedding of $\GL_n(F)$-modules
\begin{align*}
\fgl_n(F)\times F^n\times F_n&\hra \fgl_{n+1}(F)\\
(x,v,v^\ast)&\longmapsto \left(\begin{array}{cc}x&v\\v^\ast&0\end{array}\right),
\end{align*}
where $\GL_n(F)$ acts on $\fgl_{n+1}(F)$ via the adjoint action as a subgroup of $\GL_{n+1}(F)$. In particular, we have an isomorphism of $\GL_{n}(F)$-representations
\begin{equation}\label{eqn: function decomp linear}
C_c^\infty(\fgl_{n+1}(F))\iso C_c^\infty(\fgl_{n}(F)\times F^{n}\times F_{n})\otimes C_c^\infty(F).
\end{equation}
Similarly, for an $n$-dimensional Hermitian space $V$ there is a natural embedding of $U(V)$-modules
\begin{align*}
\Herm(V)\times V&\hra\Herm(V\oplus Ee_0)\\
(y,w)&\longmapsto\left(\begin{array}{cc}x&w\\\la w,-\ra_V&0\end{array}\right),
\end{align*}
where we impose that $\la e_0,e_0\ra=1$ and that the sum is direct. As in the linear case, this induces an isomorphism of $U(V)$-representations
\begin{equation}\label{eqn: function decomp unitary}
C_c^\infty(\Herm(V\oplus Ee_0))\iso C_c^\infty(\Herm(V)\times V)\otimes C_c^\infty(F).
\end{equation}
Noting that the spaces on the right-hand sides of (\ref{eqn: function decomp linear}) and (\ref{eqn: function decomp unitary}) are related by the matching of orbital integrals (\ref{eqn: JR matching}), we extend the notion of matching functions to one between $C_c^\infty(\fgl_n(F))$ and $C_c^\infty(\Herm(V_n))$ compatible with these isomorphisms. 

More specifically, we say that 
\[
X=\left(\begin{array}{cc}x&v\\v^\ast&d\end{array}\right)\in\fgl_{n+1}(F)
\]and 
\[
Y=\left(\begin{array}{cc}y&w\\\la w,-\ra_V&\lam\end{array}\right)\in\Herm(V\oplus Ee_0)
\]
match (resp. are regular semi-simple) if $(x,v,v^\ast)$ matches $(y,w)$ and $d=\lam$ (resp. if $(x,v,v^\ast)$ and $(y,w)$ are regular semi-simple in the sense of Section \ref{Section: JR theory}). 

For $f\in C_c^\infty(\fgl_{n+1}(F))$, we consider the orbital integrals
\[
\Orb^{\GL_{n}(F),\eta}(f,X)=\int_{\GL_{n}(F)}f(\Ad(g)X)\eta(\det(g))dg, \qquad\text{ for  }X\in \fgl_{n+1}(F)^{rss}.
\]
 For any $X\in \fgl_{n+1}(F)^{rss}$, we define the transfer factor $\omega:\fgl_{n+1}(F)^{rss}\to \cc$ to be
\begin{equation}\label{eqn: JR transfer factor 2}
\omega(X) = \eta\left(\det[e_{n+1}|Xe_{n+1}|\ldots|X^{n}e_{n+1}]\right),
\end{equation}
where 
\[
e_{n+1}= {}^t[0,\ldots,0,1]\in F^{n+1}.
\]
For $f_{V}\in C_c^\infty(\Herm(V\oplus Ee_0))$, we consider the orbital integrals
\[
\Orb^{U(V)}(f_{V},Y)=\int_{U(V)}f_{V}(\Ad(h)Y)dh, \qquad\text{ for  }Y\in \Herm(V\oplus Ee_0)^{rss}.
\]
The functions $f$ and $\{f_{V}\}_{V\in \calv_n}$ are said to be Jacquet--Rallis transfers if for any regular semi-simple $X\in\fgl_{n+1}(F)$ and $Y\in \Herm(V\oplus Ee_0)$ that match, we have
\begin{equation}\label{eqn: regular matching final}
\omega(X)\Orb^{\GL_{n}(F),\eta}(f,X)=\Orb^{U(V)}(f_{V},Y).
\end{equation}

\begin{Rem}
There are now two notions of ``Jacquet--Rallis transfer.'' These are on different spaces, so it will always be clear in context which comparison is meant. Nevertheless, to ensure that this does not cause confusion, we will refer to Jacquet--Rallis transfer in the sense of (\ref{eqn: JR matching}) or (\ref{eqn: regular matching final}) to specify which is intended.
\end{Rem}

We now state the two main results in this theory: the existence of Jacquet--Rallis transfers and the fundamental lemma for the Lie algebra. We note that the results of Part \ref{Part 1} \emph{do not} rely on either of these results, though both are crucial to Part \ref{Part 2}.

\begin{Thm}\cite{ZhangFourier}\label{Thm: JR transfer}
For any $f\in C_c^\infty(\fgl_{n+1}(F))$, there exists a transfer $\{f_V\}_{V\in \calv_n}$. Conversely, for any $\{f_V\}_V$, there exists a transfer $f$.
\end{Thm}

Assume now that $E/F$ is an unramified extension of $p$-adic fields and assume $V=V_n$ is our fixed split Hermitian form. 
\begin{Thm}\cite{YunJR}\label{Thm: JR fundamental lemma}
The functions $\bfun_{\fgl_{n+1}(\calo_F)}$ and $\{\bfun_{\Herm(V_n\oplus Ee_0)(\calo_F)},0\}$ are Jacquet--Rallis transfers.
\end{Thm}
\begin{Rem}This theorem was first proved by Yun for characteristic $p$ local fields when $p>n+1$ and transferred to characteristic zero by Gordan in \cite{YunJR}, provided the residual characteristic is sufficiently high.  Beuzart-Plessis gave a remarkable proof of this statement in characteristic zero \emph{with arbitrary residual characteristic} \cite{beuzart2019new} via local methods. Additionally, global proofs have appeared by W. Zhang (for $p\geq n$) \cite{ZhangAFL} and Z. Zhang (for $p>2$) \cite{zhang2021maximal}.
\end{Rem}

\subsection{Nilpotent orbital integrals}\label{Section: Xiao} We recall some results from \cite{Xiao} that we use in the sequel. Roughly speaking, one may recover $\kappa$-orbital integrals on the twisted Lie algebra $\Herm(V)$ as \emph{limits of the orbital integrals} discussed in the previous sections by studying certain singular orbits in the context of Jacquet--Rallis transfer. 

 Recall from Section \ref{measures} that for a regular semi-simple element $\de\in \Herm(V)$, there is a decomposition $$F[\de]:=F[X]/(char_{\de}(X))=\prod_{i=1}^mF_i,$$ where $F_i/F$ is a field extension and $char_{\de}(X)$ denotes the characteristic polynomial of $\de$. Setting $S_1=\{i: F_i\nsupseteq E\}$, we have $H^1(F,T_\de) = \prod_{S_1}\zz/2\zz$. 
 Recall that the subgroup 
$\cald(T_\de/F)\subset H^1(F,T_\de)$ parameterizes rational conjugacy classes in the stable conjugacy class $\calo_{st}(\de)$. Let $S$ be the union of these conjugacy classes and the set of conjugacy classes $\calo\subset\Herm(V')$ that are Jacquet--Langlands transfers of $\de$, where $V'$ represents the other isomorphism class of $n$ dimensional Hermitian space over $F$.

\begin{Prop}\cite[Proposition 3.8]{Xiao}\label{Prop: extended torsor}
There is a natural $H^1(F,T_\de)$-torsor structure on $S$ extending the above classical $\cald(T_\de/F)$-torsor structure.
In particular, there is a natural bijection between $S$ and $\prod_{S_1}\zz/2\zz$.
\end{Prop}

Now fix a regular semi-simple element $(x,v,v^\ast)\in \fgl_n(F)\times F^n\times F_n$, and let $F[x]=\prod_{i=1}^mF_i$, $\{1,\ldots, m\}=S_1\sqcup S_2$ be as above. Then $T_x=F[x]^\times$ is the centralizer of $x$ in $\GL_n(F)$, and we define $T_1=\prod_{i\in S_1}F_i^\times$.

The action of $x$ on $F^n$ induces the decomposition
\[
F^n=\bigoplus_{i=1}^mM_i,
\]
where $M_i = F_i\cdot F^n$. We similarly have $F_n = \bigoplus_{i=1}^mM_i^\ast$. With these decompositions, write $v=(v_1,\ldots,v_m)$ and $v^\ast=(v_1^\ast,\ldots,v_m^\ast)$. The assumption that $(x,v,v^\ast)$ is regular semi-simple implies that $v_i\neq0$ and $v_i^\ast\neq0$ for all $i$.

For any subset $\Sigma\subset S_1$, let $v_\Sigma$ denote the vector in $F^n = \bigoplus_{i=1}^mM_i$ where
\[
v_{\Sigma,i}=\begin{cases}v_i&:i\in \Sigma\\ 0&: i\in (S_1\setminus{\Sigma})\sqcup S_2\end{cases},
\]
and likewise for $v^\ast_\Sigma$.

 For any $f\in C_c^\infty(\fgl_n(F)\times F^n\times F_n)$ and $\Sigma\subset S_1$, define the \emph{generalized nilpotent orbital integral}
\begin{align*}
&\Orb^{\GL_n(F),\eta}(f,(x,v_\Sigma,v^\ast_{S_1\setminus{\Sigma}}))=\\&\int_{\GL_n(F)/T_x}\left(\int_{T_1}f(\Ad(g)x,gtv_\Sigma,v^\ast_{S_1\setminus{\Sigma}}t^{-1}g^{-1})\prod_{i\in\Sigma}|t_i|^s\prod_{i\in S_1\setminus{\Sigma}}|t_i|^{-s}\eta(\det t)\eta(\det g)dt\bigg|_{s=0}\right)dg,
\end{align*}
where $t_i\in F^\times_i$ and the integral is understood in terms of a natural meromorphic continuation. The point is that the subsets of $S_1$ are in canonical bijection with certain non-regular orbits in $\fgl_n(F)\times F^n\times F_n$, and these nilpotent orbital integrals are natural $(\GL_n(F),\eta)$-invariant distributions supported on these orbits.

On the other hand, for any $\zeta\in F[x]^\times$, there exists a unique space $\Herm(V_\zeta)\times V_\zeta$ containing a regular semi-simple orbit matching $(x,v,v^\ast \zeta)$; denote by $(\de_\zeta,w_\zeta)$ a representative of this orbit. Set $(\de,w):= (\de_1,w_1)$.
\begin{Lem}\cite[Lemma 4.6]{Xiao}
The map $\zeta\mapsto \de_\zeta$ defines a bijection of $H^1(F,T_{\de})$-torsors
\[
F[x]^\times/\Nm_{E[x]/F[x]}(E[x]^\times)\iso S\cong H^1(F,T_{\de}).
\]
\end{Lem}\qed

For any $(\de,w)\in \Herm(V)\times V$ matching $(x,v,v^\ast)$, we obtain a character 
\begin{align*}
\ka_\Sigma:H^1(F,T_\de)&\lra\cc^\times\\\de_\zeta\quad&\longmapsto (-1)^{\Sigma(\zeta)},
\end{align*}
where $\Sigma(\zeta)=\#\{i\notin\Sigma: \zeta_i\notin \Nm_{E_i/F_i}(E_i)\}$. It is evident that all characters $\ka\in H^1(F,T_\de)^\ast$ arise in this fashion for some $\Sigma\subset S_1$. This motivates the following germ expansion, relating these generalized nilpotent orbital integrals and $\ka$-orbital integrals on the twisted Lie algebra $\Herm(V)$.
\begin{Thm}\cite[Theorem 4.7]{Xiao}\label{Thm: Xiao's result}
Suppose that $f$ and $\{f_V\}_V$ are smooth transfers with respect to the Jacquet--Rallis transfer of (\ref{eqn: JR matching}) . Then for any regular semi-simple $(x,v,v^\ast)\in \fgl_n(F)\times F^n\times F_n$, we have the equality
\begin{equation*}
\omega(x,v,v^\ast)\Orb^{\GL_n(F),\eta}(f,(x,v_\Sigma,v^\ast_{S_1\setminus{\Sigma}})) = \sum_{\zeta\in S}\ka_\Sigma(\zeta)\int_{U(V_\zeta)/T_{\de_\zeta}}f_{V_\zeta}(\Ad(g)\de_{\zeta},0)dg,
\end{equation*}
where $\omega$ is the transfer factor in (\ref{eqn: transfer factor JR 1}) and $(\de_\zeta,w_\zeta)\in \Herm(V_\zeta)\times V_\zeta.$
\end{Thm}
Note that the right-hand side is essentially a $\kappa$-orbital integral of the function $f_V(-,0)\in C_c^\infty(\Herm(V))$. One caveat is that this sum is over $S$, rather than conjugacy classes in a single stable orbit of $\Herm(V)$.  In particular, to recover a $\ka$-orbital integral on $V$, we must apply the Jacquet--Langlands transfer to the function $f_{V'}$.

\subsection{The second reduction}\label{Section: first reduction}
Returning to the context of Theorem \ref{Thm: endoscopic base change}, we assume that $E/F$ is unramified and recall the diagram 
\begin{equation*}
\begin{tikzcd}
&\calh_{K_{n,E}}(\GL_n(E))\ar[dl,"-\ast \bfun_0"]\ar[dr,"BC"]&\\
\calh_{K_{n,E}}(X_n)\ar[rr,"H"]&&\calh_{K_n}(\GL_n(F)),
\end{tikzcd}
\end{equation*}
where $H$ denotes the isomorphism given by Hironaka. Fix the self-dual lattice $\Lam_n=\calo_E^n\subset V_n$ and the lattice $\call_n=\calo_F^n\times {\calo_F}_n\subset F^n\times F_n$. Let $\bfun_{\Lam_n}$ and $\bfun_{\call_n}$ be the indicator functions. Extension-by-zero  gives an embedding 
 \[
 \calh_{K_{n,E}}(X_n)\hra C_c^\infty(\Herm(V_n));
 \]
 composing this with tensor multiplication by $\bfun_{\Lam_n}$
 gives
\begin{align*}
 \calh_{K_{n,E}}(X_n)&\hra  C_c^\infty(\Herm(V_n)\times V_n)\\f&\mapsto f\otimes\bfun_{\Lam_n}.
\end{align*}
We similarly embed $\calh_{K_n}(\GL_n(F))$ in $C_c^\infty(\fgl_n(F)\times F^n\times F_n)$ using $\bfun_{\call_n}$. 
These latter two spaces are related by the Jacquet--Rallis transfer in the sense of (\ref{eqn: JR matching}).
\begin{Prop}\label{Prop: first reduction}
Suppose that for any $\varphi\in \calh_{K_{n,E}}(\GL_n(E))$, the functions 
\[\{(\varphi\ast \bfun_0)\otimes\bfun_{\Lam_n},0\}\text{ and }BC(\varphi)\otimes\bfun_{\call_n}\] are Jacquet--Rallis transfers in the sense of (\ref{eqn: JR matching}). Then Theorem \ref{Thm: endoscopic base change} follows.
\end{Prop}

\begin{proof}
 Fix an elliptic endoscopic datum $(U(V_a)\times U(V_b),s,\xi)$ for $X_n$, matching regular semi-simple elements $y\in X_n$ and $(y_a,y_b)\in X_a\times X_b$, and let $\ka: \cald(T_y/F)\to\cc^\times$ be the associated character.  We recall the construction of endoscopic transfer for the twisted Lie algebra $\Herm(V_n)$ from \cite{Xiao}. Consider the diagram
\[
\begin{tikzcd}
\Herm(V_n)\ar[d,dotted]&\Herm(V_n)\times V_n\ar[r,"JR"]\ar[l,"ev_0"]&\fgl_{n}(F)\times F^n\times F_n\ar[d,"PD"]\ar[l]\\
\Herm(V_a)\oplus\Herm(V_b)&\prod_{i=a,b}\Herm(V_i)\times V_i\ar[l,"ev_0"]\ar[r,"JR"]&\prod_{i=a,b}\fgl_{i}(F)\times F^i\times F_i\ar[l].
\end{tikzcd}
\]
Here, the arrows indicate relations between certain orbital integrals as follows:
\begin{itemize}
\item \underline{$ev_0$}: this arrow indicates the map $ev_0(F)(-) =F(-,0)$;
\item \underline{JR}: this arrow indicates the Jacquet--Rallis transfer;
\item\underline{PD}: this arrow indicates parabolic descent of relative orbital integrals.
\end{itemize}
Fixing $f\in C_c^\infty(\Herm(V_n))$, we will construct an endoscopic transfer $f_{a,b}$ of $f$. Choose $F\in C_c^\infty(\Herm(V_n)\times V_n)$ such that $ev_0(F)= f$ and let $\phi\in C_c^\infty(\fgl_n(F)\times F^n\times F_n)$ be a Jacquet--Rallis transfer of $\{F,0\}$. 

We now describe the parabolic descent that arises in the above diagram.
\begin{Def}
 Let $P_{(a,b)}= M_{(a,b)}N_{(a,b)}$ be the standard maximal parabolic subgroup of $\GL_n(F)$ with Levi factor $M_{(a,b)}\cong\GL_a\times \GL_b$, unipotent radical $N_{(a,b)}$, and set $\mathfrak{p}_{(a,b)}=\Lie(P_{(a,b)}(F))$. For a function $\phi\in C_c^\infty(\fgl_{n}(F)\times F^n\times F_n)$, denote by $\phi^\fp:=\phi^{\fp_{(a,b)}}$ the following \textbf{Lie-algebraic parabolic descent} of $\phi$ to $\prod_{i=a,b}\fgl_{i}(F)\times F^i\times F_i$:
\[
\phi^\fp((m_1,m_2),v, v^\ast) = \int_{\mathfrak{n}_{(a,b)}}\int_{K_n}\phi\left(k\left(\begin{array}{cc}m_1&n\\&m_2\end{array}\right)k^{-1}\right), kv, v^\ast k^{-1})dkdn,
\]
where $\mathfrak{n}_{(a,b)}=\Lie(N_{(a,b)})$ and $K_n = \GL_n(\calo_F)$.
\end{Def}

Returning to the argument, denote by $\phi^\fp$ the parabolic descent of $\phi$ to $\prod_{i=a,b}\fgl_{i}(F)\times F^i\times F_i$. We now use the Jacquet--Rallis transfer on both lower rank spaces to obtain four functions
\[
F_{\al,\be}^{a,b}\in C_c^\infty(\Herm(V_\al)\times V_\al\times \Herm(V_\be)\times V_\be),
\]
where $\al\in \calv_a$ and $\be\in\calv_b$. Set $$f^{a,b}_{\al,\be}=ev_{0}(F_{\al,\be}^{a,b})\in C_c^\infty(\Herm(V_\al)\times  \Herm(V_\be)).$$ Finally, if $\widetilde{f^{a,b}_{\al,\be}}$ denotes the Jacquet--Langlands transfer to $\Herm(V_a)\oplus \Herm(V_b)$ then define
\[
f_{a,b}=\sum_{\al,\be}(-1)^{k(\al,\be)}\widetilde{f^{a,b}_{\al,\be}},
\]
where $k(\al,\be)$ is the number of the forms $\{\al,\be\}$ that are split. Theorem 6.1 of \cite{Xiao} asserts that $f_{a,b}$ is an endoscopic transfer of $f$.

Now let $\varphi\in \calh_{K_{n,E}}(\GL_n(E))$. To prove the Proposition, we apply this approach to the  matching functions  $\{(\varphi\ast \bfun_0)\otimes\bfun_{\Lam_n},0\}$ and $BC(\varphi)\otimes\bfun_{\call_n}$. Set $F=\varphi\otimes\bfun_{\Lam_n}$ so that $ev_0(F)=\varphi$.

Fix now an auxiliary element $w\in V_n$ so that $(y,w)$ is regular semi-simple in $\Herm(V_n)\times V_n$ and let $(x,v,v^\ast)$ be a matching element in $\fgl_n(F)\times F^n\times F_n$. The assumption that $\{(\varphi\ast \bfun_0)\otimes\bfun_{\Lam_n},0\}$ and $BC(\varphi)\otimes\bfun_{\call_n}$ match and Theorem \ref{Thm: Xiao's result} implies that there exists a subset $\Sigma\subset S_1$ such that
\begin{equation*}
\omega(x,v,v^\ast)\Orb^{\GL_n(F),\eta}(BC(\varphi)\otimes\bfun_{\call_n},(x,v_\Sigma,v^\ast_{S_1\setminus{\Sigma}})) = \Orb^\ka(\varphi\ast \bfun_0,y).
\end{equation*}

 We similarly fix auxiliary vectors $w_a\in V_a$ and $w_b\in v_b$ such that $(y_a,w_a)$ is regular semi-simple in $\Herm(V_a)\times V_a$ and similarly for $(y_b,w_b)$. Using our assumption again, we know that the functions
\[
\{(\xi_{(a,b)}(\varphi)\ast \bfun_0)\otimes\bfun_{\Lam_a\times\Lam_b},0,0,0\} \text{    and     }BC(\xi_{(a,b)}(\varphi))\otimes\bfun_{\call_a\times\call_b}
\]
match with respect to Jacquet--Rallis transfer (\ref{eqn: JR matching}). In particular, we have no need to appeal to Jacquet--Langlands transfer in this case. 

We may assume that $\xi_{(a,b)}(\varphi) = \varphi_a\otimes \varphi_b$, so that $BC(\xi_{(a,b)}(\varphi)) = BC(\varphi_a)\otimes BC(\varphi_b)$. Applying Theorem  \ref{Thm: Xiao's result} for $V_a$, for any regular semi-simple $(x_a,v_a,v_a^\ast)$ matching $(y_a,w_a)\in X_a\times V_a$, we have
\begin{equation*}
\omega(x_a,v_a,v_a^\ast)\Orb^{\GL_a(F),\eta}(BC(\varphi_a)\otimes\bfun_{\call_a},(x_a,v_{a,\Sigma_a},v^\ast_{a,S_1(a)\setminus{\Sigma_a}})) = \SO(\varphi_a\ast \bfun_0,y_a),
\end{equation*}
 where the subset $\Sigma_a\subset S_1(a)$ arises from 
\[
\Sigma = (\Sigma_a,\Sigma_b)\subset S_1(a)\times S_1(b)=S,
\] with $S_1(a)$ and $S_1(b)$ as in (\ref{eqn: cohom decomp}). A similar identity holds for $V_b$.

Applying the argument outlined above, it remains to verify that
\begin{equation*}
BC(\xi_{(a,b)}(\varphi))\otimes\bfun_{\call_a\times\call_b} = (BC(\varphi)\otimes\bfun_{\call_n})^{\fp}.
\end{equation*}
Recalling the formula for $\xi_{(a,b)}$ in Lemma \ref{Lem: parabolic descent}, we check that
\[
BC(\xi_{(a,b)}(\varphi)) = \xi'_{(a,b)}(BC(\varphi)),
\]
where for any $f'\in C_c^\infty(\GL_n(F))$ and $(m_1,m_2)\in M_{a,b}(F)$
\[
\xi'_{(a,b)}(f')(m_1,m_2) = \mu_b'(\det(m_1))\mu_a'(\det(m_2))(f')^{P_{(a,b)}}(m_1,m_2),
\]
where $\mu'_s(t) = |t|_F^{s/2}$ for any $s\in \cc$ and the parabolic descent $(f')^{P_{(a,b)}}$ is defined as in Lemma \ref{Lem: parabolic descent} using the modular character of $P_{(a,b)}(F)$. Indeed, the left-hand side vanishes away from the $K_a\times K_b$-double cosets represented by elements of the form $(\varpi^{2\lam_a},\varpi^{2\lambda_b})$, where $\lam_a\in \mathrm{P}_a$ and $\lam_b\in \mathrm{P}_b$. A similar support constraint holds in the integral defining $BC(\varphi)^{P_{(a,b)}}$, and the difference in the characters $\mu_s$ and $\mu_s'$ ensures the equality.

Therefore, it suffices to show that for any $\phi\in \calh_{K_n}(\GL_n(F))$ 
\[
\xi'_{(a,b)}(\phi)\otimes\bfun_{\call_a\times\call_b}=(\phi\otimes\bfun_{\call_n})^{\fp}
\]
as functions on $M_{(a,b)}\times(F^a\oplus F^b)\times (F_a\oplus F_b)$. 

For any such $((m_1,m_2), v, v^\ast)$, it is clear that $(\phi\otimes\bfun_{\call_n})^{\fp}((m_1,m_2),v,v^\ast) =0$ unless $(v,v^\ast)\in \call_n\times \call_n^\ast =(\call_a\times\call_b)\oplus (\call_a^\ast\times \call_b^\ast).$ In particular, $(\phi\otimes\bfun_{\call_n})^{\fp}((m_1,m_2),v,v^\ast)$ equals $\bfun_{\call_a\times\call_b}(v,v^\ast)$ times
\begin{align*}
\int_{\mathfrak{n}_{(a,b)}}\int_{K_n}&\phi\left(k\left(\begin{array}{cc}m_1&n\\&m_2\end{array}\right)k^{-1}\right)dkdn\\
											&=\int_{\mathfrak{n}_{(a,b)}}\int_{K_n}\phi\left(k\left(\begin{array}{cc}m_1&\\&m_2\end{array}\right)\left(\begin{array}{cc}1&m_1^{-1}n\\&1\end{array}\right)k^{-1}\right)dkdn\\
                                                                     &=|\det(m_1)|_F^b\int_{\mathfrak{n}_{(a,b)}}\int_{K_n}\phi\left(k\left(\begin{array}{cc}m_1&\\&m_2\end{array}\right)\left(\begin{array}{cc}1&n\\&1\end{array}\right)k^{-1}\right)dkdn\\
											&=\mu'_b(\det(m_1))\mu'_a(\det(m_2))\phi^{P_{(a,b)}}(m_1,m_2),
\end{align*}
where we have used the formula
\[
\de_{P_{(a,b)}}(m_1,m_2) = |\det(m_1)|_F^b|\det(m_2)|_F^{-a}.
\]
By Lemma \ref{Lem: parabolic descent} and the explanation above, this last expression is precisely $\xi'_{(a,b)}(\phi)\otimes\bfun_{\call_a\times\call_b}$, completing the proof.\qedhere
\end{proof}

\section{The Weil representation and the third reduction}\label{Section: Weil rep}
We now wish to ``peel off'' the indicator functions $\bfun_{\Lam_n}$ and $\bfun_{\call_n}$ from the conjectured transfer for the Hecke algebra. This requires the full power of the Weil representation on the spaces $C_c^\infty(\Herm(W)\times W)$ and $C_c^\infty(\fgl_n(F)\times F^n\times F_n)$ studied in \cite{beuzart2019new}. We recall this representation now.
\subsection{The Weil representation}

Now fix an additive character $\psi:F\to \cc^\times$ of conductor $\calo_F$. Let $V$ be an $n$-dimensional Hermitian space. For an element $(x,v,v^\ast)\in \fgl_n(F)\times F^n\times F_n$, we set
\[
q(x,v,v^\ast) = v^\ast(v)\in F.
\] 
Similarly, for $(y,w)\in \Herm(V)\times V$ we set $q(y,w) = \la w,w\ra_V$.

Recall the partial Fourier transforms $\mathcal{F}$ on $C_c^\infty(\fgl_n(F)\times F^n\times F_n)$ and  $C_c^\infty(\Herm(V)\times V)$:\ for $f\in C_c^\infty(\fgl_n(F)\times F^n\times F_n)$, we set
\[
\mathcal{F}(f)(x,v,v^\ast) = \int_{F^n\times F_n}f(x,w,w^\ast) \psi( w^\ast(v)+v^\ast(w))dwdw^{\ast}.
\]
Similarly, for $f\in C_c^\infty(\Herm(V)\times V)$ we set
\[
\mathcal{F}(f)(y,w) = \int_{V}f(x,u) \psi(\Nm_{E/F}(\la u,w\ra))du.
\]

These transforms induce a Weil representation of $\SL_2(F)$ on these function spaces in the standard way. Indeed, since $\SL_2(F)$ is generated by the elements $\left(\begin{array}{cc}&1\\-1&\end{array}\right)$ and $\left(\begin{array}{cc}1&x\\&1\end{array}\right)$, we need only describe the action of these elements. For $\phi\in C_c^\infty(\fgl_n(F)\times F^n\times F_n)$, this action is given by
\[
W\left(\begin{array}{cc}1&t\\&1\end{array}\right)\phi(x,v,v^\ast) = \psi(tq(x,v,v^\ast))\phi(x,v,v^\ast),
\]
for any $t\in F$, and
\[
W\left(\begin{array}{cc}&1\\-1&\end{array}\right)\phi(x,v,v^\ast) = \mathcal{F}\phi(x,v,v^\ast).
\]
 The formulas are similar for the unitary case. 

An important property of this representation is that it descends to orbital integrals. More precisely, recall from Section \ref{Section: JR transfer} that $\cala$ denotes the categorical quotient $\fgl_n\times \mathbb{G}_a^n\times (\mathbb{G}_a^n)^\ast//\GL_n$. The image of the regular locus is an open sub-variety $\cala^{rss}\subset\cala$. 

We denote the canonical quotient maps by $$p_{\GL}:\fgl_n(F)\times F^n\times F_n\to \cala(F)$$ and $$p_V:\Herm(V)\times V\to \cala(F).$$ For any $a\in\cala^{rss}(F)$ and functions $f'\in C_c^\infty(\fgl_n(F)\times F^n\times F_n)$ and $f\in C_c^\infty(\Herm(V)\times V)$, set 
\[
O(a,f) = \begin{cases}
\Orb^{U(V)}(f,Y_a)&:p_V^{-1}(a)\neq \emptyset\text{ and }Y_a=(y,w)\in p_V^{-1}(a),\\\quad0&:\text{otherwise},\end{cases}
\]
and
\[
O(a,f')=\omega(X_a)\Orb^{\GL_n(F),\eta}(f',X_a)\text{ for any }X_a=(x,v,v^\ast)\in p_{\GL}^{-1}(a).
\]
With this notation, $f'$ and $f$ are transfers in the sense of (\ref{eqn: JR matching}) if and only if 
\[
O(a,f')=O(a,f)
\]
as functions on $\cala^{rss}(F)$. With this in mind, let
\[
\Orb(\fgl_n(F)\times F^n\times F_n) = \{a\mapsto O(a,f'): f'\in C_c^\infty(\fgl_n(F)\times F^n\times F_n)\}
\]
and let 
\[
\Orb(\Herm(V)\times V) = \{a\mapsto O(a,f): f\in C_c^\infty(\Herm(V)\times V)\}.
\]

There are natural Weil representations on $\Orb(\fgl_n(F)\times F^n\times F_n)$ and $\Orb(\Herm(V)\times V)$: as before, we need only describe the action of a unipotent element and the Weyl element. For any $t\in F$ and any $\Phi\in \Orb(\fgl_n(F)\times F^n\times F_n)$, set
\[
W\left(\begin{array}{cc}1&t\\&1\end{array}\right)\Phi(a) = \psi(tq(a))\Phi(a),
\]
where $q(a)=q(x,v,v^\ast)$ for any $(x,v,v^\ast)\in p_{\GL}^{-1}(a).$ Realizing $\Phi = O(-,f)$ for some $f\in C_c^\infty(\fgl_n(F)\times F^n\times F_n)$, then set
\[
W\left(\begin{array}{cc}&1\\-1&\end{array}\right)O(a,f) =O(a,\calf (f)). 
\]
The formulas for the unitary case are identical. 

The compatibility of Jacquet--Rallis transfer and Fourier transforms \cite[Theorem 4.17]{ZhangFourier} allows us to conclude the following result.
\begin{Prop}\cite[Proposition 1]{beuzart2019new}\label{Prop: Weil compatibility}
The Weil representations on 
\[
\text{$C_c^\infty(\fgl_n(F)\times F^n\times F_n)$ and $C_c^\infty(\Herm(V)\times V)$}
\]
 descend to the Weil representations on  
\[
\text{$\Orb(\fgl_n(F)\times F^n\times F_n)$ and $\Orb(\Herm(V)\times V)$.}
\] Moreover, these latter representations coincide on the intersection.
\end{Prop}

\subsection{The third reduction}\label{Section: second reduction statement}
We utilize these Weil representations to affect our final reduction. For this, we need to consider both forms of the Jacquet--Rallis transfer discussed in Section \ref{Section: JR theory}.

\begin{Prop}\label{Prop: second reduction}
Suppose that for $\varphi\in \calh_{K_{n,E}}(\GL_n(E))$, the functions
\[
\{\varphi\ast \bfun_0,0\}\text{  and  }BC(\varphi)
\]
are Jacquet--Rallis transfers in the sense of (\ref{eqn: regular matching final}) (applied to $\fgl_n$ rather than $\fgl_{n+1}$). Then the functions $\{(\varphi\ast \bfun_0)\otimes\bfun_{\Lam_n},0\}$ and $BC(\varphi)\otimes\bfun_{\call_n}$ are transfers in the sense of \eqref{eqn: JR matching}. 
\end{Prop}
\begin{proof}
The argument is similar to the proof of the Jacquet--Rallis fundamental lemma in \cite{beuzart2019new}. Fix $\varphi\in\calh_{K_{n,E}}(\GL_n(E))$ and consider
\begin{equation*}
    \Phi_\varphi(a):=O(a,(\varphi\ast \bfun_0)\otimes\bfun_{\Lam_n})-O(a,BC(\varphi)\otimes\bfun_{\call_n})
\end{equation*}
as a function on $\cala^{rss}(F)$. We claim that the assumption that $\{\varphi\ast \bfun_0,0\}\text{  and  }BC(\varphi)$ are transfers forces $\Phi_\varphi\equiv 0$; it is clear that this implies the proposition. Since $\Phi_\varphi$ is locally constant, it suffices to show $\Phi_\varphi(a)=0$ for $a$ in the open dense set where $q(a)\neq0,$ where $q(a)= q(x,v,v^\ast)$ for any $(x,v,v^\ast)\in p_{\GL}^{-1}(a).$ 

Note that it is immediate that $\Phi_\varphi(a)= 0$ if $|q(a)|_F>1$ as the indicator functions are supported away from such orbits. We now assume that $|q(a)|_F=1$. Supposing that $(x,v,v^\ast)\in p_{\GL}^{-1}(a)$ and $(y,w)\in p_{V_n}^{-1}(a)$, we see
\[
q(a)=q(x,v,v^\ast) = q(y,w)\in \calo_F^\times.
\]
Since $q(y,w)\in \Nm_{E/F}(E)$, there is an $\nu\in \calo_E^\times$ such that $q(a) = \Nm_{E/F}(\nu).$ Setting $e_n= {}^t[0,\ldots,0,1]\in \calo_F^n,$ we are free to conjugate $(x,v,v^\ast)\in p_{\GL}^{-1}(a)$ and $(y,w)\in p_{V_n}^{-1}(a)$ and assume that
\[
w= \nu e_n\text{ and } (v,v^\ast) = (\Nm_{E/F}(\nu)e_n,{}^te_n). 
\]
By the definition of $O(a,-)$, we have
 \[
 O(a,(\varphi\ast \bfun_0)\otimes\bfun_{\Lam_n}) = \int_{U(V_n)}(\varphi\ast \bfun_0)(\Ad(h^{-1})y)\bfun_{\Lam_n}(h^{-1}\nu e_n)dh.
 \]
 For $\bfun_{\Lam_n}(h^{-1}\nu e_n)\neq0,$, we must have $h^{-1}\nu e_n\in \calo_E^n$. Since the stabilizer of $e_n$ in $U(V_n)$ is $U(V_{n-1})$, it follows that this integral is supported on $U(\Lam_n)U(V_{n-1})$. Since the function $\varphi\ast \bfun_0\in \calh_{K_{n,E}}(X_n)$ is invariant under the action of $U(\Lam_n)$, our choice of Haar measure implies that
 \[
  O(a,(\varphi\ast \bfun_0)\otimes\bfun_{\Lam_n}) = \Orb^{U(V_{n-1})}((\varphi\ast \bfun_0),y).
 \]
 A similar argument shows that $|q(a)|_F=1$ implies that
 \[
 O(a,BC(\varphi)\otimes\bfun_{\call_n})=\omega(x)\Orb^{\GL_{n-1}(F),\eta}(BC(\varphi),x).
 \]
 Thus, our assumption implies that 
 \begin{equation}\label{Weil triv}
 \Phi_\varphi(a)=0\text{  whenever  }|q(a)|_F\geq 1.
 \end{equation}
 
 To complete the proof, we make use of the Weil representation. We first note since we assumed that $\psi$ is unramified, we have
 \[
 \calf(\phi\otimes\bfun_{\call_n}) = \phi\otimes\bfun_{\call_n},
 \]
 for any $\phi\in C_c^\infty(\fgl_n(F))$; a similar statement holds for any $\phi'\in C_c^\infty(\Herm(V_n))$. Considering the Weil representation on 
 \[
\Orb(\fgl_n(F)\times F^n\times F_n)\cap\Orb(\Herm(V_n)\times V_n),
 \] Proposition \ref{Prop: Weil compatibility}  now implies that 
 \[
 W\left(\begin{array}{cc}&1\\-1&\end{array}\right)\Phi_\varphi=\calf(\Phi_\varphi) = \Phi_\varphi.
 \]
 Moreover, (\ref{Weil triv}) implies that for any $t\in \mathfrak{p}_F^{-1}$, 
 \[
  W\left(\begin{array}{cc}1&t\\&1\end{array}\right)\Phi_\varphi=\psi(tq(a))\Phi_\varphi = \Phi_\varphi.
 \]
 Since $\SL_2(F)$ is generated by $ \left(\begin{array}{cc}&1\\-1&\end{array}\right)$, and$\left (\begin{array}{cc}1&t\\&1\end{array}\right)$ for $ t\in \mathfrak{p}_F^{-1}$, it follows that $$W(g)\Phi_\varphi=\Phi_\varphi$$ for all $g\in \SL_2(F)$.
 
 Now for any $a\in \cala^{rss}(F)$ with $q(a)\neq0,$ there exists a $t\in F$ such that $\psi(tq(a))\neq1$. But then
 \[
\psi(tq(a))\Phi_\varphi(a)=  W\left(\begin{array}{cc}1&t\\&1\end{array}\right)\Phi_\varphi(a) = \Phi_\varphi(a),
 \] showing that $\Phi_\varphi(a)=0.$ This proves the proposition.
\end{proof}

We now arrive at the final reduction of Theorem \ref{Thm: endoscopic base change}.
\begin{Thm}\label{Thm: JR fundamental lemma for algebra}
For any $\varphi\in \calh_{K_{n,E}}(\GL_{n}(E))$, and for any $X\in \GL_{n}(F)^{rss}$, we have
\begin{equation*}
\omega(X)\Orb^{\GL_{n-1}(F),\eta}(BC(\varphi),X)=\begin{cases}\Orb^{U(V_{n-1})}(\varphi\ast \bfun_0,Y)&:X\leftrightarrow Y\in X^{rss}_{n},\\\qquad0&:\text{otherwise}.\end{cases}
\end{equation*}
\end{Thm}
We prove this in Section \ref{Section: final touches} by spectral techniques. Note that combining Propositions \ref{Prop: first reduction} and \ref{Prop: second reduction} with this theorem completes the proof of Theorem \ref{Thm: endoscopic base change}. By Proposition \ref{Prop: initial reduction}, we conclude Theorem \ref{Thm: full fundamental lemma}.

\part{Spectral transfer and a comparison of relative trace formulas}\label{Part 2}

In this part, we prove Theorem \ref{Thm: JR fundamental lemma for algebra}. Our approach is a comparison of relative trace formulas we refer to as the twisted Jacquet--Rallis trace formula. This name indicates both a strong analogy with the Jacquet--Rallis case, as well as our dependence on the Jacquet--Rallis transfer and fundamental lemma for the Lie algebra in Theorems \ref{Thm: JR transfer} and \ref{Thm: JR fundamental lemma} to obtain the desired geometric comparison.

Let $E/F$ denote a quadratic extension of number fields. Heuristically, the comparison of Jacquet--Rallis may be stated in terms of the matching of orbits
\[
\GL_n(E)\backslash \GL_n(E)\times \GL_{n+1}(E)/ \GL_n(F)\times \GL_{n+1}(F)
\]
with
\[
\bigsqcup_{V\in \calv_n}U(V)\backslash U(V)\times U(V\oplus Ee_0)/U(V),
\]
where $\calv_n$ runs over a set of representatives of the isomorphism classes of $n$
dimensional Hermitian spaces, and the Hermitian form on $V\oplus Ee_0$ is determined by that of $V$.

The first observation is that the matching of orbital integrals in Theorem \ref{Thm: JR fundamental lemma for algebra} may be studied globally by \emph{switching the roles} of the rational linear group and the unitary group in the Jacquet--Rallis case. This leads to a matching of orbits
\[
\bigsqcup_{V\in \calv_n}\bigsqcup_{W\in \calv_{n+1}}\GL_n(E)\backslash \GL_n(E)\times \GL_{n+1}(E)/  U(V)\times U(W),
\]
where we impose no assumptions on $V$ and $W$, with the orbits
\[
\GL_n(F)\backslash \GL_n(F)\times \GL_{n+1}(F)/\GL_n(F).
\]
This matching of orbits suggests a comparison of relative trace formulas, the geometric side of which may be calibrated to study the comparison in Theorem \ref{Thm: JR fundamental lemma for algebra}; see Section \ref{Section: final section}. This allows us to translate the problem into one of \emph{spectral transfer of relative characters}, a classical method for proving fundamental lemmas for Hecke algebras in the context of the Arthur-Selberg trace formula; see, for example, \cite{Clozelbasechange}, \cite{halesfundamental}, and \cite{Luometaplectic}. In our context, the main result is Theorem \ref{Thm: base change Twisted Jacquet--Rallis} which gives a matching of orbital integrals corresponding to the identity of relative characters given in Lemma \ref{Lem: unramified spectral transfer}.

To access the local relative characters, we observe that the spectral results of Feigon, Lapid, and Offen \cite{FLO} and Jacquet \cite{JacquetQuasisplit} on unitary periods of cuspidal automorphic forms are precisely what we need to make the spectral comparison manageable. We review the necessary results in Section \ref{Section: unitary periods}. In particular, we have the factorization (\ref{eqn: period product}). This crucial input enables access to local relative characters by the relative trace formulas. In fact, the spectral results of \cite{FLO} are so complete that our comparison does not appear to reveal any new information about unitary periods. On the other hand, the global theory of our comparison does not rely on any previous work on these periods. On the other hand, sufficiently refined results about non-vanishing of central values of Rankin-Selberg $L$-functions of the form of \cite{li2009central} would enable our comparison to give a new proof of several of the main results of these works.

In the next section, we establish our notational conventions for this part, highlighting important changes from the notation in Part \ref{Part 1}. Section \ref{Section: part 2 orbits} covers the local geometric comparison of orbital integrals, proving existence of smooth transfer and the fundamental lemma for the unit element by reducing our comparison to the Lie algebra version of Jacquet--Rallis transfer as in Section \ref{Section: JR theory}. Special care is needed when incorporating the action of the center, reflecting the fact that the norm map
\[
\Nm:\A_E^\times \to \A_F^\times
\]
is not surjective. We discuss the necessary details in Section \ref{Section: dealing with center}

After this, Section \ref{Section: factorizations} reviews the global and local theory of the invariant distributions we use to build the relative characters. The comparison of 
 relative trace formulas occurs in Section \ref{Section: comparison}. The main result of these sections is the transfer of global relative characters in Theorem \ref{Prop: main global tool}. We remark that the concerns with the central character mentioned above manifests here in the statement of Proposition \ref{Prop: simple character id}. We then prove a weak transfer of local relative characters in Section \ref{Section: weak comparison}.

Finally, we use these results to prove the fundamental lemma for this comparison in Section \ref{Section: final section}; this is Theorem \ref{Thm: base change Twisted Jacquet--Rallis}. The point is to reduce the local equality of orbital integrals to a statement about transfer of global relative characters by first globalizing the orbital integrals and then using the comparison of relative trace formulas. The results of Sections \ref{Section: factorizations} and \ref{Section: weak comparison} then reduce this problem to a local spectral identity at a single finite place, which we verify directly. Theorem \ref{Thm: JR fundamental lemma for algebra} is then readily deduced from Theorem \ref{Thm: base change Twisted Jacquet--Rallis} in Section \ref{Section: final touches}, completing the proof of Theorem \ref{Thm: full fundamental lemma}.

\begin{Rem}
For the reader who is inclined to believe that most of the analytic properties of the Jacquet--Rallis relative trace formula comparison are enjoyed by our set up, we recommend skipping Section \ref{Section: comparison} except for the statement of Theorem \ref{Prop: main global tool} as it mirrors \cite[Section 2]{ZhangFourier} closely. Some additional care is needed to isolate the comparison for a single pair of Hermitian spaces, but this is not difficult.
\end{Rem}
\begin{Rem}
In the final application, we work with globally quasi-split unitary group to prove Theorem \ref{Thm: JR fundamental lemma for algebra}. Despite this, we develop the comparison in general as restricting to the quasi-split case does not simplify the arguments, and in some instances would overly complicate the notation. The general comparison may also be of independent interest.
\end{Rem}

\section{Preliminaries}In this section, we fix our conventions regarding groups, Hermitian spaces, and measures. In order for this part to be self-contained, we only continue to hold to those conventions established in Section \ref{Section: Prelim} and do not refer to Part \ref{Part 1} for notation. This allows for additional flexibility, despite a good amount of notation being consistent across both parts. For example, we scrub our notations for orbital integrals, transfer factors, etc. unless making explicit reference to a formula.

\subsection{Involutions}

For a field $F$, recall the element
\[
w_n:=\left(\begin{array}{cccc}&&&1\\&&-1&\\&\Ddots&&\\(-1)^{n-1}&&&\end{array}\right)\in\GL_n(F).
\]
For any $F$-algebra $R$ and $g\in \GL_n(R)$, we define 
\[
g^\theta = w_n{}^tg^{-1}w_n.
\]
Now suppose that $E/F$ is a quadratic \'{e}tale algebra and consider the restriction of scalars $\Res_{E/F}(\GL_n)$. Then for any $F$-algebra $R$ and $g\in \Res_{E/F}(\GL_n)(R)$, we set \[
g^\sig=\overline{g}
\]
to be the Galois involution associated to the extension $E/F$. Note that $w_n\in X_n$ when $n$ is odd and if $\xi\in E$ is a trace-zero element, then $\xi w_n\in X_n$ when $n$ is even.\\

\noindent\underline{\textbf{Important notational difference:}} In this part, we set 
\[
V_n=\begin{cases}V_{w_n}&: n \text{ odd},\\
                 V_{\xi w_n}&: n \text{ even},\end{cases}
\]and work with the  quasi-split unitary group $\Res_{E/F}\GL_n(E)^{\theta\circ\sig}= \U(V_n)$. We make this choice as it will be convenient to have a form that is split both globally and locally.\\

\subsection{Groups and Hermitian spaces}
Let $F$ be a field and fix $E/F$ a quadratic \'{e}tale algebra. Let $F^n$ be a fixed $n$ dimensional vector space, $F^{n+1}=F^n\oplus Fe_0$ with a fixed vector $e_0$. This gives rise to an embedding of  $\GL_n$ as the subgroup of $\GL_{n+1}$ preserving this decomposition:
\[
g\mapsto \left(\begin{array}{cc}g&\\&1\end{array}\right).
\]
With this, set $G=\GL_n\times \GL_{n+1}$ and $H\cong \GL_n\subset G$, where $H$ is embedded diagonally:
\[
g\mapsto \left(g,\left(\begin{array}{cc}g&\\&1\end{array}\right)\right).
\]

Now consider the product
\[
X_n\times X_{n+1}
\] parameterizing pairs of Hermitian vector spaces of dimension $n$ and $n+1$. A point $(x,y)\in X_n\times X_{n+1}$ determines the unitary groups
\[
\U(V_x)\hra\Res_{E/F}(\GL_n)
\]
and
 \[
 \U(V_y)\hra\Res_{E/F}(\GL_{n+1}).
 \]
 We set $G'=\Res_{E/F}(\GL_n)\times \Res_{E/F}(\GL_{n+1})$ and $H'\cong \Res_{E/F}(\GL_n)$ embedded diagonally as above. For any $(x,y)\in X_n\times X_{n+1}$, set $H'_{x,y}=\U(V_x)\times \U(V_y).$ Note that 
 \[
 H'_{w_n,w_{n+1}}=\U(V_n)\times \U(V_{n+1})
 \]
 is a product of quasi-split unitary groups.

\subsection{Representations and Whittaker models}
Suppose that $F$ is a local field and let $\calg_n(F)$ be the set of equivalence classes of generic representations of $\GL_n(F)$. For any non-trivial additive character $\psi: F\to \cc^\times,$
we denote by $\psi_0$ the generic character of $N_n(F)$
\[
\psi_0(u)=\psi\left(\sum_iu_{i,i+1}\right).
\]

For a quadratic \'{e}tale algebra $E/F$, we let $\psi'=\psi\circ\Tr_{E/F}$ denote the induced additive character and $\psi_0'$ the corresponding generic character of $N_n(E)$. Let $\calg_n^\sig(E)$ denote the set of equivalence classes of admissible generic representations of $\GL_n(E)$ that are isomorphic to their Galois twists. Such representations arise as the base change of a representation $\pi\in \calg_n(F)$ on $\GL_n(F)$ by \cite{ArthurClozel}; we write $\Pi=BC(\pi)$ to denote this relationship. It follows from \cite[Theorem 0.2.1]{FLO} that $\Pi$  has non-trivial invariant $U(V_x)$-invariant functionals for any $x\in X_n$.

For any $\pi\in \calg_n(F)$ we denote by $\pi^\vee$ the abstract contragredient representation. 
Set $\calw(\pi):=\calw^\psi(\pi)$ to be the Whittaker model of $\pi$ with respect to the generic character $\psi_0$. The action is given by
\[
\calw(g,\pi)W(h) = W(hg),\quad g,h\in \GL_n(F),\: W\in \calw(\pi).
\]
Then we obtain an isomorphism
\[
\hat{(\cdot)}:\calw(\pi)^\theta\lra \calw^{\psi^{-1}}(\pi^\vee),
\]
given by $\hat{W}(g) = W(g^\theta).$

\section{Orbital integrals and transfer}\label{Section: part 2 orbits}
We begin by describing the regular semi-simple orbits and the matching of orbits between our two models. We then describe the local orbital integrals and describe the necessary transfer of test functions and fundamental lemma needed for our global applications. Care is needed when taking the action of the center into account; see Section \ref{Section: dealing with center}.

\subsection{Matching and transfer}
Let $F$ be a field and let $E/F$ be a quadratic \'{e}tale algebra over $F$. 
\subsubsection{Linear side} Recall $G=\GL_n\times \GL_{n+1}$ and $H=\GL_n$ regarded as a subgroup of $G$ via the diagonal embedding. We define the regular semi-simple locus $
 G(F)^{rss}
$ to be the set of points $\ga$ such that the double coset $H(F)\ga H(F)\subset G(F)$ is closed and of maximal possible dimension. 
\begin{Lem}\label{Lem: orbit linear reduction}
Let $\GL_{n+1}(F)^{rss}$ denote the locus of elements $g$ such that, under the adjoint action of $\GL_{n}(F)$, the orbit of $g$ is closed and of maximal dimension. Then there is a natural bijection
\[
H(F)\backslash G(F)^{rss}/H(F)\iso \GL_{n+1}(F)^{rss}/\GL_n(F).
\]
\end{Lem}
\begin{proof}
This follows from considering the natural map
\begin{align*}
  \GL_n(F)\backslash \GL_n(F)\times \GL_{n+1}(F)&\lra \GL_{n+1}(F)\\ \GL_n(F)(h,g)\qquad\quad&\longmapsto\quad h^{-1}g.\qedhere
\end{align*}
\end{proof}
\subsubsection{Twisted side} Recall that $\calv_n$ denotes our set of $\GL_n(E)$-orbit representatives for $X_n$. Thus, 
\[
\{V_x:x\in \calv_n\}
\] is a fixed set of representatives for the isomorphism classes of $n$-dimensional Hermitian spaces over $E$. In this part, we always require that $w_n\in \calv_n$ in keeping with our choice of split Hermitian space. Denoting the $GL_{n}(E)$-orbit of $x\in \calv_n$ by $X_n^x,$ there is a decomposition
\begin{equation}\label{eqn: GL orbits}
X_n = \bigsqcup_{x\in\calv_n}X_n^x.
\end{equation}

For any $x\in X_n$, set
\[
y(x)=\left(\begin{array}{cc}x&\\&1\end{array}\right)\in X_{n+1}.
\]
Then $V_{y(x)}=V_x\oplus Ee_0$ where the sum is orthogonal and $\la e_0,e_0\ra_{y(x)}=1$. With this construction, there is a natural embedding of unitary groups
\[
U(V_x)\hra U\left(V_{y(x)}\right).
\]
Note that if when $F$ is $p$-adic and when $E/F$ is unramified, then $V_{y(x)}$ is split if $V_x$ is split, albeit with a Hermitian form conjugate to $w_{n+1}$.

For any $y\in \calv_{n+1}$, denote by
\begin{equation}\label{eqn: translate}
    X_y\subset \Herm(V_y)
\end{equation}
 the set of invertible elements in the twisted Lie algebra $\Herm(V_y).$ Note that any $x\in X_y$ is a product $x=x'\cdot y$ where $x'\in X_{n+1}$.

For any pair $(x,y)\in X_n\times X_{n+1}$, consider the subgroups $H'_{x,y}=\U(V_x)\times \U(V_y)\subset G'=\Res_{E/F}(\GL_n)\times \Res_{E/F}(\GL_{n+1})$ and $H'=\Res_{E/F}\GL_n$ embedded diagonally. Set $G'(F)^{rss}$ to be the set of points $\de$ such that the double coset $H'(F)\de H'_{x,y}(F)\subset G'(F)$ is closed and of maximal possible dimension.

We have a similar reduction of the regular orbits in this case. 
\begin{Lem}\label{Lem: orbit twisted reduction}
For any $x\in X_n$, define $X_{y(x)}^{rss}$ to be the set of elements $z$ such that, under the adjoint action of $U(V_x)$, the orbit of $z$ is closed and of maximal dimension. Then there is a natural bijection
\[
\bigsqcup_{y\in\calv_{n+1}}H'(F)\backslash G'(F)^{rss}/H'_{x,y}(F)\iso X_{{y(x)}}^{rss}/U(V_x).
\]
\end{Lem}
\begin{proof}
In view of the decomposition (\ref{eqn: GL orbits}), this follows by considering the map
\begin{align*}
\bigsqcup_{y\in\calv_{n+1}}H'(F)\backslash G'(F)/\{1\}\times U(V_y)&\lra X_{{y(x)}}\\ H'(F)(g_1,g_2)&\longmapsto (g_1^{-1}g_2)y{}^t(\overline{g_1^{-1}g_2})y(x).\qedhere
\end{align*}
\end{proof}

\begin{Prop}\label{Prop: matching of RSS loci}
There is a natural matching of regular semi-simple orbits, giving a bijection:
\[
H(F)\backslash G(F)^{rss}/H(F)\iso \bigsqcup_{x\in\calv_n}\bigsqcup_{y\in \calv_{n+1}}H'(F)\backslash G'(F)^{rss}/H'_{x,y}(F).
\]
\end{Prop}

\begin{proof}
By Lemmas \ref{Lem: orbit linear reduction} and \ref{Lem: orbit twisted reduction}, the claim reduces to the claim that there is a natural matching
\[
\GL_{n+1}(F)^{rss}/\GL_n(F)\iso\bigsqcup_{x}X_{{y(x)}}^{rss}/U(V_x).
\]
This is precisely the setting of the Lie algebra version of Jacquet--Rallis matching of orbits described in Section \ref{Section: JR theory}. We need only check that this matching respects restriction to the invertible elements of both sides. This may be checked directly via the explicit invariant polynomials reviewed below, but is more readily seen from noting that regular semi-simple elements $g\in \fgl_{n+1}(F)^{rss}$ and $x\in \Herm(V_{y(x)})^{rss}$ match if and only if, viewed naturally as elements of $\fgl_{n+1}(E)$, they are conjugate by $\GL_n(E)$. 
\end{proof}
We say that two regular semi-simple elements $\ga$ and $\de$ \emph{match} with respect to $(x,y)$ and write $\ga\xleftrightarrow{x,y} \de$ if the orbits 
\[
[\ga]\in H(F)\backslash G(F)^{rss}/H(F) \text{ and }[\de]\in H'(F)\backslash G'(F)^{rss}/H'_{x,y}(F)
\]
match in the sense of the preceding proposition.

\subsubsection{Invariant polynomials}\label{Section: invariant polys for quot} We recall the invariant polynomials used in \cite{ZhangFourier} as this will aid certain arguments in Section \ref{Section: final section}. Let 
\[
X=\left(\begin{array}{cc}A&b\\c&d\end{array}\right)\in \fgl_{n+1}(F), \text{ where }\:A\in \fgl_n(F),\:b\in F^n,\:c\in F_n,\text{ and }d\in F.
\]
Then we define the invariant map $\pi:\fgl_{n+1}(F)\to \A(F)^{2n+1}$ by 
\begin{align}\label{eqn: invariant map JR}
c(X) = (a_1(X),\ldots, a_n(X), b_0(X),\ldots, b_{n-1}(X), d)=(c_i(X))_{i=1}^{2n+1}
\end{align}
where 
\begin{equation}\label{invariant polynomials}
a_i(X) = \mathrm{Tr}(\wedge^i A),\text{   and    }b_j(X) = c\cdot A^j\cdot b.
\end{equation}
These polynomials are similarly defined for $Y\in  \Herm(V_{y(x)})$ and two regular semi-simple elements $X\in\fgl_{n+1}(F)^{rss}$ and $Y\in \Herm(V_{y(x)})^{rss}$ match if and only if they have the same invariants \cite{rallis2007multiplicity}. 

By a slight abuse of notation, we define the invariant polynomials $c_i:G(F)\to F$ for $i=1,\ldots,{2n+1}$ by setting 
\[
c_i(\ga) := c_i(\ga_1^{-1}\ga_2),
\]
where $\ga=(\ga_1,\ga_2)\in G(F).$ Similarly, for any pair $(x,y)\in X_n\times X_{n+1}$, we define the invariant polynomials $c^{x,y}_i:G'(F)\to F$ for $i=1,\ldots,{2n+1}$ by setting 
\[
c^{x,y}_i(\de) := c_i(\pi_{x,y}(\de))
\]
where $\de=(\de_1,\de_2)\in G'(F)$ and 
\[
\pi_{x,y}(\de) :=({\de_1}^{-1}\de_2)y{}^t\overline{({\de_1}^{-1}\de_2)}y(x)\in X_{y(x)}^y.
\]
Here, $X_{y(x)}^y$ denotes the $y(x)$-translate of the $\GL_n(E)$-orbit of $y$, combining the notation \eqref{eqn: GL orbits} with \eqref{eqn: translate}.

\subsection{Orbital integrals}
Assume now that $F$ is a local field, and let $E/F$ be a quadratic \'{e}tale algebra.
\subsubsection{Linear Side} Let $f\in C_c^\infty(G(F))$. We define the relative orbital integrals of interest 
\begin{equation*}
   \Orb^\eta(f,\ga):=\int_{H(F)}\int_{H(F)}f(h_1^{-1}(\ga_1,\ga_2) h_2)\eta(\det(h_2))dh_1dh_2, 
\end{equation*}
where $\ga=(\ga_1,\ga_2)\in G(F)^{rss}$ is a regular semi-simple element. This assumption implies that the centralizer of $\ga$ is trivial and that the orbit of $\ga$ is closed, so the integral is well defined. Consider the function $\tilde{f}\in C_c^\infty(\GL_{n+1}(F))$ defined as
\begin{equation}\label{eqn: tilde contraction}
 \tilde{f}(g) := \int_{H(F)}f(h^{-1}(1,g))dh.   
\end{equation}
Then the map
\begin{align*}
C_c^\infty(G(F))&\lra C_c^\infty(\GL_{n+1}(F))\\ f&\longmapsto \tilde{f}
\end{align*}
is surjective. Since the integrals are absolutely convergent, a simple rearrangement gives 
\begin{equation}\label{eqn: linear OI reduction}
\Orb^\eta(f,\ga) = \Orb^{\GL_n(F),\eta}(\tilde{f},\ga_1^{-1}\ga_2):=\int_{\GL_n(F)}\tilde{f}(h^{-1}\ga_1^{-1}\ga_2h)\eta(\det(h))dh.
\end{equation}
This orbital integral is of the type arising on the linear side of the Jacquet--Rallis transfer in the sense of (\ref{eqn: regular matching final}).

The transfer factor in this case is built out of the transfer factor (\ref{eqn: JR transfer factor 2}) for the Lie algebra version of Jacquet--Rallis transfer. For an element $X\in \fgl_{n+1}(F)^{rss}$, set 
\[
\omega(X) = \eta\left(\det([e_{n+1}|Xe_{n+1}|\ldots|X^{n}e_{n+1}])\right),
\]
where 
\[
e_{n+1}= {}^t[0,\ldots,0,1]\in F^{n+1}.
\]
\begin{Def}\label{Def: newish transfer factor}
 We define the{ transfer factor} $\Omega: G(F)^{rss}\to \cc$ by
\begin{equation*}
\Omega(\ga_1,\ga_2) := \omega(\ga_1^{-1}\ga_2).
\end{equation*}
\end{Def}

\subsubsection{Twisted side} 
For any pair $(x,y)\in X_n\times X_{n+1}$, we define the orbital integral
\begin{equation*}
\Orb(f',\de) := \int_{H'(F)}\int_{H'_{x,y}(F)}f'(h_1^{-1}(\de_1,\de_2) h_2)dh_1dh_2,
\end{equation*}
where $f'\in C_c^\infty(G'(F))$ and $\de=(\de_1,\de_2)\in G'(F)^{rss}$. Similarly to the previous case, we first define $\tilde{f'}:X^y_{y(x)}:\to \cc$ by
\[
\tilde{f'}(gy{}^t\overline{g}y(x))=\int_{H'(F)}\int_{U(V_y)}f'(h^{-1}(1,gu))dhdu.
\]

We see that
\begin{equation}\label{eqn: twisted OI reduction}
\Orb(f',\de) =\Orb^{U(V_x)}(\tilde{f'}, \pi_{x,y}(\de)):= \int_{U(V_x)}\tilde{f'}(h^{-1}\pi_{x,y}(\de)h)dh.
\end{equation}
This orbital integral is of the type arising  on the unitary side of the Jacquet--Rallis transfer in the sense of (\ref{eqn: regular matching final}).
\subsubsection{Taking care of the center}\label{Section: dealing with center}
Let $Z_{G}\subset G$ (respectively, $Z_{G'}\subset G'$) denote the center of $G$ (resp. $G'$). For reasons of convergence, we need to take the actions of the centers into account. This is more subtle than in \cite{ZhangFourier}, as the natural norm map
\[
\Nm:=\Nm_{E/F}: Z_{G'}\lra Z_G
\]
is not surjective on points globally and locally. 

We now assume that $E/F$ is a quadratic extension of either global or local fields of characteristic zero. Fixing $(x,y)\in X_n\times X_{n+1}$, consider the $Z_{G'}(F)H'(F)\times H'_{x,y}(F)$-action on $G'(F)$. 

Following the reductions above, it suffices to consider the $Z_{G'}(F)\times U(V_x)$- action on $X_{y(x)}.$ Here $U(V_x)$ acts via conjugation, while the center acts by
\begin{equation}\label{eqn: unitary center twist}
 (z_1,z_2)\circ s = (z_1^{-1}z_2)s\overline{(z_1^{-1}z_2)}.   
\end{equation}
Set
\begin{align}\label{uniform stab}
\mathcal{Z}_0&=\{[(z_1,z_2),z_1]\in Z_{G'}(F)\times U(V_x): z_1\in Z_{U(V_x)}(F),\:z_2\in Z_{U(V_y)}(F)\}\nonumber\\&\cong Z_{U(V_x)}(F)\times Z_{U(V_y)}(F).
\end{align}
It is simple to check that $\mathcal{Z}_0$ acts trivially under the above action. For any $s\in X_{y(x)}$, we may write
\[
s=\left(\begin{array}{cc}A&b\\\la b,-\ra_x&d\end{array}\right),
\]
where $A\in \Herm(V_x),$ $b\in E^n$, and $d\in F.$ 
A simple calculation shows that the invariants 
\[
\Tr(A) \text{  and  }d
\]
are scaled by a non-zero norm class under the $Z_{G'}(F)\times U(V_x)$-action, so that their norm classes are invariant. In analogy to \cite{ZhangFourier}, we call $s\in X_{y(x)}$ $Z$-regular semi-simple if it is regular semi-simple in $X_{y(x)}$ and if
\[
\Tr(A),d\in F^\times.
\]
This gives a Zariski-open, dense subset of $X_{y(x)}.$
\begin{Lem}\label{Lem: Z-reg on unitary side}
If $s$ is $Z$-regular semi-simple, then its centralizer under the $Z_{G'}(F)\times U(V_x)$-action is $\mathcal{Z}_0$ and its orbit is closed. In particular, a $Z$-regular semi-simple element is $Z_{G'}\times \U(V_x)$-regular semi-simple.
\end{Lem}
\begin{proof}
If $(z,h)\circ s=s,$ then since $\Tr(A)$ and $d$ are invertible, we may augment $(z,h)$ by an element of $\mathcal{Z}_0$ to assume that $z=(1,1)$. But now $h$ lies in the centralizer of $s$ under the adjoint action of $U(V_x).$ This is trivial since $s$ is regular semi-simple, proving the first claim.

We now note that when $\Tr(A),d\in F^\times,$ the rational functions 
\begin{equation}\label{Z-reg invariants}
\frac{\Tr(\wedge^iA)}{\Tr(A)^i}\qquad\text{    and    }\qquad\frac{\la b,A^jb\ra_x}{\Tr(A)^{j+1}d}
\end{equation}
 for $1\leq i\leq n$ and $0\leq j\leq n-1$ are invariant under $Z_{G'}(F)\times U(V_x)$. We claim that two $Z$-regular semi-simple elements $s_1$ and $s_2$ are in the same $Z_{G'}(F)\times U(V_x)$-orbit if and only if they have the same values under the invariants (\ref{Z-reg invariants}) and
 \[
 \Tr(A_1)\equiv \Tr(A_2)\mod{\Nm(E^\times)} \text{ and } d_1\equiv d_2\mod{\Nm(E^\times)},
 \]
 where $A_i$ and $d_i$ are as above. Indeed, sufficiency is immediate. To prove necessity, suppose that they have the same invariants and norm classes. By augmenting $s_2$ to $z\circ s_2$ for an appropriate central element $z\in Z_{G'}(F),$ we may assume that $\Tr(A_1)=\Tr(A_2)$ and $d_1=d_2.$
 
 Considering the invariants above, this implies that  $\Tr(\wedge^iA_1)=\Tr(\wedge^iA_2)$ for each $i$ and $\la b_1,A_1^jb_1\ra_x=\la b_2,A_2^jb_2\ra_x$ for all $j.$ These are precisely the invariants noted in (\ref{invariant polynomials}), so it follows from our assumption that $s_1$ and $s_2$ are regular semi-simple that they lie in the same $U(V_x)$-orbit. As in \cite[Lemma 2.1]{ZhangFourier}, this implies that the $Z_{G'}(F)\times U(V_x)$-orbit of $s_1$ is closed.
\end{proof}
We say that $\de\in G'(F)$ is $Z$-regular semi-simple if $\pi_{x,y}(\de)$ is.

For the linear case, if $Z_G(F)\subset G(F)$ is the center, we similarly consider the action of $Z_G(F)H(F)\times H(F)$ on $G(F)$. As before this reduces to considering the $Z_G(F)\times \GL_n(F)$-action on $\GL_{n+1}(F),$ where $(z_1,z_2)\in Z_G(F)$ acts on $g\in \GL_{n+1}(F)$ by
\begin{equation}\label{eqn: linear center twist}
(z_1,z_2)\circ g =\left(\begin{array}{cc}z^{-1}_1&\\&1\end{array}\right) gz_2.
\end{equation} We say that $g\in \GL_{n+1}(F)$ is $Z$-regular semi-simple if it is regular semi-simple under the $\GL_n(F)$-action and $\Tr(A),d\neq0$ where
\[
g=\left(\begin{array}{cc}A&b\\c&d\end{array}\right)\quad\text{where}\quad A\in\fgl_{n}(F),\:b,{}^tc\in F^n,\:\text{and}\: d\in F.
\]
A similar but easier argument now shows that a $Z$-regular semi-simple element of $\GL_{n+1}(F)$ has trivial centralizer under $Z_G(F)\times \GL_n(F)$ and has a closed orbit. We say $\ga=(\ga_1,\ga_2)$ is $Z$-regular semi-simple if $\ga_1^{-1}\ga_2$ is. 

\begin{Rem}\label{Rem: oof the center}
   Recall the matching of regular semi-simple orbits from Proposition \ref{Prop: matching of RSS loci}. This restricts to a bijection between $Z$-regular semi-simple loci as the non-vanishing assumptions respect the matching of invariant polynomials. However, the proof of Lemma \ref{Lem: Z-reg on unitary side} shows that when $E/F$ is a quadratic field extension, there is only a surjective map
  \[
\bigsqcup_{x\in\calv_n}\bigsqcup_{y\in \calv_{n+1}}H'(F)\backslash G'(F)^{Z-rss}/H'_{x,y}(F)Z_{G'}(F)\lra H(F)\backslash G(F)^{Z-rss}/H(F)Z_G(F).
\]
When $F$ is local and $E/F$ is a field extension, this map is $2$-to-$1$. Indeed, the $Z_{G}(F)$-action on $\GL_{n+1}(F)$ scales the invariants  $\Tr(A)$ and $d$ by any value of $F^\times$, while the $Z_{G'}(F)$-action on $X_{y(x)}$ preserves the norm class of these invariants. Our remedy is to only scale by the Zariski-open subgroup
\[
\Nm(Z_{G'}(F))\subset Z_G(F).
\]
This is not a geometric notion, but is well-defined on $A$-points for any $F$-algebra $A$. Indeed, the following lemma follows immediately from the definitions \eqref{eqn: unitary center twist} and \eqref{eqn: linear center twist}. 
\end{Rem}
\begin{Lem}\label{matching via center}
 If $\de\in G'(F)$ and $\ga\in G(F)$ are $Z$-regular semi-simple elements that match in the sense of Proposition \ref{Prop: matching of RSS loci}, and $z\in Z_{G'}(F)$, then $\de z$ matches $\ga \Nm(z)$.
\end{Lem}

Now assume $E/F$ is local, and assume that $\de\in G'(F)$ is $Z$-regular semi-simple. For any central character $\omega':Z_{G'}(F)\to \cc^\times$, we note that the integral
\[
\int_{Z_{G'}(F)}\Orb(f',z\de)\omega'(z)dz
\]
is absolutely convergent by the closed orbit assertion of the lemma and vanishes unless $\omega'$ is trivial on $Z_{H_{x,y}'}(F)$. In this case, $\omega'=\omega\circ \Nm$ for some character $\omega: Z_{G}(F)\to \cc^\times$, and we set
\begin{align}\label{eqn: unitary OI final}
 \Orb_{\omega'}(f',\de)&:=\int_{H'(F)}\int_{Z_{H'_{x,y}}(F)\backslash H'_{x,y}(F)}\int_{Z_{G'}(F)}f'(h_1^{-1}(\de_1,\de_2)z h_2)\omega'(z)dzdh_1dh_2\nonumber\\
 &=\int_{Z_{G'}(F)}\Orb(f',z\de)\omega(\Nm(z))d{z}.
\end{align}
Note that the integrand is stated in terms of the variable $z\in Z_{G'}(F)$, but that it depends only on $\Nm(z)$, and is independent of the lift. In particular, there are four characters $\omega$ satisfying $\omega\circ\Nm=\omega'$, and the above integral does not depend on a choice of such character.

In light of Remark \ref{Rem: oof the center} and Lemma \ref{matching via center}, for any central character $\omega:Z_{G}(F)\to\cc^\times$ and $f\in C_c^\infty(G(F))$, we restrict $\omega$ to the open subgroup 
\[
\Nm(Z_{G'}(F))\cong \Nm_{E/F}(E^\times)^2
\]
and set
\begin{align}\label{eqn: twisted OI final}
    \Orb^\eta_\omega(f,\ga)&:=\int_{H(F)}\int_{ H(F)}\int_{\Nm(Z_{G'}(F))}f(h_1^{-1}(\ga_1,\ga_2)z h_2)\omega(z)\eta(\det(h_2))dzdh_1dh_2\nonumber\\
    &=\int_{\Nm(Z_{G'}(F))}\Orb^\eta(f,z\ga)\omega(z)dz.
\end{align}
The integration is absolutely convergent. Note that $[Z_G(F): \Nm(Z_{G'}(F))]=4$.
\subsection{Smooth transfer}

We say that functions $f\in C_c^\infty(G(F))$ and $\{f'_{x,y}\}_{x,y}$ with $f'_{x,y}\in C_c^\infty(G'(F))$ and $(x,y)\in \calv_n\times \calv_{n+1}$ match or are transfers if for any matching regular semi-simple orbits $\ga\xleftrightarrow{x,y}\de$, the following identify holds
\begin{equation}\label{eqn: relevant transfer}
\Omega(\ga)\Orb^\eta(f,\ga)=\Orb(f'_{x,y},\de).
\end{equation}

When $E=F\times F$, the transfer of functions may be made explicit. Here, $\calv_n$ and $\calv_{n+1}$ are both singletons and $\eta$ is trivial. For $k=n,n+1$, we may fix isomorphisms $\GL_k(E)\cong \GL_k(F)\times \GL_k(F)$ such that the unitary groups $U(V_x)\cong \GL_n(F)\hra\GL_n(E)$ and $U(V_y)\cong \GL_{n+1}(F)\hra \GL_{n+1}(E)$ are sent to
\[
U(V_x)\cong\{(g,g^\theta)\in \GL_n(F)\times \GL_n(F): g\in \GL_n(F)\}
\]
and 
\[
U(V_y)\cong\{(g,g^\theta)\in \GL_{n+1}(F)\times \GL_{n+1}(F): g\in \GL_n(F)\},
\]
where we recall that for $g\in \GL_k(E)$, $g^\theta = w_k{}^tg^{-1}w_k$. The proof of the next proposition is a simple computation, which we omit.
\begin{Prop}\label{Prop: split transfer}
When $E=F\times F$ as above, the functions $f_1\otimes f_2\in C_c^\infty(G(F)\times G(F))$ and $f_1\ast f_2^{\theta\vee}\in C_c^\infty(G(F))$ are smooth transfers. Here $\ast$ denotes convolution and 
\[
f^{\theta\vee}(g) = f(g^{-\theta}).
\]
\end{Prop}

Assume now that $F$ is non-archimedean and that $E/F$ is a quadratic field extension. The existence of smooth transfer  now follows from the existence of smooth transfer for the Jacquet--Rallis transfer.
\begin{Thm}\label{Thm: smooth transfer}
Assume that $E/F$ is a quadratic extension of non-archimedean fields. For any $f\in C_c^\infty(G(F))$, there exists a transfer $\{f'_{x,y}\}_{x,y}$. Conversely, for any collection $\{f'_{x,y}\}_{x,y}$, there exists a transfer $f$.
\end{Thm}
\begin{proof} This follows from Theorem \ref{Thm: JR transfer} by the reductions (\ref{eqn: linear OI reduction}) and (\ref{eqn: twisted OI reduction}) in the previous section. Indeed, the identity of orbital integrals (\ref{eqn: relevant transfer})  may be reduced to

\[
\omega(\ga_1^{-1}\ga_2)\Orb^{\GL_n(F),\eta}(\tilde{f},\ga_1^{-1}\ga_2)=\Orb^{U(W)}(\tilde{f}_{x,y},\pi_{V_{y(x)},W}(\de)).
\]
This is precisely the context of Theorem \ref{Thm: JR transfer}.
\end{proof}

\begin{Cor}\label{Cor: transfer with central character}
Assume that $E/F$ is a quadratic extension of non-archimedean fields. Suppose that $f\in C_c^\infty(G(F))$ and $\{f'_{x,y}\}_{x,y}\subset C_c^\infty(G'(F))$ are transfers. Let $\omega$ denote a central character for $G(F)$ and let $\omega'=\omega\circ \Nm$ denote its base change to a central character for $G'(F)$. For any matching $Z$-regular semi-simple orbits $\ga\xleftrightarrow{x,y}\de$, we have
\begin{equation}\label{eqn: relevant transfer with center}
\Omega(\ga)\Orb^\eta_\omega(f,\ga)=\Orb_{\omega'}(f'_{x,y},\de).
\end{equation}
\end{Cor}
\begin{proof}
    For any $z\in Z_{G'}(F)$, Lemma \ref{matching via center} tells us $\de z$ matches $\ga\Nm(z)$. It is easy to see that 
    \[
    \Omega(\ga\Nm(z)) = \Omega(\ga).
    \]
    The corollary now follows from the formulas \eqref{eqn: unitary OI final} and \eqref{eqn: twisted OI final}.
\end{proof}

Now fix a single pair $(x,y)\in X_n\times X_{n+1}$, and consider a function $f'_{x,y}\in C_c^\infty(G'(F))$. The above theorem tells us that there exists $f\in C_c^\infty(G(F))$ such that
\begin{equation}\label{eqn: reduction of matching}
\Omega(\ga)\Orb^\eta(f,\ga)=\begin{cases}\Orb(f'_{x,y},\de)&:\ga\xleftrightarrow{x,y} \de\in G'(F)^{rss},\\\qquad0&:\text{otherwise}.\end{cases}
\end{equation} 
Consider the closed and open subset $G[{x,y}]=G_F[x,y]\subset G(F)$ such that
\[
G[{x,y}]=\{(g_1,g_2)\in\GL_n(F)\times \GL_{n+1}(F): \eta(\det(g_1)) = \eta(\det(x)),\;\eta(\det(g_2)) = \eta(\det(y))\}.
\]
\begin{Lem}\label{Lem: reduction of support}
Assume that either $E/F$ is split or that $F$ is non-archimedean. For any function $f'_{x,y}\in C_c^\infty(G'(F))$, we may choose $f\in C_c^\infty(G(F))$ satisfying \eqref{eqn: reduction of matching} such that $\supp(f)\subset G[{x,y}]$.
\end{Lem}
\begin{proof}
When $E/F$ is split, $ G[{x,y}]=G(F)$ so that the statement is vacuous. We now assume that $F$ is non-archimedean. In this case, there are four possible pairs of Hermitian spaces. We index them as follows:
\[
\{(x_i,y_j): (i,j)\in \{0,1\}^2 \text{ such that } \eta(\det(x_i))=(-1)^i,\:\eta(\det(y_j))=(-1)^j\}.
\]
There is then a decomposition of $G(F)$ into open and closed subsets.
\[
G(F) = \bigsqcup_{(i,j)\in \{0,1\}^2}G[x_i,y_j].
\]
Similarly, we may decompose $\GL_{n+1}(F) = G_0\sqcup G_1$ where 
\[
G_i= \{g\in \GL_{n+1}(F) : \eta(\det(g)) = (-1)^i\}.
\]
Recall that the map 
\begin{align*}
p: G(F)&\to \GL_{n+1}(F)\\ (g_1,g_2)&\mapsto g^{-1}_1g_2
\end{align*}
is a submersion. Since 
\[
p\left(G[x_0,y_0]\right)=p\left(G[x_1,y_1]\right)=G_0\:\:\text{ and }\:\:p\left(G[x_0,y_1]\right)=p\left(G[x_1,y_0]\right)=G_1,
\]
the disjoint unions above implies that the restrictions
\[
G[x_0,y_0]\xrightarrow{p} G_0,\:\:G[x_1,y_1]\xrightarrow{p} G_0\:\:\text{ and }\:\:G[x_0,y_1]\xrightarrow{p} G_1,\:\: G[x_1,y_0]\xrightarrow{p} G_1
\]
are all submersions. Therefore, the map 
\begin{align}\label{eqn: indexed pushforwards}
C_c^\infty(G[x_i,y_j])&\lra C_c^\infty(G_{i+j})\\ f&\mapsto \tilde{f}, \nonumber
\end{align}
is surjective, for each $(i,j)\in \{0,1\}^2$ where the sum $i+j$ is considered modulo $2$.
 To conclude the lemma, note that (\ref{eqn: reduction of matching}) is equivalent to 
\begin{equation}\label{bad notation}
\omega(g)\Orb^{\GL_{n}(F),\eta}(\tilde{f},g)=\begin{cases}\Orb^{U(V_x)}(\tilde{f}'_{x,y},h)&:g\longleftrightarrow h\in X_{{y(x)}}^{y,rss},\\\qquad0&:\text{otherwise}.\end{cases}
\end{equation} 
By construction, the support of $\tilde{f}'_{x,y}$ lies in
\[
X_{{y(x)}}^y=\{h\in X_{{y(x)}}: \eta(\det(h)) = \eta(\det(x))\eta(\det(y))=(-1)^{k({x,y})}\},
\]
where $k({x,y})\in \{0,1\}$. Thus, we may replace $\tilde{f}$ by $\tilde{f}\cdot\bfun_{G_{k({x,y})}}$ without affecting the matching (\ref{bad notation}). The surjectivity of (\ref{eqn: indexed pushforwards}) now implies that we are free to choose $f$ so that $\supp(f)\subset G[{x,y}]$, proving the lemma.
\end{proof}
\subsection{The fundamental lemma}
Now assume that $E/F$ is an unramified quadratic extension of non-archimedean local fields. 

 Set $K=G(\calo_F)$ and $\bfun_K$ the corresponding characteristic function. Also, set $K' =G'(\calo_F)$  and let $\bfun_{K'}$ denote the characteristic function. 
\begin{Thm}\label{Thm: JRL fundamental lemma}
The functions $\bfun_K$ and $\{f'_{x,y}\}$ are transfers where
\[
f'_{x,y} = \begin{cases} \bfun_{K'}&: ({x,y}) =(w_n,w_{n+1}),\\ 0&: \text{otherwise}.\end{cases}
\]
\end{Thm}
\begin{proof}
This follows from the previous reductions and Theorem \ref{Thm: JR fundamental lemma} by restricting to the integral locus such that $|\det(X)|_F=1$. As previously stated, this was recently reproved in characteristic zero with no assumption on the residue characteristic in \cite{beuzart2019new}. We therefore do not need to make any assumptions on the residue characteristic.
\end{proof}

\section{Factorization of certain global distributions}\label{Section: factorizations}
We recall the definitions of certain global and local distributions that arise in the spectral decomposition of our relative trace formulas. In this section, $F$ is a number field and $E/F$ a quadratic extension. We set $\A_F$ for the adele ring of $F$, and $\A_E$ for that of $E.$  As always, we consider the diagonal embedding $\GL_n(F)\hra \GL_n(\A_F)$. For any $n$, we set $A_n\subset \GL_n(\A_F)$ to be the connected component of the identity in the $\rr$-points of the maximal $\qq$-split torus in the center of $\GL_n(F_\infty)=\prod_{v|\infty}\GL_n(F_v).$ 
We set
\[
[\GL_n]:=A_n\GL_n(F)\backslash \GL_n(\A_F).
\]
We adopt similar notations for other algebraic groups. Finally, we set $\eta:=\eta_{\A^\times_E/\A^\times_F}$ to be the idele class character associated to the quadratic extension.

\subsection{Peterson inner product}
Suppose $\pi$ is a cuspidal automorphic representation of $\GL_n(\A_F)$, and let $\hat{\pi}\cong \pi^\vee$ denote the contragredient representation of $\pi$ realized on the space of functions $\{\phi^{{\theta}}:\phi\in \pi\}$, where $g^{{\theta}}=w_n{}^tg^{-1}w_n$. Consider the inner product 
\[
(\phi,\hat{\phi}) = \int_{A_n\GL_n(F)\backslash \GL_n(\A_F)}\phi(g)\hat{\phi}(g)dg;
\]
this is a $\GL_n(\A_F)$-invariant inner product on $\pi.$ Denote by $\W^\phi$ the $\psi_0$-th Fourier coefficient of a cusp form $\phi$:
\[
\W^\phi(g) = \int_{[N_n]}\phi(ng)\psi_0^{-1}(n)dn,
\]
where $\psi_0$ is our generic character of the unipotent subgroup $N_n(\A_F)$

Suppose now that $S$ is a finite set of places, containing the archimedean ones, such that $\pi_v$ is unramified and $\psi_{0,v}$ has conductor $\calo$ for $v\notin S$. Let $\phi\in \pi$ be such that $W^\phi$ is factorizable (for simplicity, we will say that $\phi$ is factorizable), write $\W^\phi(g) = \prod_vW_v(g_v)$, where $W_v\in \calw^{\psi_v}(\pi_v)$. Similarly, let $\hat{\phi}\in \hat{\pi}$ be factorizable and set $\W^{\hat{\phi}}(g) = \prod_v\hat{W}_v(g_v)$, where $\hat{W}_v\in \calw^{\psi^{-1}_v}(\hat{\pi}_v)=\calw^{\psi^{-1}_v}(\pi^\vee_v)$. We may assume that for all $v\notin S$, $W_v$ and $\hat{W}_v$ are spherical and normalized so that $W_v(e)=\hat{W}_v(e)=1$.

We recall the canonical inner product
\[
[\cdot,\cdot]_{\pi_v}:\calw^{\psi_v}(\pi_v)\otimes\calw^{\psi^{-1}_v}(\hat{\pi}_v)\lra \cc.
\]
It is defined by considering the integral
\[
I_s(W_v,\hat{W}_v')=L(n,\bfun_{F_v^\times})\int_{N_n(F_v)\backslash P_n(F_v)}W_v(h)\hat{W}'_v(h)|\det(h)|_F^sdh,
\]
where $W_v,W_v'\in \calw^{\psi_v}(\pi_v)$ and $P_n\cong \GL_{n-1}\rtimes \mathbb{G}_a^{n-1}$ is the mirabolic subgroup of $\GL_n.$ The integral converges for $\mathrm{Re}(s)\gg0$, and has meromorphic continuation. It is known for any local field of characteristic zero (see \cite[Appendix A]{FLO} and the references therein) that this continuation is holomorphic at $s=0$ and gives a non-degenerate $\GL_n(F_v)$-invariant pairing. We set
\[
[W_v,\hat{W}'_v]_{\pi_v}:=I_0(W_v,\hat{W}_v').
\]
Moreover, when $\pi_v$ is unramified and $W_v$ is the spherical vector normalized so that $W_v(e) =1,$ then
\begin{equation*}
[W_v,\hat{W}_v]_{\pi_v}=L(1,\pi_v\times \pi_v^\vee),
\end{equation*}
where $L(s,\pi_v\times \pi_v^\vee)$ denotes the local Rankin-Selberg $L$-factor.
\begin{Prop}\cite[Section 10.3]{FLO}\label{Prop: nice inner product}
Assume that $\phi\in\pi$ is factorizable as above. There is a corresponding factorization
\begin{align*}
(\phi,\hat{\phi})&=\Res_{s=1}L(s,\pi\times\pi^\vee)\prod_{v}[W_v,\hat{W}_v]^\natural_{\pi_v},
\end{align*}
where \begin{equation}\label{eqn: normalized inner product}
    [W_v,\hat{W}_v]^\natural_{\pi_v}=\frac{[W_v,\hat{W}_v]_{\pi_v}}{L(1,\pi_v\times\pi^\vee_v)},
\end{equation}
which equals $1$ for the normalized spherical vector.\end{Prop}

As we will need to take the central character into account, we will also use the following variant of $(\cdot,\cdot)$. For $\phi$ and $\hat{\phi}$ as above, consider the inner product 
\begin{equation*}
    \la\phi,\hat{\phi}\ra_{Pet} = \int_{Z_{\GL_n}(\A_F)\GL_n(F)\backslash \GL_n(\A_F)}\phi(g)\hat{\phi}(g)dg.
\end{equation*}
It follows from Proposition \ref{Prop: nice inner product} that 
\begin{align}\label{eqn: inner product}
   \la\phi,\hat{\phi}\ra_{Pet} &=\frac{1}{\vol(F^\times\backslash \A_F^1)}(\phi,\hat{\phi})=  \frac{\Res_{s=1}L(s,\pi\times\pi^\vee)}{\vol(F^\times\backslash \A_F^1)}\prod_{v}[W_v,\hat{W}_v]^\natural_{\pi_v};
\end{align}
a similar formula holds for $\Res_{E/F}(\GL_n)$.

\subsection{Rankin-Selberg period} The results of this section are found in \cite{JPSS}. Let $\Pi=\Pi_n\boxtimes\Pi_{n+1}$ be a generic cuspidal automorphic representation of $G(\A_F)$, where $\Pi_k$ is a generic cuspidal automorphic representation of $\GL_k(\A_F)$ ($k\in\{n,n+1\}$). The global Rankin-Selberg period is given by
\begin{equation*}
\lam_{\Pi}(\phi)  =\int _{[\GL_n]}\phi_n(h) \phi_{n+1}\left(\begin{array}{cc}h&\\&1\end{array}\right)dh,
\end{equation*}
where $\phi_k\in \Pi_k$. Now let $\calw^{\psi^{-1}}(\Pi_n)$ and $\calw^{\psi}(\Pi_{n+1})$ be the corresponding Whittaker models. The local Rankin-Selberg period is defined in terms of the local Whittaker model $\calw(\Pi_v): = \calw^{\psi_v^{-1}}(\Pi_{n,v})\boxtimes\calw^{\psi_v}(\Pi_{n+1,v})$ by
\[
\lam_{\Pi_v}(s,W_v) = \int_{N_n(F_v)\backslash\GL_n(F_v)}W_v(h)|\det(h)|^sdh,\quad s\in \cc, \: W_v\in \calw(\Pi_v).
\]
We also consider the normalized period by introducing the local Rankin-Selberg $L$-function $L(s,\Pi)$:
\begin{equation*}
\lam_{\Pi_v}^\natural(s,W_v) = \frac{\lam_{\Pi_v}(s,W_v)}{L(s+\frac{1}{2},\Pi_v)}.
\end{equation*}
For any generic $\Pi_v$, the integral $\lam_{\Pi_v}(s,\cdot)$ is absolutely convergent when $\mathrm{Re}(s)$ is large and extends meromorphicallly to $\cc$. Moreover, the normalized integral $\lam_{\Pi_v}^\natural(s,\cdot)$ is entire in $s\in \cc$, and we set
\begin{equation*}
\lam_{\Pi_v}^\natural(W_v)=\lam_{\Pi_v}^\natural(0,W_v);
\end{equation*}
this gives a non-zero element of the one-dimensional space $\Hom_{H(F_v)}(\Pi_v,\cc)$ for any generic $\Pi$.
\begin{Rem}
When $\Pi_v$ is tempered, the integral is absolutely convergent for $\mathrm{Re}(s)>-\frac{1}{2},$ so that there is no need to analytically continue the integral to $s=0$.
\end{Rem}
With our measure conventions, when $W_v$ is the normalized spherical vector and $\psi_v$ has conductor $\calo_F$, 
\[
\lam_{\Pi_v}(s,W_v) ={L(s+\frac{1}{2},\Pi_v)},
\]
where $L(s,\Pi_v) = L(s,\Pi_{n,v}\times \Pi_{n+1,v})$ is the local Rankin-Selberg $L$-factor. This implies that
\begin{equation}\label{spherical Rankin}
\lam_{\Pi_v}^\natural(W_v) =1.
\end{equation}
\begin{Prop}
We have the following decomposition when $\Pi$ is unitary, and $\phi=\phi_n\boxtimes \phi_{n+1}\in \Pi$ is factorizable:
\begin{equation}\label{eqn: Rankin product}
\lam_\Pi(\phi)  = L\left(\frac{1}{2},\Pi_n\times\Pi_{n+1}\right)\prod_v\lam_{\Pi_v}^\natural(W_v),
\end{equation}
where $\W^{{\Pi}}(g: \phi) = \prod_v{W}_v(g_v)$.
\end{Prop}
We will also need the twisted version $\lam^\eta$ of this period, where $\eta:F^\times\backslash \A_F^\times\to \cc^\times$ is a quadratic character. This distribution is given, both globally and locally, by setting
\[
\lam_{\Pi_n\boxtimes\Pi_{n+1}}^\eta = \lam_{\Pi_n\cdot\eta\boxtimes\Pi_{n+1}},
\]
and similarly for the normalized distribution.
\begin{Cor}
We have the following decomposition when $\Pi$ is unitary, and $\phi=\phi_n\boxtimes \phi_{n+1}\in \Pi$ is factorizable:
\begin{equation}\label{eqn: Rankin product twist}
\lam^\eta_{\Pi}(\phi)  = L\left(\frac{1}{2},\Pi_n\times\Pi_{n+1}\cdot\eta\right)\prod_v\lam_{\Pi_v}^{\eta_v,\natural}(W_v).
\end{equation}
\end{Cor}
\subsection{Unitary periods}\label{Section: unitary periods}
 Let $\Pi$ be a cuspidal automorphic representation of $\GL_n(\A_E)$. For any Hermitian form $x\in X_n$, we have the associated unitary group $\U(V_x)\subset \Res_{E/F}\GL_n$. For $\phi\in \Pi$, we define the \emph{unitary period} $\calp_x(\phi)$ by the (convergent) integral
\begin{equation*}
\calp_x(\phi)= \int_{[U(V_x)]}\phi(h)dh.
\end{equation*}
Then $\Pi$ is said to be \emph{distinguished by $U(V_x)$} if there exists a vector $\phi$ such that $\calp_x(\phi)\neq0.$ If $V_x=V_n$ is split, then a theorem of Jacquet \cite{JacquetQuasisplit} states that if there is a cuspidal automorphic representation $\pi$ of $\GL_n(\A_F)$ such that $\Pi=BC(\pi)$, then $\Pi$ is {distinguished by $U(V_n)$}. In general, Corollary 10.3 of \cite{FLO} gives a vast generalization to other forms.

To define the local unitary periods, we need a bit more terminology. Fix a place $v$ of $F$ and consider the quadratic \'{e}tale extension $E_v/F_v$. Let $\Pi_v\in \mathrm{Temp}(\GL_n(E_v))$ and denote by $\cale(X_n,\calw^{\psi_v'}(\Pi_v)^\ast)$ the set of all maps 
\[
\al:\Xv\times \calw^{\psi_v'}(\Pi_v)\to \cc,
\] 
which are continuous and $\GL_n(E_v)$-invariant with respect to the diagonal action. Note that  we have an isomorphism
\begin{align*}
\cale(X_n,\calw^{\psi_v'}(\Pi_v)^\ast)&\iso \bigoplus_{x\in\calv_n}\Hom_{U(V_x)}(\calw^{\psi_v'}(\Pi_v),\cc)\\
\al&\mapsto (\al(x,\cdot))_{x\in\calv_n}.
\end{align*}
Now for any such $\al$, we consider the \emph{twisted Bessel character} $J^\al_{\Pi_v}:C_c^\infty(\Xv)\to \cc$ given by
\[
J_{\Pi_v}^\al(f') = \la f'\cdot \al,\lam_1^\vee\ra,
\]
where $f'\cdot \al$ is the smooth functional
\[
W\mapsto \int_{\Xv}f'(x)\al(x,W)dx,
\]
which we identify with an element of $\calw^{{\psi_v'}^{-1}}(\Pi^\vee_v)$ via the pairing $[\cdot,\cdot]_{\Pi_v}$, and $\lam_1^\vee$ denotes the functional $\hat{W}\mapsto \hat{W}(1).$ Similarly, for $\pi_v\in\mathrm{Temp}(\GL_n(F_v))$, we define the \emph{Bessel character} $I_{\pi_v}: C_c^\infty(\GL_n(F_v))\to \cc$ by
\[
I_{\pi_v}(f) = \la f\cdot \lam_{w_n},\lam_1^\vee\ra,
\]
where $f\cdot \lam_{w_n}$ denotes the smooth functional
\[
W\mapsto \int_{\GL_n(F_v)}f(g)W(w_ng)dg,
\]
which we again identify with an element of $\calw^{\psi_v^{-1}}({\pi^\vee_v})$ via the pairing $[\cdot,\cdot]_{\pi_v}$, and $\lam_1^\vee$ denotes the functional $\hat{W}\mapsto \hat{W}(1).$ 

One of the main results of \cite{FLO} is the following theorem.
\begin{Thm}
For every $\pi_v\in \mathrm{Temp}(\GL_n(F_v))$, there exists a unique \[
\al^{\pi_v}\in \cale(X_n,\calw^{\psi'_v}(BC(\pi_v))^\ast)
\]such that the identity
\begin{equation*}
J_{BC(\pi_v)}^{\al^{\pi_v}}(f') =I_{\pi_v}(f)
\end{equation*}
holds for all pairs of test functions $(f,f')$ which are transfers in the sense defined in \cite[Section 3]{FLO}.
\end{Thm}
We refer to the functionals $\al^{\pi_v}$ as \textbf{FLO functionals}. When $E_v\cong F_v\times F_v$ is split, so that $BC(\pi_v) \cong \pi_v\otimes \pi_v$, these functionals are very simple \cite[Corollary 7.2]{FLO}:
\begin{equation}\label{eqn: split FLO identity}
\al_{(h,{}^th)}^{\pi_v}(W'\otimes W'') = \left[\calw(h,\pi_v)W',\calw(w_n,\hat{\pi}_v)\hat{W}''\right]_{\pi_v}
\end{equation}
for any $h\in \GL_n(F_v)$ and $W',W''\in \calw^{\psi_v}(\pi_v)$. 
\begin{Lem}\cite[Lemma 3.9]{FLO}
Assume that $F_v$ is non-archimedean of odd residue characteristic. Further assume that $E_v/F_v$ is an unramified extension and that $\psi_v'$ has conductor $\calo_{F_v}$. Let $\Pi_v=BC(\pi_v)$ be unramified and let $W_0\in \calw^{\psi_v'}(\Pi_v)$ denote the normalized spherical vector. Then for any $x\in X_n(\calo_{F_v})= \GL_n(\calo_{E_v})\ast w_n$
\[
\al_{x}^{\pi_v}(W_0) = L(1,\pi_v\times {\pi}_v^\vee\cdot \eta_v).
\]
\end{Lem}

We define for $W\in \calw^{\psi_v'}(\Pi_v)$
\begin{equation}\label{spherical FLO}
\al_x^{\pi_v,\natural}(W) = \frac{\al_x^{\pi_v}(W)}{ L(1,\pi_v\times {\pi}_v^\vee\cdot \eta_v)}.
\end{equation}
Returning to our extension of number fields $E/F$, we have the following factorization of unitary periods of cusp forms.
\begin{Prop}\cite[Theorem 10.2]{FLO}
Let $\pi$ be an irreducible cuspidal automorphic representation of $\GL_n(\A_F)$, and let $\Pi = BC(\pi)$. Then for any $x\in X_n$, we have
\begin{equation}\label{eqn: period product}
\calp_x(\phi) = 2L(1,\pi\times {\pi}^\vee\cdot \eta) \prod_v\al_x^{\pi_v,\natural}(W_v),
\end{equation}
where $\W^{{\Pi}}(g: \phi) = \prod_v{W}_v(g_v)$.
\end{Prop}

\subsection{Global relative characters}
Assume that $\pi=\pi_n\boxtimes \pi_{n+1}$ is a cuspidal automorphic representation of $G(\A_F)$.
\begin{Def}
We define the global \emph{relative character} $I_{\pi}$ as the following distribution: for $f\in C_c^\infty(G(\A_F))$, we set
\begin{equation}\label{eqn: global bessel}
I_{\pi}(f) =\sum_\phi\frac{\lam_\pi(\pi(f)\phi)\lam^\eta_{\pi^\vee}(\hat{\phi})}{\la\phi,\hat{\phi}\ra_{Pet}},
\end{equation}
where the sum runs over an orthogonal basis for $\pi$ and where $\hat{\phi}(g) = \phi(g^{\theta})$ is a vector in the contragredient representation $\pi^\vee$.
\end{Def}

Now let $\Pi=\Pi_n\boxtimes \Pi_{n+1}$ be an irreducible cuspidal automorphic representation of $G'(\A_F)$. For any $(x,y)\in \calv_n\times \calv_{n+1}$, we set ${\calp}_{x,y}= \calp_x\otimes \calp_y$ to be the product of unitary periods.
\begin{Def}
We define the global \emph{twisted relative character} $J_{\Pi}^{x,y}$ as the following distribution: for $f'\in C_c^\infty(G'(\A_F))$, we set
\begin{equation}\label{eqn: twisted global bessel}
J_{\Pi}^{x,y}(f') =\sum_\phi\frac{\lam_\Pi(\Pi(f')\phi)\calp_{x,y}(\hat{\phi})}{\la\phi,\hat{\phi}\ra_{Pet}},
\end{equation}
where the sum runs over an orthogonal basis for $\Pi$ and where $\hat{\phi}(g) = \phi(g^{\theta})$.
\end{Def}

Note that $J_{\Pi}^{x,y}\equiv 0$ unless $\Pi=\Pi_n\boxtimes \Pi_{n+1}$ is $H'_{x,y}$-distinguished. When this is the case, a theorem of Jacquet \cite{JacquetKloosterman2} implies that there must exist a cuspidal automorphic representation $\pi=\pi_n\boxtimes \pi_{n+1}$ of $G(\A_F) = \GL_n(\A_F)\times \GL_{n+1}(\A_F)$ such that
\[
\Pi_i=BC(\pi_i), \qquad i=n,n+1.
\]
Since $\Pi_i$ is cuspidal, we know $\pi_i\not\cong\pi_i\cdot \eta$. Therefore, the Rankin-Selberg $L$-function $L(s,\pi\times \pi^\vee\cdot \eta)$ is holomorphic at $s=1$. For $i=n,n+1$, we may use the relation
\begin{equation}\label{L value relation}
L(s,\Pi_i\times \Pi_i^\vee) = L(s,\pi_i\times \pi_i^\vee)L(s,\pi_i\times \pi_i^\vee\cdot \eta),
\end{equation}
we see that
\begin{equation}\label{eqn: special value}
\frac{\Res_{s=1}(L(s,\Pi_i\times \Pi_i^\vee))}{L(1,\pi_i\times \pi_i^\vee\cdot \eta)}= \Res_{s=1}(L(s,\pi_i\times \pi_i^\vee))\neq 0.
\end{equation}

\subsection{Local relative characters and factorization}
Denote by $\calw(\Pi_v)$ the Whittaker model $\calw^{{\psi'}_v^{-1}}(\Pi_{n,v})\otimes \calw^{\psi'_v}(\Pi_{n+1,v})$. Also denote by ${\al}^{\pi_v}_{x,y}= \al_x^{\pi_{n,v},\natural}\otimes \al_y^{\pi_{n+1,v},\natural}$ the product of FLO functionals.
\begin{Def}\label{Def: local functional}
\begin{enumerate}
\item We define the \emph{normalized local twisted relative character} $J^{x,y,\natural}_{\Pi_v}$ associated to a unitary generic representation $\Pi_v$ of $G'(F_v)$ and $(x,y)\in \calv_n\times \calv_{n+1}$ as follows: for $f'_v\in C_c^\infty(G'(F_v))$
\begin{equation*}
J^{x,y,\natural}_{\Pi_v}(f'_v) = \sum_{W_v}\frac{\lam_{\Pi_v}^\natural(\Pi_v(f'_v)W_v){\al}^{\hat{\pi}_v}_{x,y}(\hat{W}_v)}{[W,\hat W]^\natural_{\Pi_v}},
\end{equation*}
where the sum ranges over an orthogonal basis for $\calw(\Pi_v)$. We denote by $J_{\Pi_v}^{x,y}$ the distribution defined using the \emph{unnormalized} local periods.
\item We similarly define the \emph{normalized local relative character} $I^\natural_\pi$ for any unitary generic representation $\pi_v$ of $G(F_v)$:  for $f_v\in C_c^\infty(G(F_v))$
\begin{equation*}
I^\natural_{\pi_v}(f_v) = \sum_{W_v}\frac{\lam_{\pi_v}^\natural(\pi_v(f_v)W_v)\lam_{\hat{\pi}_v}^{\eta_v,\natural}(\hat{W}_v)}{[W,\hat W]^\natural_{\pi_v}},
\end{equation*}
where the sum ranges over an orthogonal basis for $\calw(\pi_v)$. We denote by $I_{\pi_v}$ the distribution defined using the \emph{unnormalized} local periods.
\end{enumerate}
\end{Def}

We now state the factorization results.

\begin{Prop}\label{Prop: twisted factorization}
Consider an irreducible cuspidal automorphic representation $\Pi=\otimes_v'\Pi_v$ and fix a pair $(x,y)$. If $J_{\Pi}^{x,y}$ is not identically zero, let $\pi=\pi_n\boxtimes \pi_{n+1}$ be an irreducible cuspidal automorphic representation of $G(\A_F)$ such that $\Pi=BC(\pi)$.

We have the product decomposition: for any factorizable $f'=\prod_vf_v'\in C_c^\infty(G'(\A_F))$,
\begin{equation}\label{eqn: twisted Bessel factors}
J_{\Pi}^{x,y}(f') = 4\vol(E^\times\backslash \A_E^1)^2\frac{L\left(\frac{1}{2},\Pi_n\times \Pi_{n+1}\right)}{\Res_{s=1}(L(s,\pi\times \pi^\vee))}\prod_vJ^{x,y,\natural}_{\Pi_v}(f'_v). 
\end{equation}
\end{Prop}
\begin{proof}
The product decomposition follows immediately from (\ref{eqn: inner product}), (\ref{eqn: Rankin product}), (\ref{eqn: period product}), and (\ref{eqn: special value}).
\end{proof}
This implies that the global twisted relative character $J_{\Pi}^{x,y}$ is non-vanishing if and only if
\begin{enumerate}
\item the global $L$-factor $L\left(\frac{1}{2},\Pi_n\times \Pi_{n+1}\right)$ is non-vanishing, and
\item the local FLO functionals ${\al}^{\pi_v}_{x,y}= \al_x^{\pi_{n,v},\natural}\otimes \al_y^{\pi_{n+1,v},\natural}$ are non-vanishing for every place $v$ of $F$.
\end{enumerate}
This is due to the non-vanishing of the local Rankin--Selberg periods for any generic representation $\Pi_v$ \cite{JPSS}.
\begin{Rem}
    In defining $J^{x,y,\natural}_{\Pi_v}$ in Definition \ref{Def: local functional}, we need to fix an auxiliary choice of a representation $\pi_v = \pi_{n,v}\boxtimes\pi_{n+1,v}$. Since we only consider the global cuspidal setting, the global factorizable distributions in which these local factors occur do not depend on this choice (see the proof of Theorem 10.2 in \cite{FLO}).
\end{Rem}
In the linear case, we have a similar factorization.

\begin{Prop}\label{Prop: linear factorization}
Consider an irreducible cuspidal automorphic representation $\pi=\otimes_v'\pi_v$. For any factorizable $f=\prod_vf_v\in C_c^\infty(G(\A_F))$, the relative character $I_{\pi}$ associated to $\pi$ factorizes as
\begin{equation}\label{eqn: Bessel factors}
I_{\pi}(f) = \vol(F^\times\backslash \A_F^1)^2\frac{L\left(\frac{1}{2},\Pi_n\times \Pi_{n+1}\right)}{\Res_{s=1}(L(s,\pi\times \pi^\vee))}\prod_vI^\natural_{\pi_v}(f_v),
\end{equation}
where $\Pi_n\boxtimes \Pi_{n+1}=BC(\pi_n)\boxtimes BC(\pi_{n+1})$.
\end{Prop}
\begin{proof}
This similarly follows from (\ref{eqn: inner product}), (\ref{eqn: Rankin product}), and (\ref{L value relation}).
\end{proof}

In particular, for any cuspidal automorphic representation $\pi_n\boxtimes \pi_{n+1}$, the global relative character is non-vanishing if and only if the central $L$-value $L\left(\frac{1}{2},\Pi_n\times \Pi_{n+1}\right)$ is non-vanishing.

\section{Comparison of relative trace formulas}\label{Section: comparison}
The main result of this section is Theorem \ref{Prop: main global tool}, establishing the following spectral transfer of global relative characters. Assume that $E/F$ is a quadratic extension of number fields that is split at all archimedean places of $F$. Let $\Pi=\Pi_n\boxtimes\Pi_{n+1}$ be a cuspidal automorphic representation of $G'(\A_F)$ satisfying certain local assumptions and such that $\Pi\cong \Pi^\sig$. Then for any pair of Hermitian forms $(x,y)\in X_n\times X_{n+1}$,  Theorem \ref{Prop: main global tool} implies that there exist ``nice'' matching test functions $f'\in C_c^\infty(G'(\A_F))$ and $f\in C_c^\infty(G(\A_F))$ such that
\begin{equation*}\label{eqn: global character identity}
   \frac{1}{L(1,\eta)^2} J_{\Pi}^{x,y}(f') = \sum_{\pi\in \calb(\Pi)}I_{\pi}(f),
\end{equation*}
where $\calb(\Pi)$ is the (finite) set of cuspidal automorphic representations $\pi$ of $G(\A_F)$ such that $\Pi=BC(\pi)$.

We prove this statement as well as one involving the factorizations \eqref{eqn: twisted Bessel factors} and \eqref{eqn: Bessel factors} via a comparison of (simple forms of) two relative trace formulas, which we now introduce. Much of this section mirrors \cite[Section 2]{ZhangFourier}, as the trace formula comparisons are similar.

\subsection{The linear side} Suppose that $f\in C_c^\infty(G(\A_F))$ and consider the automorphic kernel
\[
K_f(x,y)=\sum_{\ga\in G(F)}f(x^{-1}\ga y).
\]
We consider the distribution on $C_c^\infty(G(\A_F))$
\[
I(f) = \displaystyle \int_{[H]}\int_{[H]}K_f(h_1,h_2)\eta(\det(h_2))dh_2dh_1.
\]
 We also consider versions of this distribution $I_\omega$, where $\omega$ is a central character for $G(\A_F)$ by replacing $K_f$
\[
K_{f,\omega}(x,y)=\int_{[Z_G]}\sum_{\ga\in G(F)}f(x^{-1}\ga zy)\omega(z)dz
\] Note that the intersection of $Z_G$ and $H$ is trivial.

These integrals do not converge in general. Following \cite{ZhangFourier}, we introduce the space of \emph{nice} test functions.
\begin{Def}\label{nice test functions} We say $f=\prod_vf_v\in C_c^\infty(G(\A_F))$ is a \emph{nice test function} with respect to the central character $\omega=\prod_v\omega_v$ (or just $\omega$-nice) if 
\begin{enumerate}
\item For at least one finite place $v_1$ that splits in $E$, the function $f_{v_1}$ is \emph{essentially a matrix coefficient} of a supercuspidal representation with respect to $\omega_{v_1}$: this means that
\[
f_{v_1,\omega_{v_1}}(g) =  \int_{Z_G(F_{v_1})}f_{v_1}(gz)\omega_{v_1}(z)dz
\]
is a matrix coefficient of a supercuspidal representation of $G(F_{v_1})$.
\item For at least one split place $v_2\neq v_1$, the test function $f_{v_2}$ is supported on the $Z$-regular locus of $G(F_{v_2})$. This place is not required to be non-archimedean.
\end{enumerate}
\end{Def}
\begin{Rem}
We have opted for this definition to maintain the similarity between our trace formula comparison and that of \cite{ZhangFourier}. A more flexible definition of nice function is given by \cite{BeuzartIIC} in terms of cuspidal Bernstein components at the place $v_1$. For comparison, we are restricting to test functions at a single non-archimedean place that lie in the summand of $C_c^\infty(G(F))$ corresponding to a finite union of supercuspidal Bernstein components. On the other hand, we restrict to compactly supported test functions at archimedean places while Beuzart-Plessis allows Schwartz functions.
\end{Rem}
\begin{Lem}\label{Lem: support nice linear}
Let $\omega$ be a unitary character of $Z_G(F)\backslash Z_G(\A_F)$. Suppose that $f=\prod_vf_v$ is nice with respect to $\omega$. Then
\begin{enumerate}
\item As a function on $H(\A_F)\times H(\A_F)$, $K_f(x,y)$ is compactly supported modulo $H(F)\times H(F)$. In particular, $I(f)$ converges absolutely.
\item As a function on $H(\A_F)\times H(\A_F)$, $K_{f,\omega}(x,y)$ is compactly supported modulo $H(F)\times H(F)Z_G(\A_F)$. In particular, $I_\omega(f)$ converges absolutely.
\end{enumerate}
\end{Lem}
\begin{proof}
The argument is verbatim as in the case of the Jacquet--Rallis relative trace formula \cite[Lemma 2.2]{ZhangFourier}.
\end{proof}
This implies that when $f$ is nice, we have the decomposition into a finite sum of integrals
\[
I(f) = \sum_{\ga}\Orb^\eta(f,\ga),
\]
where the sum is over regular semi-simple $\ga\in H(F)\backslash G(F)/H(F)$ and 
\[
\Orb^\eta(f,\ga):=\int_{H(\A_F)}\int_{H(\A_F)}f(h_1^{-1}\ga h_2)\eta(\det(h_2))dh_2dh_1= \prod_v\Orb^{\eta_v}(f_v,\ga_v).
\]
\begin{Lem}
   For $(i,j)\in \{0,1\}^2$, let $\eta_{i,j}:[Z_G]\to \cc^\times$ denote the characters 
\[
\eta_{i,j}(z_1,z_2)=\eta(z)^i\eta(z)^j.
\]
Fix a unitary central character $\omega: [Z_G]\to \cc^\times$, and let $\omega_{i,j}:=\omega\eta_{i,j}$. If $f=\bigotimes_vf_v$ is $\omega$-nice, we have the decomposition
    \[
    \frac{1}{4}\sum_{(i,j)}I_{\omega_{i,j}}(f) = \sum_{\ga}\Orb_\omega^\eta(f,\ga)
    \]where the sum is over $Z$-regular semi-simple $\ga\in H(F)\backslash G(F)/\Nm(Z_{G'}(F))H(F)$ and where
\begin{align*}
\Orb_\omega^\eta(f,\ga)&=\int_{H(\A_F)}\int_{H(\A_F)}\int_{\Nm(Z_G(\A_E))}f(h_1^{-1}\ga zh_2)\eta(\det(h_2))\omega(z)dzdh_2dh_1\\&=\prod_v\Orb^{\eta_v}_{\omega_v}(f_v,\ga_v),
\end{align*}
where the local orbital integral $\Orb^{\eta_v}_{\omega_v}(f_v,\ga_v)$ is defined in \eqref{eqn: twisted OI final}.
\end{Lem}
\begin{proof}
    When $f$ is $\omega$-nice, Lemma \ref{Lem: support nice linear} implies the following decomposition for $I_\omega$:
\[
I_\omega(f) = \sum_{\ga}\int_{Z_G(F)\backslash Z_G(\A_F)}\Orb^\eta(f,z\ga)\omega(z)dz,
\]
where the sum is over $Z$-regular semi-simple $\ga\in H(F)\backslash G(F)/H(F)$. Using absolute convergence, we compute
\begin{align*}
  \frac{1}{4}\sum_{(i,j)}I_{\omega_{i,j}}(f)&= \frac{1}{4}\sum_{(i,j)}\sum_{\ga}\int_{Z_G(F)\backslash Z_G(\A_F)}\Orb^\eta(f,z\ga)\omega_{i,j}(z)dz\\  
  &= \frac{1}{4}\sum_{\ga}\sum_{(i,j)}\int_{Z_G(F)\backslash Z_G(\A_F)}\Orb^\eta(f,z\ga)\omega_{i,j}(z)dz\\ 
    &=\sum_{\ga}\sum_{(i,j)}\int_{Z_G(F)\backslash Z_G(F)\Nm(Z_{G'}(\A_F))}\Orb^\eta(f,z\ga)\omega(z)dz,
\end{align*}
where we note that the characters $\omega_{i,j}$ all agree on the open subgroup $ Z_G(F)\Nm(Z_{G'}(\A_F))$. We have also used that 
\[
 [Z_G(\A_F): Z_G(F)\Nm_{E/F}(Z_{G'}(\A_F))]=4.
\]
The Hasse norm theorem applied to the quadratic extension $E/F$ now implies that
\[
Z_G(F)\backslash Z_G(F)\Nm(Z_{G'}(\A_F)) = \Nm(Z_{G'}(F))\backslash\Nm(Z_{G'}(\A_F)),
\]
with an identification of measures. Unfolding the integral over $z$, we obtain the desired formula.
\end{proof}

If $\pi=\pi_n\boxtimes \pi_{n+1}$ is a cuspidal automorphic representation of $G(\A_F)$, recall the definition (\ref{eqn: global bessel}) of the relative character $I_\pi$.
\begin{Prop}\label{Prop: twisted RTF}
Let $\omega$ be a unitary central character for $\G(\A_F)$. If $f$ is $\omega$-nice, then we have the equality
\begin{equation}\label{eqn: linear RTF}
 4\sum_{\ga}\Orb_\omega^\eta(f,\ga)=\sum_{\pi}I_{\pi}(f),
\end{equation}
where the first sum is over $Z$-regular semi-simple orbits $$\ga\in H(F)\backslash G(F)^{Z-rss}/\Nm(Z_{G'}(F))H(F),$$ and 
where the second sum runs over irreducible cuspidal automorphic representations of $G(\A_F)$ with central character $\omega_{i,j} = \omega\eta_{i,j}$ for some $(i,j)\in \{0,1\}^2$.
\end{Prop}
\begin{proof}
The proof of this is standard (relying on the foundational analysis of $K_{f,\omega}$ in \cite{rogawskiGL}) so we omit the details. See \cite[Theorem 2.3]{ZhangFourier} for an analogous argument and \cite[Theorem 18.2.2]{GetzHahnbook} for a general treatment including absolute convergence of both sides. We remark it is this step which implicitly uses the normalization $\la\cdot,\cdot\ra_{Pet}$ in \eqref{eqn: inner product} for the inner product used in $I_{\pi}$ (cf. \cite[proof of Theorem 4.3]{ZhangRankin}).
\end{proof}

\subsection{The twisted side}
This unitary side is easier than the previous case. Fix a pair of Hermitian forms $(x,y)\in X_n\times X_{n+1}$. Recall that $G' = \Res_{E/F}(\GL_n\times \GL_{n+1})$ and consider the two subgroups $H'= \Res_{E/F}(\GL_n)$ embedded diagonally and $H'_{x,y}= \U(V_x)\times \U(V_y)$. For $f'\in C_c^\infty(G'(\A_F))$, we form the analogous kernel $K_{f'}$ and consider the distribution
\[
J^{x,y}(f') = \displaystyle\int_{[H']}\int_{[H'_{x,y}]}K_{f'}(h_1,h_2)dh_2dh_1.
\]
We also consider versions of this distribution $J^{x,y}_{\omega'}$, where $\omega'$ is a central character for $Z_{G'}(\A_F)$ that is trivial on $(Z_{G'}\cap H'_{x,y})(\A_F)$, by replacing $K_{f'}$ with 
\[
K_{f',\omega'}(x,y)=\int_{[Z_{G'}]}\sum_{\ga\in G'(F)}f'(x^{-1}\ga zy)\omega'(z)dz
\] Note that such a central character is the base change of a central character $\omega$ of $G(\A_F)$. That is, it is of the form $\omega'=\omega\circ \Nm$, where $\Nm: Z_{G'}(\A_F)\to Z_G(\A_F)$ is the norm map.

As in the linear case, we introduce the space of \emph{nice} test functions. We say $f'=\prod_vf'_v\in C_c^\infty(G'(\A_F))$ is \emph{nice} with respect to the central character $\omega'$ (or just $\omega'$-nice) if 
\begin{enumerate}
\item For at least one finite place $v_1$ that splits in $E$, the function $f'_{v_1}$ is \emph{essentially a matrix coefficient} of a supercuspidal representation with respect to $\omega'_{v_1}$: this means that
\[
f'_{v_1,\omega'_{v_1}}(g) =  \int_{Z_{G'}(F_{v_1})}f'_{v_1}(gz)\omega'_{v_1}(z)dz
\]
is a matrix coefficient of a supercuspidal representation of $G'(F_{v_1})$.
\item For at least one split place $v_2\neq v_1$, the test function $f'_{v_2}$ is supported on the $Z$-regular locus of $G'(F_{v_2})$. This place is not required to be non-archimedean.
\end{enumerate}

If $\Pi=\Pi_n\boxtimes \Pi_{n+1}$ is a cuspidal automorphic representation of $G'(\A_F)$, recall the definition (\ref{eqn: twisted global bessel}) of the twisted relative character $J_{\Pi}^{x,y}$.
\begin{Prop}\label{Prop: unitary RTF}
Let $\omega'$ be a unitary central character as above (so that $\omega'=\omega\circ \Nm$ for some character $\omega$). If $f'$ is $\omega'$-nice, then we have the equality
\begin{equation}\label{eqn: twisted RTF}
 (2L(1,\eta))^2\sum_{\de}\Orb_{\omega'}(f',\de)=\sum_{\Pi}J_{\Pi}^{x,y}(f'),
\end{equation}
where the first sum is over $Z$-regular semi-simple orbits$$\de\in H'(F)\backslash G'(F)^{Z-rss}/Z_{G'}(F)H'_{x,y}(F),$$
we have
\[
\Orb_{\omega'}(f',\de)=\int_{H'(\A_F)}\int_{Z_{H'_{x,y}}(\A_F)\backslash H'_{x,y}(\A_F)}\int_{ Z_{G'}(\A_F)}f'(h_1^{-1}\de h_2z)\omega'(z)dzdh_2dh_1,
\]and 
where the sum on the right-hand side runs over irreducible cuspidal automorphic representations of $G'(\A_F)$ with central character $\omega'$.
\end{Prop}
\begin{proof}
    The absolute convergence and spectral expansion follow as in the previous setting by the niceness assumptions. The only thing to remark on is the unfolding of the geometric side. With the assumption that $f'$ is $\omega'$-nice, we may unfold $J_{\omega'}(f)$ and obtain
    \begin{align*}
        \sum_{\de}\vol([Z_{H'_{x,y}}])&\int_{H'(\A_F)}\int_{Z_{H'_{x,y}}(\A_F)\backslash H'_{x,y}(\A_F)}\int_{ Z_{G'}(\A_F)}f'(h_1^{-1}\de h_2z)\omega'(z)dzdh_2dh_1\\
        &= \sum_{\de}(2L(1,\eta))^2\Orb_{\omega'}(f',\de),
    \end{align*}
    where the sum ranges over $Z$-regular semi-simple orbits in $H'(F)\backslash G'(F)^{rss}/Z_{G'}(F)H'_{x,y}(F),$ and we have used the observation \eqref{uniform stab} that the stabilizer of such an element is isomorphic to $Z_{H'_{x,y}}$. Our measure normalizations from Section \ref{measures} gives the formula
    \[
    \vol([Z_{H'_{x,y}}])=(2L(1,\eta))^2,
    \]
   where $L(s,\eta)$ is the completed $L$-function associated to the quadratic character $\eta=\eta_{E/F}$.
\end{proof}

\subsection{Global matching of test functions}
Suppose now that $f=\bigotimes_vf_v\in C_c^\infty(G(\A_F))$ and $\{f'_{x,y}\}_{(x,y)\in \calv_n\times\calv_{n+1}}$ with $f'_{x,y}=\bigotimes_vf_{x,y,v}'\in C_c^\infty(G'(\A_F))$ where $f_{x,y}'=0$ for all but finitely many $(x,y)$. Suppose further that for each $v$, the local functions $f_v$ and $\{f_{x,y,v}'\}$ are smooth transfers of each other.

If we consider global $Z$-regular semi-simple classes $\ga\in G(F)$ and $\de\in G'(F)$ that match with respect to the pair $(x,y)$, it follows that for each place $v$
\[
\Omega_v(\ga)\Orb^{\eta_v}(f_v,\ga)=\Orb(f_{x,y,v}',\de).
\]
  Noting that the transfer factor satisfies the global product formula
\[
\prod_v\Omega_v(\ga) = 1\text{  whenever  }\ga\in G(F),
\]
 this implies the comparison of global orbital integrals
 \[
 \Orb^\eta(f,\ga) = \prod_v\Omega_v(\ga)\Orb^{\eta_v}(f_v,\ga)=\prod_v\Orb(f_{x,y,v}',\de) = \Orb(f_{x,y}',\de).
 \]
Using the $Z$-regular semi-simple assumption, Corollary \ref{Cor: transfer with central character} gives the matching of orbital integrals with central character
\begin{equation}\label{global orbital match}
 \Orb^\eta_\omega(f,\ga)=\Orb_{\omega'}(f_{x,y}',\de),
\end{equation}
where $\omega=\prod_v\omega_v:Z_G(F)\backslash Z_G(\A_F)\to \cc^\times$ is a unitary central character of $G(\A_F)$ and $\omega'=\omega\circ\Nm.$

 To simplify the statements of a few results, we introduce the following terminology.
\begin{Def}\label{Def: efficient}
    Fix a character $\omega$ of $Z_G(\A_F)$ and let $\omega'$ denote its base change. Fix a split place $v_0$ and a supercuspidal representation $\pi_{v_0}$ of $G(F_{v_0})$. We say that the functions
\[
\text{$f\in C_c^\infty(G(\A_F))$ and $\{f'_{x,y}\}_{(x,y)\in\calv_n\times \calv_{n+1}}$}\subset C_c^\infty(G'(\A_F))
\]
are \textbf{efficient transfers} for $(\omega,\pi_{v_0})$ if 
\begin{enumerate}
    \item each function is a  nice test functions with respect to the characters $\omega$ and $\omega'$ respectively,
    \item they are smooth transfers of each other,
    \item  $f_{v_0}$ is essentially a matrix coefficient of $\pi_{v_0}$ and that $f'_{x,y,v_0}$ is related to $f_{v_0}$ as in Proposition \ref{Prop: split transfer}; note that this implies that $f'_{x,y,v_0}$  is also essentially a matrix coefficient of the base change $\Pi_{v_0}\simeq \pi_{v_0}\otimes \pi_{v_0}$ of $\pi_{v_0}$.
\end{enumerate}
\end{Def}
\subsection{Comparison}To obtain the necessary refined comparison, we make use of the automorphic-Cebotarev-density theorem of Ramakrishnan.
\begin{Thm}\cite{ramakrishnan2018theorem}\label{ramakrishnan2018theorem}
Let $E/F$ be a quadratic extension of global fields. Two cuspidal automorphic representations $\Pi_1$ and $\Pi_2$ of $\Res_{E/F}(\GL_n)(\A_F)$ are isomorphic if and only if $\Pi_{1,v}\cong \Pi_{2,v}$ for almost all places $v$ of $F$ that are split in $E/F$.
\end{Thm}
 Recall that for a cuspidal automorphic representation $\Pi$ of $G'(\A_F)$, we denote by $\calb(\Pi)$ the (finite) set of cuspidal automorphic representations $\pi$ of $G(\A_F)$ such that $\Pi=BC(\pi)$. We have the following comparison of trace formulas.
\begin{Prop}\label{Prop: simple character id}
 For almost all split places $v$, we fix an irreducible unramified representation $\pi_v^0$. Additionally, for a fixed split place $v_0$, we fix a supercuspidal representation $\pi_{v_0}$ of $G(F_{v_0})$. Then there exists at most one cuspidal automorphic representation $\Pi$ of $\G'(\A_F)$ such that if $f\in C_c^\infty(G(\A_F))$ and $\{f'_{x,y}\}_{(x,y)\in\calv_n\times \calv_{n+1}}$ are efficient transfers for $(\omega, \pi_{v_0})$, then
\begin{equation*}
\frac{1}{L(1,\eta)^2}\sum_{x,y}J_{\Pi}^{x,y}(f'_{x,y})=\sum_{\pi\in \calb(\Pi)}I_\pi(f),
\end{equation*}
where the sum on the left runs over all $(x,y)\in \calv_n\times \calv_{n+1}$ while the sum on the right runs over all (cuspidal) automorphic representations $\pi$ of $G(\A_F)$ such that
\begin{enumerate}
\item\label{property1}$\pi_v\cong \pi_v^0$ for almost all split $v$, 
\item\label{property2} $\pi_{v_0}$ is our fixed supercuspidal representation,
\end{enumerate}
and where $\Pi=BC(\pi)$ is the base change of any $\pi$ appearing in the sum.
\end{Prop}
\begin{Rem}
    Note that the existence of $\Pi$ depends only on whether the set of cuspidal automorphic representations $\pi$ of $G(\A_F)$ satisfying \eqref{property1} and \eqref{property2} is non-empty.
\end{Rem}
\begin{proof}
Let $f$ and $\{f'_{x,y}\}_{(x,y)\in\calv_n\times \calv_{n+1}}$ be efficient transfers as in the statement of the proposition. We may assume that all test functions are factorizable. Let $S\supset S_\infty$ be a finite set of places containing all infinite places such that all Hermitian spaces $V_x$ and $V_y$ with $f'_{x,y}\neq 0$ are unramified outside $S$. Enlarging $S$ if necessary, we may assume that
\begin{enumerate}
\item\label{inert} for any inert place $v\notin S$, $f_v$ and $f'_{x,y,v}$ are units of the spherical Hecke algebras. These match by the fundamental lemma (Theorem \ref{Thm: JRL fundamental lemma});
\item\label{split} for any split place $v\notin S$, $f'_{x,y,v}\in \calh_{{K}_v'}(G'(F_v))$ matches $f_v\in \calh_{{K}_v}(G(F_v))$ in the sense of Proposition \ref{Prop: split transfer} . More precisely, if we write $f'_{x,y,v}=f_1\otimes f_2\in \calh_{K_v'}(G'(F_v))$ with respect to the identification $G'(F_v)\simeq \G(F_v)\times G(F_v)$, then the function $f_1\ast f_2^\theta\in \calh_{K_v}(\G(F_v))$ matches $f'_{x,y,v}$ as a special case of Proposition \ref{Prop: split transfer}. 
\end{enumerate}
 Write $f=f_S\otimes f^S$, where $f^S\in \calh_{{K}^S}(G(\A_F^S))$, where $\A_F^S=\prod_{v\notin S}F_v$ and ${K}^S=\prod_{v\notin S}K_v$; similarly, we write $f_{x,y}=f'_{x,y,S}\otimes f_{x,y}^S$ with $ f_{x,y}^S\in \calh_{{K'}^S}(G'(\A_F^S))$, where ${K'}^S=\prod_{v\notin S}K_v'$.

With these notations, the matching of global orbital integrals (\ref{global orbital match}) combine with the geometric expansions in Propositions \ref{Prop: twisted RTF} and \ref{Prop: unitary RTF} to give the identity
\[
\frac{1}{L(1,\eta)^2}\sum_{x,y}J_{\omega'}^{x,y}(f'_{x,y,S}\otimes f_{x,y}^S)=\sum_{(i,j)}I_{\omega_{i,j}}(f_S\otimes f^S).
\]
Applying the spectral sides of Propositions \ref{Prop: twisted RTF} and \ref{Prop: unitary RTF}, we obtain the identity
\begin{equation*}
\frac{1}{L(1,\eta)^2}\sum_{x,y}\sum_{\Pi}J_{\Pi}^{x,y}(f'_{x,y,S}\otimes f_{x,y}^S)=\sum_{\pi}I_\pi(f_S\otimes f^S),
\end{equation*}
where $\Pi$ and $\pi$ run over cuspidal automorphic representations with the prescribed central characters and supercuspidal component at $v_0$. For the unramified representations $\Pi^S$ (resp. $\pi^S$), let $\lam_{\Pi^S}$ (resp. $\lam_{\pi^S}$) be the Hecke-trace functionals of $\calh_{{K'}^S}(G'(\A_F^S))$ (resp. $\calh_{{K}^S}(G(\A_F^S))$). Then we observe (cf. \cite[proof of Proposition 2.10]{ZhangFourier}) that
\[
I_\pi(f_S\otimes f^S)=\lam_{\pi^S}(f^S)I_\pi(f_S\otimes \bfun_{K^S}),
\]
and 
\[
J_{\Pi}^{x,y}(f'_{x,y,S}\otimes f_{x,y}^S)=\lam_{\Pi^S}(f_{x,y}^S)J_{\Pi}^{x,y}(f'_{x,y,S}\otimes\bfun_{{K'}^S}).
\]
Since we only allow non-identity elements of the local Hecke algebras at S or places of $F$ that split in $E$, we may view the above two equations as identities of linear functionals on the Hecke algebra $\calh_{{K'}^{S,split}}(G'(\A_F^{S,split}))$, where the superscript $split$ indicates that we only take the product over the split places outside of $S$. To see this, note that the match $f'_{x,y,v}=f_1\otimes f_2\in \calh_{K_v'}(G'(F_v))$ with the function $f_1\ast f_2^\theta\in \calh_{K_v}(\G(F_v))$ as a special case of Proposition \ref{Prop: split transfer} induces a linear map
\[
\calh_{{K'}^{S,split}}(G'(\A_F^{S,split}))\lra \calh_{{K}^{S,split}}(G(\A_F^{S,split}))
\]
which we compose with $I_\pi(f_S\otimes -)$ to obtain a linear functional on $\calh_{{K'}^{S,split}}(G'(\A_F^{S,split}))$.

By the infinite linear independence of Hecke characters (see \cite[Appendix]{BadulescuJRnotes} for a short proof), for any fixed $\otimes_v \pi^0_v$ we obtain the sum
\[
\frac{1}{L(1,\eta)^2}\sum_{x,y}\sum_\Pi J_{\Pi}^{x,y}(f'_{x,y})=\sum_{\pi\in \calb}I_\pi(f),
\]
where $\calb$ is the set of cuspidal automorphic representations satisfying \eqref{property1} and \eqref{property2}, and where $$\Pi\in \{\Pi:  \text{for almost all split primes, }\Pi_v= BC(\pi_v)\text{ for some }\pi\in \calb\}.$$ Applying Theorem \ref{ramakrishnan2018theorem}, we see that there is at most one representation appearing on the left-hand side. Furthermore, this implies that $\calb=\calb(\Pi)$.
\end{proof}

We now fix a pair $(x,y)\in X_n\times X_{n+1}$ and obtain a comparison of relative characters which is compatible with the factorizations in Propositions \ref{Prop: twisted factorization} and \ref{Prop: linear factorization}.

\begin{Thm}\label{Prop: main global tool}
Suppose that $E/F$ is a quadratic extension of number fields such that every archimedean place $v|\infty$ of $F$ splits in $E$. Fix $(x,y)\in X_n\times X_{n+1}$. Let $\Pi$ be a cuspidal automorphic representation of $G'$ with central character $\omega'$ such that
\begin{enumerate}
    \item $\Pi\cong\Pi^\sig$, and
    \item there is a split place $v_0$ and a supercuspidal representation $\pi_0$ of $G(F_{v_0})$ such that $\Pi_{v_0}\simeq \pi_0\otimes\pi_0$ is the (supercuspidal) base change of $\pi_0$.
\end{enumerate} Consider a $\omega'$-nice factorizable function $f'\in C_c^\infty(G'(\A_F))$ satisfying that $f_{v_0}$ is essentially a matrix coefficient of $\Pi_{v_0}$. There exists an $\omega$-nice factorizable function $f\in C_c^\infty(G(\A_F))$ matching $\{f'_{x',y'}\}_{(x',y')}$, where
\[
f'_{x',y'} =\begin{cases}
    f'&: (x',y') = (x,y),\\
    0&: \text{otherwise,}
\end{cases}
\] and $(f,\{f'_{x',y'}\})$  are efficient transfers for $(\omega,\pi_{v_0})$. We have the identity
\begin{equation}\label{eqn: global spectral transfer}
\frac{1}{L(1,\eta)^2}J_{\Pi}^{x,y}(f')=\sum_{\pi\in \calb(\Pi)}I_\pi(f).
\end{equation}
 If $f' = \bigotimes_vf'_{v}$ and $f=\bigotimes_vf_v$ where the pairs $(f'_{v}, f_v)$ match for each place $v$, then $f$ may be chosen so that
\begin{equation}\label{eqn: super important}
\prod_vJ^{x,y,\natural}_{\Pi_v}(f'_{v}) =\prod_vI^\natural_{\pi_v}(f_v),
\end{equation}
where $\pi=\otimes_v'\pi_v\in \calb(\Pi)$.
\end{Thm}
\begin{proof}
That such a transfer $f\in C_c^\infty(G(\A_F))$ exists follows from Proposition \ref{Prop: split transfer}, Theorem \ref{Thm: smooth transfer}, and the properties of the $Z$-regular semi-simple loci. We now apply the previous proposition to the unramified representation $\otimes_v\pi_v^0$ where $v$ runs over those split places over which $\Pi$ is unramified and $\pi_v^0$ is determined by
\[
\Pi_v=BC(\pi_v^0) \cong \pi_v^0\otimes \pi_v^0.
\]
This gives (\ref{eqn: global spectral transfer}).

Now fix $\pi\in \calb(\Pi)$. Let $\eta_{i,j}:G(\A_F)\to \cc^\times$ be the characters 
\[
\eta_{i,j}(g_1,g_2)=\eta(\det(g_1))^i\eta(\det(g_2))^j.
\]
Note that 
\[
\calb(\Pi)=\{\pi\cdot\eta_{i,j}:(i,j)\in \{0,1\}^2\},
\]
where if $\pi = \pi_n\boxtimes \pi_{n+1},$ then 
\[
\pi\cdot\eta_{i,j} := \pi_n\cdot(\eta^i\circ\det)\boxtimes\pi_{n+1}\cdot(\eta^j\circ\det).
\]
By Lemma \ref{Lem: reduction of support} and our assumptions on the global extension $E/F$, we may assume that for each place $v$ of $F$,
\[
\supp(f_v) \subset G_v[{x,y}]:=G_{F_v}[x,y].
\]
Since the two Hermitian forms $x$ and $y$ are global, this implies that $f=\eta_{i,j}\cdot f$ for any $(i,j)\in \{0,1\}^2$.

Considering the local distribution $I_{\pi_v}$, we have
\[
I_{\pi_v\eta_{v,i,j}}(f_v) =I_{\pi_v}(f_v\cdot\eta_{v,i,j}).
\]
Combining this with the product formula (\ref{eqn: Bessel factors}) implies that
\[
\sum_{(i,j)}I_{\pi\eta_{i,j}}(f) =\sum_{(i,j)}I_{\pi}(f\cdot\eta_{i,j})=4I_\pi(f).
\]
Thus, the global matching of relative characters becomes
\[
\frac{1}{L(1,\eta)^2}J_{\Pi}^{x,y}(f')=4I_\pi(f)
\]
whenever $f'$ and $f$ are matching functions as in the proposition. Combining this global identity with the product formulas \eqref{eqn: twisted Bessel factors} and \eqref{eqn: Bessel factors} gives \eqref{eqn: super important}, where we use the volume calculation
\[
\frac{\vol(E^\times \backslash\A_E^1)}{L(1,\eta)}=\vol(F^\times \backslash\A_F^1).\qedhere
\]
\end{proof}

\section{Weak transfer of local relative characters}\label{Section: weak comparison}

In this section, we show that Theorem \ref{Prop: main global tool} implies a weak form of the local transfer of relative characters for matching test functions. Here ``weak form'' means that our results only apply to certain representations $\pi$. This is sufficient for our final (geometric) goal. 
 
\subsection{Split places and non-vanishing under regular support conditions} 
In the global comparison, we imposed certain support conditions at a single place $v$ of our number field in order to affect a simple trace formula. As we are only making the regular semi-simple support assumption at split places, the local distributions are precisely the ones discussed in \cite[Appendix A]{ZhangFourier}. 
This allows for the following non-vanishing result.

\begin{Lem}\label{Lem: supercuspidal non-vanishing}
Assume that $F$ is a non-archimedean local field. Suppose that $\pi$ is a supercuspidal representation of $G(F)$ with central character $\omega$. Then  there exists a matrix coefficient $\Phi$ of $\pi$, and a test function $f\in C_c^\infty(G(F))$ such that
\begin{itemize}
\item $f_\omega(g) = \displaystyle\int_{Z_G(F)}f(gz)\omega(z)dz =\Phi(g)$ for all $g\in G(F)$, and 
\item there exists a $Z$-regular semi-simple element $\ga$ such that $\Orb_\omega(f,\ga)\neq0.$
\end{itemize}
\end{Lem}
\begin{Rem}
    The orbital integral above is a special case of the integrals $\Orb^\eta_\omega$ in \eqref{eqn: twisted OI final}. We are here considering a place of $F$ which splits in $E$, so $\eta=1$.
\end{Rem}
\begin{proof}
The first requirement follows from the surjectivity of the map
\begin{align*}
C_c^\infty(G(F))&\lra C_c^\infty(Z_G(F)\backslash G(F),\omega^{-1})\\
f&\longmapsto f_\omega.
\end{align*}
Since $\pi$ is supercuspidal, any matrix coefficient $\Phi$ lies in $C_c^\infty(Z_G(F)\backslash G(F),\omega)$, so there exists an $f$ satisfying $f_\omega=\Phi$.

Recall now that for any generic representation $\pi$, the relative character $I_{\pi}^\natural$ is a non-zero distribution \cite{JPSS}. For simplicity, we work instead with the unnormalized distribution $I_\pi$. Since the pair $(G,H)$ is a strongly-tempered spherical pair, a theorem of Sakellaridis and Venkatesh \cite[Section 6]{SakVenkperiods} tells us that there exists a vector $W_0\in \calw(\pi)$ such that the local Rankin-Selberg period $\lam_{\pi}$ may be expressed as 
\[
\lam_\pi(W) = \int_{H(F)}[\calw(\pi,h)W,\hat{W}_0]_\pi dh.
\] 
With this, define the matrix coefficient 
\[
\Phi_0(g) =[\calw(\pi,g)W_0,\hat{W}_0]_\pi.
\]
Ichino and Zhang show in \cite[Appendix A]{ZhangFourier} that $\Phi_0$ satisfies the properties that the integral
\[
\Orb(\Phi_0,\ga) = \int_{H(F)}\int_{H(F)}\Phi_0(h^{-1}_1\ga h_2)dh_1dh_2
\]
is convergent on a subset $G(F)_{con}\subset G(F)$ the compliment of which has measure zero. Moreover, this orbital integral is non-zero on a subset of positive measure. In particular, since $G(F)^{Z-rss}$ is Zariski open and dense, there exists an element $\ga\in G(F)^{Z-rss}$ such that
\begin{equation}\label{sufficient}
\Orb(\Phi_0,\ga) \neq 0.
\end{equation}
Lemma \ref{Lem: supercuspidal non-vanishing} now follows from the following lemma and Theorem A.2 of \cite{ZhangFourier}, which states that there is a function $f'\in C^\infty_c(G(F)^{Z-rss})$ such that $I_\pi(f')\neq0$.
\begin{Lem}\cite[Lemma A.5]{ZhangFourier}
The function $g\mapsto \Orb(\Phi_0,g)$ on $G(F)$ is locally $L^1$ and for any $f\in C_c^\infty(G(F))$, we have
\[
I_\pi(f) = \int_{G(F)}f(g)\Orb(\Phi_0,g)dg.
\]
\end{Lem}
In particular, if $f'\in C^\infty_c(G(F)^{Z-rss})$ such that $I_\pi(f')\neq0$, the lemma implies that we cannot have $\Orb(\Phi_0,-)|_{G(F)^{Z-rss}}\equiv 0$. 

Now take $f\in C_c^\infty(G(F))$ such that $f_\omega=\Phi_0$ and let $\ga\in G(F)^{Z-rss}$ be an element satisfying (\ref{sufficient}). Since $\ga$ is semi-simple and $f$ has compact support, the orbital integral is absolutely convergent and we may rearrange the integration to find
\begin{align*}
\Orb_\omega(f,\ga)&=\int_{Z_G(F)}\Orb(f,\ga z)\omega(z)dz\\ &=\int_{H(F)}\int_{H(F)}\int_{Z_G(F)}f(h^{-1}_1\ga zh_2)\omega(z)dzdh_1dh_2\\&= \Orb(\Phi_0,\ga) \neq0. \qedhere
\end{align*}
\end{proof}

In particular, this ensures that when we work globally, there always exists global test functions $f$ with  $I_{\pi}^\natural(f)\neq0$ that also satisfies the assumptions of Lemma \ref{Lem: supercuspidal non-vanishing} at at least one finite place. To ensure that we have a similar non-vanishing statement for $J_{\Pi}^{x,y,\natural}$ under such a $Z$-regular support assumption, we give a direct local transfer of relative characters in the split case. For simplicity, we work with the unnormalized distributions $J_{\Pi}^{x,y}$ and $I_{\pi}$.

We continue to assume that $F$ is local and now set $E=F\times F$. As before, we fix isomorphisms $\GL_k(E)\cong \GL_k(F)\times \GL_k(F)$ such that our unitary groups are given by
\[
U(V_n)\cong\{(g,g^\theta)\in \GL_n(F)\times \GL_n(F): g\in \GL_n(F)\}
\]
and 
\[
U(V_{n+1})\cong\{(g,g^\theta)\in \GL_{n+1}(F)\times \GL_{n+1}(F): g\in \GL_{n+1}(F)\}.
\]
Set $J_{\Pi}:=J^{w_n,w_{n+1}}_{\Pi}$.

\begin{Prop}
Consider matching smooth functions $f_1\otimes f_2\in C_c^\infty(G(F)\times G(F))$ and $f=f_1\ast f_2^{\theta\vee}\in C_c^\infty(G(F))$. Then for any irreducible representation $\pi$ of $G(F)$,
\[
J_{BC(\pi)}(f_1\otimes f_2) = I_\pi(f).
\]
\end{Prop}
\begin{proof} let $w_0=(w_n,w_{n+1})\in G(F)$. Identifying $\Pi=BC(\pi) = \pi\boxtimes \pi$, (\ref{eqn: split FLO identity}) implies that for any ${W}',{W}''\in \calw(\pi)$
\[
\al_{(h,{}^th)}^{\pi}(W'\otimes W'') = \prod_{k\in\{n,n+1\}}\left[\calw(h,\pi_k)W_k',\calw(w_k,\hat{\pi}_k)\hat{W}_k''\right]_{\pi_k},
\]
where $\pi=\pi_n\boxtimes\pi_{n+1}$. For the purposes of computing $J_{\Pi}$, we note that for $\hat{W}',\hat{W}''\in \calw(\hat{\pi})$
\begin{equation}\label{helpful}
\al_{(w_0,w_0)}^{\hat{\pi}}(\hat{W}'\otimes \hat{W}'') =  \prod_{k\in\{n,n+1\}}\left[\calw(w_k,\pi_k)W_k',\calw(w_k,\hat{\pi}_k)\hat{W}_k''\right]_{\pi_k}=\left[W'',\hat{W}'\right]_{\pi}.
\end{equation}
Thus, we have
\begin{align*}
J_{BC(\pi)}(f_1\otimes f_2)&=\sum_{W'\boxtimes W''}\frac{\lam_\pi(\pi(f_1)W')\lam_\pi(\pi(f_2)W'')\al_{(w_0,w_0)}^{\hat{\pi}}(\hat{W}'\otimes \hat{W}'') }{[W',\hat{W}']_\pi[W'',\hat{W}'']_\pi}\\
&=\sum_{W'}\frac{\lam_\pi(\pi(f_1)W')\lam_\pi(\pi(f_2)W') }{[W',\hat{W}']_\pi},
\end{align*}
where we use (\ref{helpful}) and the fact that we are summing over an orthogonal basis to reduce to the sum over a single basis element $W'\in \calw^{\psi}(\pi)$. We now claim that
\[
\lam_\pi(\pi(f_2)W')=\lam_{\hat{\pi}}(\hat{\pi}(f^\theta_2)\hat{W}').
\]
This follows from the fact that
\[
\hat{\pi}(f^\theta)\hat{W}(h) = \int_{G(F)}f^\theta(g)\hat{W}(hg)dg = \int_{G(F)}f(g^\theta){W}(h^\theta g^\theta)dg = \pi(f)W(h^\theta),
\]
and that the change of variables $h\mapsto h^\theta$ is unimodular. Applying this, we obtain
\begin{align*}
\sum_{W'}\frac{\lam_\pi(\pi(f_1)W')\lam_{\hat{\pi}}(\hat{\pi}(f^\theta_2)\hat{W}') }{[W',\hat{W}']_\pi}
&=\sum_{W'}\frac{\lam_\pi(\pi(f_1\ast f_2^{\theta\vee})W')\lam_{\hat{\pi}}(\hat{W}') }{[W',\hat{W}']_\pi}= I_\pi(f).
\end{align*}
\end{proof}

\subsection{Unramified case}
We now consider the case that $E/F$ is an unramified extension of non-archimedean local fields. Let $\calh_{K'}(G'(F))$ denote the spherical Hecke algebra for $G'(F)$ and let $\calh_K(G(F))$ the corresponding algebra for $G(F)$. We have the morphism
\[
BC:\calh_{K'}(G'(F))\to\calh_K(G(F)),
\]
defined by $Sat(BC(\varphi))(\pi) = Sat(\varphi)(BC(\pi))$ for any spherical representation $\pi$ of $G(F)$. This morphism is injective. 
\begin{Rem}
    Our notation for the Satake transform $Sat(\varphi)$ here is inconsistent with the notation from Section \ref{Section: morphism Satake}. To relate them, let $Sat_k$ denote the transform from Section \ref{Section: morphism Satake} on $\GL_k(E)$, with $k\in \{n,n+1\}$.  If $\Pi = \Pi_n\boxtimes\Pi_{n+1}$ is an unramified representation of $\G'(F)$, the two notions are related by evaluating $\Pi$ at its Satake parameters. More precisely, if $\varphi = \varphi_n\otimes \varphi_{n+1}$ with $\varphi_k\in \calh_{K_{k,E}}(\GL_k(E))$, then 
    \[
    Sat(\varphi)(\Pi) = Sat_n(\varphi_n)(s_1,\ldots,s_n)Sat_{n+1}(\varphi_{n+1})(s_1',\ldots, s_{n+1}'),
    \]
    where $\pi_n$ is the spherical representation associated to the unramified principle series induced from 
    \[
   \left(\begin{array}{ccc}
a_1 &    &  \\
&\ddots     & \\
&&a_n
\end{array}\right)\in T_n(E)\longmapsto \prod_{i=1}^n|a_i|_E^{s_i-\frac{1}{2}(n+1-2i)},
    \]
    and similarly for $\pi_{n+1}$ with the variables $(s_1',\ldots, s_{n+1}')$.
\end{Rem}

\begin{Lem}\label{Lem: unramified spectral transfer}
Let $f'\in \calh_{K'}(G'(F))$ and let $(x,y)\in G'(F)\ast(w_n,w_{n+1})$, where $\ast$ denotes the action on $X_n\times X_{n+1}$. For any representation $\pi$ of $G(F)$, we have
\[
J^{x,y,\natural}_{BC(\pi)}(f') = I^\natural_\pi(BC(f')).
\]
\end{Lem}
\begin{proof}
When $\pi$ is not unramified, both sides are zero so we can assume $\pi$ is unramified. Choosing a basis containing the normalized spherical vector, we see that $f'\in\calh_{K'}(G'(F))$ acts by projecting onto the unramified line, on which it acts by the spherical eigenvalue $Sat(f')(BC(\pi))$. 

Recalling the normalizations of the various functionals on the normalized spherical Whittaker vector (\ref{eqn: normalized inner product}),(\ref{spherical Rankin}), and (\ref{spherical FLO}), the left hand side reduces to a single term given by
\[
J^{x,y,\natural}_{BC(\pi)}(f')=Sat(f')(BC(\pi))\frac{\lam^\natural_{BC(\pi)}(W_0^{BC(\pi)})\al_{x,y}^{\pi,\natural}(W_0^{BC(\pi)})}{[W_0^{BC(\pi)},\hat{W}_0^{BC(\pi)}]_{BC(\pi)}^\natural} = Sat(f')(BC(\pi)),
\]
where $W_0^{BC(\pi)}$ is the normalized spherical Whittaker function for $BC(\pi)$ on which
\[
\lam^\natural_{BC(\pi)}(W_0^{BC(\pi)})=\al_{x,y}^{\pi,\natural}(W_0^{BC(\pi)})=[W_0^{BC(\pi)},\hat{W}_0^{BC(\pi)}]_{BC(\pi)}^\natural=1.
\] A similar argument shows
\[
I^\natural_\pi(BC(f'))=Sat(BC(f'))(\pi)\frac{\lam_\pi^\natural(W_0^\pi)\lam_\pi^{\eta,\natural}(W_0^\pi)}{[W_0^{\pi},\hat{W}_0^{\pi}]_{\pi}^\natural}=Sat(BC(f'))(\pi),
\]
where $W_0^{\pi}$ is the normalized spherical Whittaker function for $\pi$.  The result follows from the definition of the base change homomorphism $BC$.
\end{proof}

\subsection{Weak transfer of relative characters}

For non-split places more generally, the global Theorem \ref{Prop: main global tool} implies the following weak local spectral transfer of relative characters. %
\begin{Prop}\label{Prop: weak transfer}
Assume that $E/F$ is a quadratic extension of number fields such that every archimedean place $v|\infty$ of $F$ splits in $E$. Let $\Pi=BC(\pi)$ be an irreducible cuspidal automorphic representation of $G'(\A_F)$ and $(x,y)\in X_n\times X_{n+1}$ such that there exists a \emph{nice} test function $f'$ such that $$J_{\Pi}^{x,y}(f')\neq 0.$$

Then for any non-split place $v_0$ of $F$, there exists a non-zero constant $C(\Pi_{v_0},x,y)\in \cc^\times$ depending only on $(x,y)$ and $\Pi_{v_0}$ such that for any matching function $f_{v_0}\in C_c^\infty(G(F_{v_0}))$ and $\{f'_{x',y'}\}_{(x',y')}$, where
\[
f'_{x',y'} =\begin{cases}
    f'_{v_0}&: (x',y') = (x,y),\\
    0&: \text{otherwise,}
\end{cases}
\]we have
\begin{equation*}
J_{\Pi_{v_0}}^{x,y,\natural}(f'_{v_0})=C(\Pi_{v_0},x,y)I^\natural_{\pi_{v_0}}(f_{v_0}).
\end{equation*}
\end{Prop}
\begin{proof}
Let $\A_F^{v_0}$ denote the adeles away from the place $v_0$ and let ${f'}^{v_0}= \prod_{v\neq v_0}f'_{v}\in C_c^\infty(G'(\A_F^{v_0})$ be a factorizable test function. Using the factorization (\ref{eqn: twisted Bessel factors})  we have the equality
\[
J_{\Pi}^{x,y}(f'_{v_0}\otimes {f'}^{v_0}) = CJ_{\Pi_{v_0}}^{x,y,\natural}(f'_{v_0}).
\]
Since $J_{\Pi}^{x,y}\not\equiv 0$, we may choose ${f'}^{v_0}$ so that $C\neq0$. Moreover, since the distribution is non-vanishing for nice test functions, we know that there is a finite split place $v_1$ (necessarily distinct from $v_0$) such that $\Pi_{v_1}$ is supercuspidal. We may assume that $f'_{v_1}\in C_c^\infty(G'(F_{v_1}))$ is essentially a matrix coefficient of $\Pi_{v_1}$. Additionally, we may impose that there exists a second split place $v_2$ such that the local test function $f'_{v_2}$ is supported in the $Z$-regular semi-simple locus. In particular, we know that 
\begin{itemize}
\item the test function $f'_{v_0}\otimes {f'}^{v_0}$ is nice, and 
\item $C=4\frac{L\left(\frac{1}{2},\Pi_n\times \Pi_{n+1}\right)}{\Res_{s=1}(L(s,\pi\times \pi^\vee))}\prod_{v\neq v_0}J^{x,y,\natural}_{\Pi_v}(f'_{v})\neq0$.
\end{itemize}
Now by Theorem \ref{Prop: main global tool}, there exists a factorizable test function $f^{v_0}=\prod_{v\neq v_0}f_v\in C_c^\infty(G(\A_F^{v_0}))$ such that for any function $f_{v_0}$ as in the statement of the proposition matching $f'_{v_0}$, the test function $f=f_{v_0}\otimes f^{v_0}$ is nice and 
\[
\frac{1}{L(1,\eta)^2}J_{\Pi}^{x,y}(f')=I_\pi(f).
\]
Since we chose ${f'}^{v_0}$ such that $C\neq0$, the factorization (\ref{eqn: Bessel factors}) implies that there is a non-zero constant $C'$ such that 
\begin{equation*}
CJ^{x,y,\natural}_{\Pi_{v_0}}(f'_{v_0})=J_{\Pi}^{x,y}(f')=4I_\pi(f)=C'I^\natural_{\pi_{v_0}}(f_{v_0}).
\end{equation*}
Since the initial test function $f'_{v_0}$ was arbitrary, the constant $$C(\Pi_{v_0},x,y) = C^{-1}C'\neq0$$ is independent of functions $f'_{v_0}$ and $f_{v_0}$, finishing the proof.
\end{proof}

Combining this with our unramified computation, we have the following corollary for unramified representations.
\begin{Cor}\label{Cor: unramified spectral transfer}
Let all notations be as in the previous proposition. If $\Pi_{v_0}$ is unramified and $(x,y)\in G'(F)\ast(w_n,w_{n+1})$, then $C(\Pi_{v_0},x,y)=1$.
\end{Cor}
\begin{proof}
By the proposition, for any pair of matching functions $f'_{v_0}\in C_c^\infty(G'(F_{v_0}))$ and $f_{v_0}\in C_c^\infty(G(F_{v_0}))$ we have
\begin{equation*}
J_{\Pi_{v_0}}^{x,y,\natural}(f'_{v_0})=C(\Pi_{v_0},x,y)I^\natural_{\pi_{v_0}}(f_{v_0}),
\end{equation*}
for some $C(\Pi_{v_0},x,y)\in\cc^\times$.

Assume now that $\Pi_{v_0}$ is unramified, and recall that the fundamental lemma (Theorem \ref{Thm: JRL fundamental lemma}) states that we can take $f'_{v_0}=\bfun_{K'_{{v_0}}}$ and $f_{v_0}=\bfun_{K_{v_0}}$. As these functions lie in the spherical Hecke algebras of the two groups and $BC(f'_{v_0})= f_{v_0}$, Lemma \ref{Lem: unramified spectral transfer} implies that 
\[
J_{\Pi_{v_0}}^{x,y,\natural}(f'_{v_0})=I^\natural_{\pi_{v_0}}(f_{v_0})=1\neq0.
\] It follows that $C(\Pi_{v_0},x,y)=1$.
\end{proof}

\section{A base change fundamental lemma and the proof of Theorem \ref{Thm: full fundamental lemma}}\label{Section: final section}
The following application of our local and global spectral results will suffice to prove Theorem \ref{Thm: JR fundamental lemma for algebra}.

\begin{Thm}\label{Thm: base change Twisted Jacquet--Rallis}
Let $E/F$ be an unramified extension of $p$-adic local fields. For any $\varphi\in \calh_{K'}(G'(F))$, the function $BC(\varphi)\in \calh_K(G(F))$ matches the functions $\{\phi_{x,y}\}_{(x,y)\in \calv_n\times \calv_{n+1}}$ where
\[
\phi_{x,y}=\begin{cases} \quad\varphi&:(x,y)=(w_n,w_{n+1}),\\\quad0&:\text{otherwise}.\end{cases}
\]
More precisely, for any $\ga\in G(F)^{rss}$ if $\ga\xleftrightarrow{w_n,w_{n+1}}\de$, then
\begin{equation}\label{goal}
\Omega(\ga)\Orb^{\eta}(BC(\varphi),\ga)=\Orb(\varphi,\de),
\end{equation}
 and $\Orb^\eta(BC(\varphi),\ga)=0$ if $\ga\xleftrightarrow{x,y}\de$ for some $(x,y)$ not in the $G'(F)$-orbit of $(w_n,w_{n+1})$.
\end{Thm}

 Note that for $\varphi\in\calh_{K'}(G'(F))$, we automatically have $\supp(BC(\varphi))\subset G[w_n,w_{n+1}]$. In particular, if $\ga\xleftrightarrow{x,y}\de$ for some $(x,y)$ not in the $G'(F)$-orbit of $(w_n,w_{n+1})$, then 
\[
\Orb^\eta(BC(\varphi),\ga)=0,
\]giving the vanishing statement of the theorem. We may thus focus on matching pairs $\ga\xleftrightarrow{w_n,w_{n+1}}\de$ of regular semi-simple elements. 

The idea of the proof is to reduce this statement to the spectral transfer in Lemma \ref{Lem: unramified spectral transfer}. To make this more precise, we begin by reducing the theorem to a matching of orbital integrals that arise in the geometric expansions of relative trace formulas \eqref{eqn: linear RTF} and \eqref{eqn: twisted RTF}.
\begin{Lem}
Let $E/F$ be an unramified extension of $p$-adic local fields, and let $\varphi\in \calh_{K'}(G'(F))$. Suppose that for any unitary central character $\omega:Z_G(F)\to \cc^\times$ with base change $\omega'$ and any matching $Z$-regular semi-simple orbits $\ga\xleftrightarrow{w_n,w_{n+1}}\de$, 
\begin{equation}\label{eqn: reduce to integrals with central character}
 \Omega(\ga)\Orb_\omega^{\eta}(BC(\varphi),\ga) =\Orb_{\omega'}(\varphi,\de)   
\end{equation}
and $\Orb_\omega^{\eta}(BC(\varphi),\ga)=0$ when $\ga\xleftrightarrow{x,y}\de$ for some $(x,y)$ not in the $G'(F)$-orbit of $(w_n,w_{n+1})$. Then Theorem \ref{Thm: base change Twisted Jacquet--Rallis} holds.
\end{Lem}
\begin{proof}
    To begin, we claim that it suffices to prove (\ref{goal}) for $Z$-regular semi-simple classes. Indeed, the equality reduces via  (\ref{eqn: linear OI reduction}) and (\ref{eqn: twisted OI reduction}) to 
\begin{equation*}
\Orb^{U(V_n)}(\widetilde{\varphi},\pi_{w_n,w_{n+1}}(\de)) = \omega(\ga_1^{-1}\ga_2)\Orb^{\GL_n(F),\eta}(\widetilde{BC(\varphi)},\ga_1^{-1}\ga_2),
\end{equation*}
where $\ga=(\ga_1,\ga_2).$ The orbital integrals here are precisely those arising in Jacquet--Rallis transfer (\ref{eqn: regular matching final}), implying that they are locally constant on the regular semi-simple locus \cite[Lemma 3.12]{ZhangFourier}.

Stated in terms of the categorical quotient
\begin{align*}
\cala'=\GL_{n+1}\sslash\GL_{n}\cong X_{n+1}\sslash\U(V_n)\cong \A^{2n+1},
\end{align*}
where the identification with $\A^{2n+1}$ is given by the invariant maps $c$ and $c^{w_n,w_{n+1}}$  as in \eqref{eqn: invariant map JR}, we may view
\[
\Phi_\varphi(x)=\Orb^{U(V_n)}(\widetilde{\varphi},\pi_{w_n,w_{n+1}}(\de)) - \omega(\ga_1^{-1}\ga_2)\Orb^{\GL_n(F),\eta}(\widetilde{BC(\varphi)},\ga_1^{-1}\ga_2)
\]as a smooth function on the regular semi-simple locus ${\cala'}^{rss}(F)$, where $c(\ga)=c^{w_n,w_{n+1}}(\de)=x$. Arguing as in the proof of Proposition \ref{Prop: second reduction}, it suffices to show that $\Phi_\varphi\equiv0$ on the open dense set
\[
\{(a_1,\ldots,a_{2n+1})\in {\cala'}^{rss}(F): a_1\neq0\:\text{and}\: a_{2n+1}\neq0\}.
\]
This is precisely the $Z$-regular semi-simple locus (see Section \ref{Section: dealing with center}), proving the claim. 

We may thus assume that $\ga\xleftrightarrow{w_n,w_{n+1}}\de$ are matching $Z$-regular semi-simple elements. Consider the function $\widetilde{O}_\varphi: Z_{G'}(F)\to \cc$ given by
\begin{align*}
\widetilde{O}_\varphi(z):=\widetilde{O}_\varphi^{\ga,\de}(z) &= \Orb(\varphi,\de z)-\Omega(\ga\Nm(z))\Orb^{\eta}(BC(\varphi),\ga\Nm(z))\\ &= \Orb(\varphi,\de z)-\Omega(\ga)\Orb^{\eta}(BC(\varphi),\ga\Nm(z)).
\end{align*}
Recalling that $\Orb(\varphi,z\de) = \Orb(\varphi,z'\de)$ if $\Nm(z) = \Nm(z')$ by Lemma \ref{Lem: Z-reg on unitary side}, we see that $\widetilde{O}_\varphi$ factors through the norm map
\[
\widetilde{O}_\varphi = {O}_\varphi\circ \Nm,
\]
to give a function ${O}_\varphi: \Nm(Z_{G'}(F))\to \cc$. For any unitary character $\omega: Z_{G}(F)\to S^1$, we consider the Fourier transform
\begin{align*}
    \hat{O}_{\varphi}(\omega) &= \int_{\Nm(Z_{G'}(F))}{O}_\varphi(z)\omega(z)dz\\
                            &=\Orb_{\omega'}(\varphi,\de)-\Omega(\ga)\Orb_\omega^{\eta}(BC(\varphi),\ga),
   \end{align*}
which is absolutely convergent. 

Since $\varphi$ is unramified, it follows that $\widetilde{O}_\varphi(uz)=\widetilde{O}_\varphi(z)$ for any $u\in Z_{G}(\calo_F)\cong (\calo_F^\times)^2$. In particular, $\hat{O}_{\varphi}(\omega) = 0$ unless $\omega$ is unramified, so that there exist unique real numbers 
\[
0\leq t_1,t_2< \frac{\pi}{\log(q)}
\]
such that if $(z_1,z_2)\in \Nm(Z_{G'}(F))$, then $\omega(z_1,z_2) = |z_1|_F^{it_1}|z_2|_F^{it_2}$. By Pontryagin duality on the discrete group $\Nm(Z_{G'}(F))/Z_{G}(\calo_F)$, we obtain
\[
{O}_\varphi(z)=\int_{T_q}\hat{O}_{\varphi}(\omega)\omega(z)^{-1}d\omega,
\]
where $T_q\simeq (S^1)^2$ is the compact group of unramified unitary characters, and $d\omega$ is the appropriate Haar measure. This implies the lemma since if the matching \eqref{eqn: reduce to integrals with central character} holds for all $\omega$, then $\hat{O}_{\varphi}\equiv0$.
\end{proof}

It therefore suffices to fix an arbitrary unramified unitary character $\omega$ of $Z_G(F)$ and prove \eqref{eqn: reduce to integrals with central character}. We now outline the strategy. The first observation is that for any unramified representation $\pi$ of $G(F)$ with $\Pi=BC(\pi)$, Lemma \ref{Lem: unramified spectral transfer} tells us that for any $\varphi\in \calh_{K'}(\G'(F))$
\begin{equation}\label{eqn: tool for the win2}
J^{w_n,w_{n+1},\natural}_{\Pi}(\varphi)=I^\natural_{\pi}(BC(\varphi)).    
\end{equation}
The idea is to use the comparison of relative trace formulas of Theorem \ref{Prop: main global tool} to deduce the matching of orbital integrals in Theorem \ref{Thm: base change Twisted Jacquet--Rallis} from this spectral identity.

To do this, we first globalize the problem by fixing an extension of number fields $\cale/\calf$ with a fixed finite place $v_0$ such that $\cale_{v_0}/\calf_{v_0}\cong E/F$ and comparing global orbital integrals. This relies on an approximation argument (Lemma \ref{Lem: approx glob}) to relate local matching of orbital integrals to global orbits. Using our local results from Section \ref{Section: weak comparison}, we then construct a matching pair of test functions so that the entire geometric expansions of the distributions $J_{\omega'}$ and $I_\omega$ are non-zero and supported at our \emph{single regular semi-simple orbit} (see Proposition \ref{Prop: good global test function}). The precise comparison of factorizable distributions in Theorem \ref{Prop: main global tool} then relates the resulting global transfer of relative characters to local relative characters. Augmenting our global test functions at the single place $v_0$, Corollary \ref{Cor: unramified spectral transfer} and \eqref{eqn: tool for the win2} combine to imply the desired identity of orbital integrals.

\subsection{Proof of Theorem~\ref{Thm: base change Twisted Jacquet--Rallis}}
Fix $\varphi\in\calh_{K'}(G'(F))$ and let $\de\in G'(F)^{rss}$. As noted above, we may assume there exists a matching pair $\ga\xleftrightarrow{w_n,w_{n+1}}\de$ of $Z$-regular semi-simple elements.

We now globalize our quadratic extension. That is, we let $\cale/\calf$ be a quadratic extension of number fields such that every archimedean place of $\calf$ splits in $\cale$ and there exists a place $v_0$ of $\calf$ such that $\cale_{v_0}/\calf_{v_0}\cong E/F$. 
 We also set aside two distinct split places $v_{cusp}$ and $v_{reg}$.

 Fix unramified unitary character $\omega$ of $Z_G(F)$, and by an abuse of notation we let $\omega=\prod_v\omega_v$ be a global unitary character of $[Z_G]$ such that $\omega_{v_0}=\omega$ (it is clear that such an $\omega$ exists). Let $\omega'$ denote the base change of $\omega$ to $\cale$.
 
We now construct nice global matching test functions $f'$ and $f$ such that the distributions 
\[
J^{w_n,w_{n+1}}_{\omega'}(f')\:\text{ and }\: I_{\omega_{i,j}}(f)\qquad (\text{for }(i,j)\in \{0,1\}^2)
\] have particularly simple geometric expansions. Let $\pi_{v_{cusp}}$ be a supercuspidal automorphic representation of $G(\calf_{v_{cusp}})$. By Lemma \ref{Lem: supercuspidal non-vanishing}, we may find a test function $f_{v_{cusp}}$ which is essentially a matrix coefficient of $\pi_{v_{cusp}}$ and such that there exists a $Z$-regular semi-simple element $\ga_{cusp}$ such that\footnote{For ease of notation, we drop the superscripts indicating the place in notation for the orbital integrals throughout the proof. }
\[
\Orb^{\eta}_\omega(f_{v_{cusp}},\ga_{cusp})\neq0.
\]
Let $f'_{v_{cusp}}$ be essentially a matrix coefficient of $BC(\pi_{cusp})$ matching $f_{v_{cusp}}$; it is clear that such an $f'_{v_{cusp}}$ exists. Then we know
\[
\Orb_\omega(f'_{v_{cusp}},\de_{cusp}) = \Omega_{v_{cusp}}(\ga_{cusp})\Orb^{\eta}_\omega(f_{v_{cusp}},\ga_{cusp})\neq0,
\]
where $\de_{cusp}\leftrightarrow \ga_{cusp}$.

For the place $v_{reg}$, we want to choose $f_{v_{reg}}$ possessing non-vanishing $Z$-regular semi-simple orbital integrals such that $\supp(f_{v_{reg}})\subset G(\calf_{v_{reg}})^{rss}.$ Such functions certainly exist; for example, if we let $\pi_{reg}$ be a supercuspidal representation of $G(\calf_{v_{reg}})$, we can apply Lemma \ref{Lem: supercuspidal non-vanishing} to obtain an $\widetilde{f}_{v_{reg}}$ and a $Z$-regular semi-simple element $\ga_{reg}$ such that
\[
\Orb^{\eta}_\omega(f_{v_{reg}},\ga_{reg})\neq0.
\]
Multiplying  $\tilde{f}_{v_{reg}}$ by the indicator function of a compact open subset $Z\subset  G(\calf_{v_{reg}})^{rss}$ such that $$Z_G(\calf_{v_{reg}})H(\calf_{v_{reg}})\ga_{reg}H(\calf_{v_{reg}})\subset Z$$
gives such a function.
We now similarly obtain $f'_{v_{reg}}$ and $\de_{reg}$ matching $f_{v_{reg}}$ and $\ga_{reg}$ so that
\[
\Orb_{\omega'}(f'_{v_{reg}},\de_{reg}) = \Omega_{v_{reg}}(\ga_{reg})\Orb^{\eta}_{\omega}(f_{v_{reg}},\ga_{reg})\neq0.
\]

 To study orbital integrals at $\ga$ and $\de$ by global means, we first approximate these points by global elements. 
 \begin{Lem}\label{Lem: approx glob} Let $\ga\xleftrightarrow{w_n,w_{n+1}}\de$ be our matching pair of $Z$-regular semi-simple elements. There exist matching global $Z$-regular semi-simple elements 
$a\in G(\calf)^{rss}$ and $b\in G'(\calf)^{rss}$ such that
\[
\Orb_{\omega'}(f'_{v_{cusp}},b)=\Orb_{\omega'}(f'_{v_{cusp}},\de_{cusp})=\Orb_\omega^{\eta}(f_{v_{cusp}},\ga_{cusp})=\Orb_\omega^{\eta}(f_{v_{cusp}},a)\neq0,
\]
\[
\Orb_{\omega'}(f'_{v_{reg}},b)=\Orb_{\omega'}(f'_{v_{reg}},\de_{reg})=\Orb_\omega^{\eta}(f_{v_{reg}},\ga_{reg})=\Orb_\omega^{\eta}(f_{v_{reg}},a)\neq0,
\]
\[
\Orb_{\omega'}(\varphi,b)=\Orb_{\omega'}(\varphi,\de),
\]
and
\[
\Omega(a)\Orb^{\eta}_\omega(BC(\varphi),a)= \Omega(\ga)\Orb^{\eta}_\omega(BC(\varphi),\ga).
\]
 \end{Lem}
 \begin{proof}
    Since the diagonal embeddings
\[
G(\calf)\hra G(F)\times G(\calf_{v_{cusp}})\times G(\calf_{v_{reg}})
\]
and
\[
G'(\calf)\hra G'(F)\times G'(\calf_{v_{cusp}})\times G'(\calf_{v_{reg}})
\]
are dense, local constancy of the orbital integrals implies that we may find $a$ and $b$ matching such that the claimed identities hold.
 \end{proof} 
 In particular, to prove (\ref{goal}), it suffices to prove the equality with $\ga=a$ and $\de=b$. For this we utilize the comparison of relative trace formulas of Section \ref{Section: comparison}. 
 
 \begin{Prop}\label{Prop: good global test function}
 \begin{enumerate}
     \item There exist functions $f=\bigotimes_v f_v\in C_c^\infty(G(\A_{\calf}))$ and $f'=\bigotimes_v f'_v\in C_c^\infty(G'(\A_{\calf}))$ which are an efficient matching pair of functions for $(\omega, \pi_{v_{cusp}})$ such that
\begin{equation}\label{eqn: isolation matching}
I_\omega(f) = \Orb^{\eta}_\omega(f,a)= \Orb_{\omega'}(f',b)=\frac{1}{4L(1,\eta)^2}J_{\omega'}^{w_n,w_{n+1}}(f') \neq0,
\end{equation}
where $\eta=\eta_{\cale/\calf}$.
\item\label{Final change to prop} With $f$ and $f'$ as above, set \[\text{$\hat{f}=BC(\varphi)\otimes\bigotimes_{v\neq v_0}f_v$ and $\hat{f}'=\varphi\otimes\bigotimes_{v\neq v_0}f'_{v}$.}\]
Then $\hat{f}$ and $\hat{f}'$ are  nice functions satisfying 
\[
 \Orb^{\eta}_\omega(\hat{f},a')= \Orb_{\omega'}(\hat{f}',b')=0
\]
for any global match regular semi-simple elements $a'\xleftrightarrow{w_n,w_{n+1}} b'$ lying in distinct orbits from $a\xleftrightarrow{w_n,w_{n+1}} b$. In particular, we have
\[
I_{\omega}(\hat{f}) = \Orb_\omega^{\eta}(\hat{f},a)\:\:\text{   and    }\:\: \frac{1}{4L(1,\eta)^2}J_{\omega'}^{w_n,w_{n+1}}(\hat{f}')=\Orb_{\omega'}(\hat{f}',b).
\]
 \end{enumerate}
 \end{Prop}

 \begin{proof}
Let $S$ be a finite set of places of $\calf$ containing all infinite places and  the places $v_0$, $v_{cusp}$, and $v_{reg}$ such that for each $v\notin S$, $a\in K_v$ and $b\in K'_v$. For every $v\in S\setminus{\{v_{cusp},v_{reg}\}}$, select matching $f'_v$  and $f_{v}$ such that
\[
\Orb_{\omega'}(f'_{v},b) = \Omega(a)\Orb^{\eta}_\omega(f_v,a)\neq0.
\]
For each place $v\in S$, let $C_v\subset G(\calf_v)$ be a compact set containing the support of $f_v$ and assume that $C_{v_0}$ is large enough to contain the support of $BC(\varphi)$; set $$C=\prod_{v\in S}C_v\times \prod_{v\notin S}K_v\subset G(\A_{\calf}).$$
For all places $v\notin S$, we take $f_v= \bfun_{K_v}$ and $f'_{v}=\bfun_{K'_v}$ to be the unit spherical functions. In particular, the fundamental lemma Theorem \ref{Thm: JR fundamental lemma} implies that
\[
\Orb(f'_{v},b) = \Omega(a)\Orb^{\eta}(f_v,a)\neq0
\]
for all $v\notin S$. The non-vanishing follows since $f'_{v}=\bfun_{K'_v}$ is a non-negative function with $f'_v(b)\neq0$. 
\begin{Lem}
    If $z\in Z_{G'}(\calf_v)$ then 
\[
\Orb(\bfun_{K'_v},bz)\neq 0 \text{ if and only if }z\in Z_{G'}(\calo_v).
 \] In particular, if $\omega_v'$ is a unramified unitary central character, then
 \[
 \Orb_{\omega'}(\bfun_{K'_v},b) = \int_{\Nm(Z_{G'}(\calo_v))}  \Orb(\bfun_{K'_v},bz)\omega'_v(z)dz =  \Orb(\bfun_{K'_v},b)\neq0.
 \]
\end{Lem}
\begin{proof}
    By definition, $\Orb(\bfun_{K'_v},bz)\neq 0$ implies that there exist $h_1\in H'(\calf_v)$ and $h_2\in H_{w_n,w_{n+1}}(\calf_v)$ such that $h_1^{-1}bzh_2\in K'_v$. In particular, the invariant polynomials $c^{w_n,w_{n+1}}_i$ from Section \ref{Section: invariant polys for quot} take integral values at $bz$. But if $z=(z_1I_n,z_2 I_{n+1})$, the final polynomial scales by $z_2$, so that
    \[
    c^{w_n,w_{n+1}}_{2n+1}(bz) = z_2c^{w_n,w_{n+1}}_{2n+1}(b).
    \]
    This forces $z_2\in \calo_v$. Similarly, $ c^{w_n,w_{n+1}}_1(bz) = z_1 c^{w_n,w_{n+1}}_1(b)$. Note that the assumption of $Z$-regularity is precisely the assumption that 
    \[
    c^{w_n,w_{n+1}}_1(b),c^{w_n,w_{n+1}}_{2n+1}(b)\neq0.
    \]
    The lemma follows.
\end{proof}

Now set $f= \otimes_vf_v$ and $f'=\otimes_vf'_{v}$. By linearity, we may assume without loss of generality that 
\[
\supp(f)\subset G_\A[{w_n,w_{n+1}}]:=\prod'_{v}G_v[w_n,w_{n+1}]
\]Our choices ensure that $f$ and $f'$ are an efficient matching pair of functions for $(\omega, \pi_{v_{cusp}})$ and that 
\[
\Orb_{\omega'}(f',b) = \Orb^{\eta}_{\omega}(f,a)\neq0.
\]

We claim that we may augment $f$ and $f'$ so that this is the only non-vanishing global orbital integral for $f$. Our assumption on the support of $f_{v_{reg}}$ already reduces this to $Z$-regular semi-simple classes.

To see this, recall that the matching of orbits may be characterized by the invariant polynomials (\ref{eqn: invariant map JR}), denoted by $\{c_i\}_{i=1}^{2n+1}$ for $G$ and $\{c_i^{{w_n,w_{n+1}}}\}_{i=1}^{2n+1}$ for $G'$.  For each $i$, the image of our compact set $C\subset G(\A_{\calf})$ under $c_i$ gives a compact subset $c_i(C)\subset \A_{\calf}$. Since $\calf\subset \A_{\calf}$ is discrete, the intersection $c_i(C)\cap \calf$ is finite for each $i=1,\ldots, 2n+1$. Since $\supp(f)\subset C$, this implies that $\Orb^\eta(f,a')=0$ outside of a finite set of orbits
\[
Q_f\subset H(F)\backslash G(F)^{Z-rss}/\Nm(Z_{G'}(F))H(F)
\]

Let $S_{aux}$ be a set of $|Q_f|-1$ unramified places of $\calf$ disjoint from $S$, and fix a bijection between $S_{aux}$ and $Q_f-\{a\}$.
For each $v\in S_{aux}$, set $\widetilde{f}_{v} = \bfun_{G_{v}[l_v]}\cdot\bfun_{K_{v}}$ and $\widetilde{f}'_{v} = \bfun_{G'_{v}[l_v]}\cdot\bfun_{K'_{v}}$, where
\[
G_{v}[l_v]=\{g\in G(\calf_{v}): c_i(g)\in c_i(a)+\mathfrak{p}_{v}^{l_v}, \text{ for all $i=1,\ldots,2n+1$}\}
\]
and
\[
G'_{v}[l_v]=\{g\in G'(\calf_{v}): c^{w_n,w_{n+1}}_i(g)\in c^{w_n,w_{n+1}}_i(b)+\mathfrak{p}_{v}^{l_v}, \text{ for all $i=1,\ldots,2n+1$}\}.
\]
As the polynomials $c_i$ are the invariant polynomials of $H'(\calf_{v})\times H'(\calf_{v})$ acting on $G(\calf_{v})$ and $c^{w_n,w_{n+1}}_i$ are the invariant polynomials of $H'(\calf_{v})\times H'_{x,y}(\calf_{v})$ acting on $G'(\calf_{v})$, we see that $\widetilde{f}_{v}$ and $\widetilde{f}'_{v}$ match for any choice of $l_v\in \zz_{\geq0}$ and that
\[
\Orb_{\omega'}(\widetilde{f}'_{v},b) = \Omega(a)\Orb^{\eta}_\omega(\widetilde{f}_{v},a)\neq0.
\]

Recalling out compact set $C\subset G(\A_{\calf})$, we set $C[l]$ to be the compact set with $K_v$ replaced with $K_v\cap G_{v}[l_v]$ for all $v\in S_{aux}$. For each $v\in S_{aux}$, we now choose $l_v$ large enough that if $a(v)$ represents the orbit in $Q_f-\{a\}$ associated to $v\in S_{aux}$, then $c_i(a(v))\notin c_i(C[l])\cap F$ for some $i$. 

Replacing $f_v$ and $f'_v$ by $\widetilde{f}_{v}$ and $\widetilde{f}'_{v}$ for each $v\in S_{aux}$, we conclude that $f=\bigotimes_v f_v$ and $f'=\bigotimes_v f'_v$ satisfy that
\[
\frac{1}{4}\sum_{(i,j)}I_{\omega_{i,j}}(f) = \Orb^\eta_\omega(f,a) = \Orb_{\omega'}(f',b) = \frac{1}{4L(1,\eta)^2}J_{\omega'}^{w_n,w_{n+1}}(f').
\]
Taking into account the support constraint $\supp({f})\subset G_\A[{w_n,w_{n+1}}]$, an argument as in the proof of Theorem \ref{Prop: main global tool} gives the formula in the proposition.

To prove \eqref{Final change to prop}, we note by inspection of our construction of $f=\otimes_vf_v$ and $f'=\otimes_vf_v'$ below, the support of the factors $f_{v_0}$ and $f_{v_0}'$ do not play a role in the vanishing statements. Note $\supp(\hat{f})\subset G_\A[{w_n,w_{n+1}}]$ since $\supp(BC(\varphi))\subset G_{v_0}[{w_n,w_{n+1}}]$.
\end{proof}

\begin{proof}[Finishing the proof of Theorem~\ref{Thm: base change Twisted Jacquet--Rallis}]
We keep the notations from Proposition \ref{Prop: good global test function}. To finish the proof, it suffices to show that
\[
I_{\omega}(\hat{f}) = \frac{1}{4L(1,\eta)^2}J^{w_n,w_{n+1}}_{\omega'}(\hat{f}').
\]
For this, consider the spectral expansions
\[
I_{\omega}(\hat{f})= \sum_\pi I_\pi(\hat{f})\text{ and }J_{\omega'}^{w_n,w_{n+1}}(\hat{f}')=\sum_\Pi J^{w_n,w_{n+1}}_{\Pi}(\hat{f}').
\] 
Note that if $\pi\cong \pi\cdot\eta_{i,j}$ for any non-trivial $(i,j)\in\{0,1\}^2$, then $BC(\pi)$ is not cuspidal. Then for any ${f'}^\circ$ matching $\hat{f}$,  $$J^{w_n,w_{n+1}}_{BC(\pi)}({f'}^\circ)=0.$$ By Theorem \ref{Prop: main global tool}, this forces $I_\pi(\hat{f})=0$. Thus, we may assume that $\pi\not\cong\pi\cdot\eta_{i,j}$ for any non-trivial $(i,j)\in\{0,1\}^2$. Recalling that $J_{\Pi}^{w_n,w_{n+1}}\equiv 0$ if  $\Pi\not\cong\Pi^\sig$ by \cite[Theorem 0.1]{FLO}, it suffices to show that 
\begin{equation}\label{eqn: global spectral transfer hat}
\frac{1}{L(1,\eta)^2}J_{\Pi}^{w_n,w_{n+1}}(\hat{f}')=\sum_{\pi\in \calb(\Pi)}I_\pi(\hat{f})
\end{equation}
for all cuspidal automorphic representations $\Pi$ such that $\Pi\cong\Pi^\sig$. Since $\supp(\hat{f})\subset G_\A[{w_n,w_{n+1}}]$, the argument in the proof of Theorem \ref{Prop: main global tool} implies that this reduces to showing that for any such $\pi$,
\begin{equation}\label{eqn: super important hat}
J^{w_n,w_{n+1},\natural}_{\Pi_{v_0}}(\varphi)\prod_{v\neq v_0}J^{w_n,w_{n+1},\natural}_{\Pi_v}(f'_{v}) =I^\natural_{\pi_{v_0}}(BC(\varphi))\prod_{v\neq v_0}I^\natural_{\pi_v}({f}_v).
\end{equation}
Note that if $\pi$ is not unramified at $v_0$, then both sides are zero. We thus assume that $\pi_{v_0}$ is unramified.

  Theorem \ref{Prop: main global tool} tells us that there exists a test function $f^\circ\in C_c^\infty(G(\A_{\calf}))$ matching $\hat{f}'$ such that for all such cuspidal representations 
\[
J^{w_n,w_{n+1},\natural}_{\Pi_{v_0}}(\varphi)\prod_{v\neq v_0}J^{w_n,w_{n+1},\natural}_{\Pi_v}(f'_{v}) =I^\natural_{\pi_{v_0}}(f^\circ_{v_0})\prod_{v\neq v_0}I^\natural_{\pi_v}(f^\circ_v).
\]
We may assume that $f^\circ_v ={f}_v$ for all $v\neq v_0$. Corollary \ref{Cor: unramified spectral transfer} now tells us that 
\[
J^{w_n,w_{n+1},\natural}_{\Pi_{v_0}}(\varphi)=I^\natural_{\pi_{v_0}}(f^\circ_{v_0}), 
\]
so that 
\[
J^{w_n,w_{n+1},\natural}_{\Pi_{v_0}}(\varphi)\prod_{v\neq v_0}J^{w_n,w_{n+1},\natural}_{\Pi_{v}}(f'_{v}) =J^{w_n,w_{n+1},\natural}_{\Pi_{v_0}}(\varphi)\prod_{v\neq v_0}I^\natural_{\pi_v}(f_v).
\]
 Lemma \ref{Lem: unramified spectral transfer} now states that
\[
J^{w_n,w_{n+1},\natural}_{\Pi_{v_0}}(\varphi)=I^\natural_{\pi_{v_0}}(BC(\varphi)),
\]
allowing us to conclude (\ref{eqn: super important hat}). Since this holds for all $\Pi$, we obtain \eqref{eqn: global spectral transfer hat} for all cuspidal representations $\Pi$. This implies
\[
\Orb_{\omega'}(\hat{f}',b) = \frac{1}{4L(1,\eta)^2}J_{\omega'}^{w_n,w_{n+1}}(\hat{f}')=I_{\omega}(\hat{f}) = \Orb^\eta_\omega(\hat{f},a)
\]
Factoring the orbital integrals and using the matching at all places $v\neq v_0$ implies the identity \eqref{eqn: reduce to integrals with central character}, and Theorem \ref{Thm: base change Twisted Jacquet--Rallis} follows. 
\end{proof}

\subsection{Proof of Theorem \ref{Thm: full fundamental lemma}}\label{Section: final touches}
We continue to assume that $E/F$ is an unramified extension of $p$-adic local fields. In Part \ref{Part 1}, we reduced Theorem  \ref{Thm: full fundamental lemma} to Theorem \ref{Thm: JR fundamental lemma for algebra} which states that for any $\varphi\in \calh_{K_{n,E}}(\GL_{n}(E))$, and for any $X\in \GL_{n}(F)^{rss}$, we have
\begin{equation*}
\omega(X)\Orb^{\GL_{n-1}(F),\eta}(BC(\varphi),X)=\begin{cases}\Orb^{U(V_{n-1})}(\varphi\ast \bfun_0,Y)&:X\leftrightarrow Y\in X^{rss}_{n}
,\\\qquad0&:\text{otherwise}.\end{cases}
\end{equation*}
We deduce this from Theorem \ref{Thm: base change Twisted Jacquet--Rallis}. We remark that while we proved Theorem \ref{Thm: base change Twisted Jacquet--Rallis} with respect to the split Hermitian form $w_n$, it is easy to see that the result holds with respect to the identify form $I_n$ since the two Hermitian spaces are isomorphic under the assumption that $E/F$ is unramified.

Considering the contraction map
\begin{align*}
C_c^\infty(\GL_{n-1}(E)\times \GL_n(E))&\lra C_c^\infty(\GL_{n}(E))\\ f&\longmapsto \tilde{f} 
\end{align*}
defined by \eqref{eqn: tilde contraction}, there exists a natural lift 
\[
\Phi\in \calh_{K_{n-1,E}\times K_{n,E}}(\GL_{n-1}(E)\times \GL_n(E))
\]such that $\tilde{\Phi}=\varphi$. Indeed, the function $\Phi=\bfun_{K_{n-1,E}}\otimes \varphi$ works.

Recall the commutative diagram
\begin{equation*}
\begin{tikzcd}
&\calh_{K_{n,E}}(\GL_{n}(E))\ar[dl,swap,"-\ast \bfun_0"]\ar[dr,"BC"]&\\
\calh_{K_{n,E}}(X_n)\ar[rr,"H"]&&\calh_{K_n}(\GL_{n}(F)),
\end{tikzcd}
\end{equation*}
where $\ast \bfun_0$ indicates convolution with the unit element and $H$ denotes the $\calh_{K_{n,E}}(\GL_n(E))$-module isomorphism of Hironaka. As we are multiplying both sides by the unit of the appropriate Hecke algebra, a simple computation and the commutativity of the above diagram imply that 
\[
\widetilde{BC(\Phi)} = BC(\tilde{\Phi})=H(\varphi\ast \bfun_0).
\]
 Now Theorem \ref{Thm: base change Twisted Jacquet--Rallis} implies that $\{\Phi,0\}$ and $BC(\Phi)$ are transfers of one another.
 
 To make this useful, we first lift $X\in \GL_n(F)^{rss}$ to a regular semi-simple $\ga=(\ga_1,\ga_2)\in [\GL_{n-1}(F)\times \GL_{n}(F)]^{rss}$ and lift $Y$ to $\de=(\de_1,\de_2)\in [\GL_{n-1}(E)\times\GL_n(E)]^{rss}$. The relations of orbital integrals and transfer factors in \eqref{Def: newish transfer factor}, (\ref{eqn: linear OI reduction}), and (\ref{eqn: twisted OI reduction}) thus imply
\begin{align*}
\Orb^{U(V_{n-1})}(\varphi\ast \bfun_0,Y) &=\Orb({\Phi},\de)\\
                                                         & = \Omega(\ga)\Orb^\eta(BC(\Phi),\ga)\qquad \text{(Theorem \ref{Thm: base change Twisted Jacquet--Rallis})}\\&= \omega(X)\Orb^{\GL_{n-1}(F),\eta}(BC(\varphi),X).
\end{align*}
Additionally, the vanishing component of Theorem \ref{Thm: base change Twisted Jacquet--Rallis} gives the correct vanishing of orbital integrals for $BC(\varphi)$, completing the proof of  Theorem \ref{Thm: JR fundamental lemma for algebra}.\qed


\bibliographystyle{alpha}

\bibliography{bibs_FL}
\end{document}